%% file: ms.tex
\documentclass[12pt]{article}

\usepackage{amsmath} 
\usepackage{amsthm} 
\usepackage{amsfonts} 
\usepackage{amssymb} 
\usepackage{stmaryrd} 
\usepackage{mathtools} 

\usepackage{bm} 
\usepackage{bbm} 

\usepackage{array} 
\usepackage[inline]{enumitem} 

\usepackage{tikz} 
\usetikzlibrary{cd} 
\usetikzlibrary{calc} 

\usepackage{subcaption} 

\usepackage{pdftexcmds} 
\makeatletter
\newcommand\compareStringsDo[5]{%
  \lowercase{\ifcase\pdf@strcmp{#1}{#2}}%
    #4\or
    #5\else
    #3\fi
}
\makeatother

\usepackage[%
  pdftex,
  plainpages=false,
  pdfpagelabels,
  colorlinks,
  citecolor=black,
  linkcolor=black,
  urlcolor=black,
  filecolor=black,
  bookmarksopen=false
]{hyperref} 
\usepackage[all]{hypcap} 

\newtheoremstyle{thmstyle}
  {\medskipamount}
  {\smallskipamount}
  {\slshape}
  {0pt}
  {\bfseries}
  {.}
  { }
  {\thmname{#1}\thmnumber{ #2}{\normalfont\thmnote{ (#3)}}}

\newtheoremstyle{plainstyle}
  {\medskipamount}
  {\smallskipamount}
  {\rmfamily}
  {0pt}
  {\bfseries}
  {.}
  { }
  {\thmname{#1}\thmnumber{ #2}{\normalfont\thmnote{ (#3)}}}

\theoremstyle{thmstyle}
\newtheorem{theorem}{Theorem}[section]
\newtheorem{lemma}[theorem]{Lemma}

\newtheorem{proposition}[theorem]{Proposition}

\theoremstyle{plainstyle}
\newtheorem{definition}[theorem]{Definition}
\newtheorem{remark}[theorem]{Remark}
\newtheorem{discussion}[theorem]{Discussion}

\newtheorem{example}[theorem]{Example}

\newenvironment{proofof}[1]{\begin{proof}[Proof of #1.]}{\end{proof}}

\setlist[enumerate]{label={\roman*.}, ref={(\roman*)}}

\newcommand{\rn}{\bm}
\newcommand{\df}{\stackrel{\text{def}}{=}}
\newcommand{\place}{\mathord{-}}
\def\symdiff{\mathbin{\triangle}}
\newcommand{\comp}{\mathbin{\circ}}
\newcommand{\rest}{\mathord{\vert}}
\newcommand{\floor}[1]{\ensuremath{\left\lfloor#1\right\rfloor}}

\newcommand{\cat}{\textsc}
\newcommand{\leqLL}{\leq_{\operatorname{LL}}}
\newcommand{\lessLL}{<_{\operatorname{LL}}}

\newcommand{\function}[2]{\colon #1 \rightarrow #2}
\newcommand{\injection}[2]{\colon #1 \rightarrowtail #2}
\newcommand{\interpret}[2]{\colon #1 \leadsto #2}
\newcommand{\tot}{\leftrightarrow}

\DeclareMathOperator{\Hom}{Hom}
\DeclareMathOperator{\Aut}{Aut}
\DeclareMathOperator{\im}{im}

\newcommand{\HomT}[1]{\Hom^+(\cA[#1],\RR)}

\def\Tind{T_{\operatorname{ind}}}
\def\tind{t_{\operatorname{ind}}}
\def\Dopen{D_{\operatorname{open}}}

\newcommand{\NN}{\mathbb{N}}

\newcommand{\RR}{\mathbb{R}}

\newcommand{\One}{\mathbbm{1}}

\newcommand{\cA}{\mathcal{A}}

\newcommand{\cC}{\mathcal{C}}
\newcommand{\cD}{\mathcal{D}}
\newcommand{\cE}{\mathcal{E}}
\newcommand{\cF}{\mathcal{F}}
\newcommand{\cG}{\mathcal{G}}
\newcommand{\cH}{\mathcal{H}}
\newcommand{\cK}{\mathcal{K}}
\newcommand{\cL}{\mathcal{L}}
\newcommand{\cM}{\mathcal{M}}
\newcommand{\cN}{\mathcal{N}}
\newcommand{\cP}{\mathcal{P}}

\def\prop{\texttt}

\def\UInduce{\prop{UInduce}}

\def\AEHP{\prop{AEHP}}
\def\EHP{\prop{EHP}}

\newcommand{\TGraph}{T_{\operatorname{Graph}}}
\newcommand{\TLinOrder}{T_{\operatorname{LinOrder}}}
\newcommand{\TPerm}{T_{\operatorname{Perm}}}
\DeclareMathOperator{\Th}{Th}
\DeclareMathOperator{\Forb}{Forb}

\def\Erdos{Erd\H{o}s}
\def\Los{\L o\'{s}}
\def\Dolezal{Dole\v{z}al}
\def\Grebik{Gre\'{i}k}
\def\Hladky{Hladk\'{y}}
\def\Rozhon{Rozho\v{n}}
\def\Caratheodory{Carath\'{e}odory}
\def\Lovasz{Lov\'{a}sz}
\def\Szemeredi{Szemer\'{e}di}

\title{Countable Ramsey}
\author{Leonardo N.~Coregliano\thanks{Institute for Advanced Study, {\tt lenacore@ias.edu}. This
    material is based upon work supported by a grant from the Institute for Advanced Study School of
    Mathematics.} \and%
  Maryanthe Malliaris\thanks{University of Chicago, {\tt mem@math.uchicago.edu}. Research partially
    supported by NSF-BSF 2051825.}%
}
\date{\today}

\begin{document}
\maketitle

\begin{abstract}
  \footnotesize
  The celebrated \Erdos--Hajnal Conjecture says that in any proper hereditary class of finite graphs
  we are guaranteed to have a clique or anti-clique of size $n^c$, which is a much better bound than
  the logarithmic size that is provided by Ramsey's Theorem in general. On the other hand, in
  uncountable cardinalities, the model-theoretic property of stability guarantees a uniform set much
  larger than the bound provided by the \Erdos--Rado Theorem in general.

  Even though the consequences of stability in the finite have been much studied in the literature,
  the countable setting seems a priori quite different, namely, in the countably infinite the notion
  of largeness based on cardinality alone does not reveal any structure as Ramsey's Theorem already
  provides a countably infinite uniform set in general. In this paper, we show that the natural
  notion of largeness given by upper density reveals that these phenomena meet in the countable: a
  countable graph has an almost clique or anti-clique of positive upper density if and only if it
  has a positive upper density almost stable set. Moreover, this result also extends naturally to
  countable models of a universal theory in a finite relational language.

  Our methods explore a connection with the notion of convergence in the theory of limits of dense
  combinatorial objects, introducing and studying a natural approximate version of the
  \Erdos--Hajnal property that allows for a negligible error in the edges (in general, predicates)
  but requires linear-sized uniform sets in convergent sequences of models (this is much stronger
  than what stable regularity can provide as the error is required to go to zero). Finally,
  surprisingly, we completely characterize all hereditary classes of finite graphs that have this
  approximate \Erdos--Hajnal property. The proof highlights both differences and similarities with
  the original conjecture.
\end{abstract}

\section{Introduction}

The celebrated Ramsey's Theorem~\cite{Ram29} guarantees that sufficiently large structures have
uniform substructures. Without any extra restrictions, the size of the guaranteed uniform
substructure is typically tiny in comparison to the ambient structure: for example, to guarantee
a clique or independent set of size $n$ in a graph, its size must be exponential in $n$. The famous
\Erdos--Hajnal Conjecture~\cite{EH89} then asks if this bound can be improved to polynomial in $n$
if we restrict the problem to any hereditary (i.e., closed under induced subgraphs) proper subclass
of finite graphs. Several hereditary proper subclasses of finite graphs are known to satisfy the
\Erdos--Hajnal Conjecture (see~\cite{Chu14} for a survey).

In the uncountable, Ramsey type theorems also detect important differences between structures. This
is best stated in the language of logic, specifically set theory and model theory: in general, the
\Erdos--Rado Theorem~\cite{ER56} gives a tower bound on the size of a uniform subset, however, in
structures which are stable in the sense of model theory, see below, we can extract uniform subsets
of essentially the same size as the model.

However, in the case of countable structures balanced between these two, all structures seem to
behave in the same way since the infinite version of Ramsey's Theorem yields a uniform set of the
maximum possible cardinality of $\aleph_0$, so it is natural to ask if there is any version of
uniformity that would be able to distinguish between countable structures.

We propose an answer through probability, more specifically through the language of graph limits and
continuous combinatorics (see~\cite{Lov12} for the graph case and~\cite{Aus08,AC14,CR20a} for the
case of universal theories on finite relational languages). When considering universal theories of
graphs, we work with graphons, which are continuum-sized limits of convergent sequences of finite
graphs, in which relative sizes of sets of vertices are encoded by a probability measure. We start
by proving a dichotomy theorem for graphons (Theorem~\ref{thm:stablegraphon}): a graphon contains a
positive measure almost clique or a positive measure almost independent set if and only if it has an
almost stable positive measure subgraphon (see Definitions~\ref{def:subgraphon}
and~\ref{def:almoststablegraphon}). We also give examples to illustrate how the negative side works,
preventing any positive measure uniform sets for basic instances of instability (quasirandom graphs
and ``recursive'' half-graphs). We also generalize this theorem for arbitrary universal theories $T$
in finite relational languages (Theorem~\ref{thm:stabletheon}): a $T$-on (i.e., a limit of a
convergent sequence of finite models of $T$) has a positive measure ``uniform'' sub-object if and
only if it has a positive measure sub-object in which all predicate symbols are stable (see
Definition~\ref{def:almoststabletheon}).

Since graphons and theons are limits of convergent sequences, the dichotomy can be pushed down to
convergent sequences (Theorems~\ref{thm:stableconvseqgraph}) as: a convergent sequence of finite
graphs $(H_n)_{n\in\NN}$ has non-negligible (i.e., linear-sized) sets $U_n\subseteq V(H_n)$ such
that the edge density in $(H_n\rest_{U_n})_{n\in\NN}$ either converges to $1$ or to $0$, (i.e., it
is an almost clique or an almost independent set) if and only if there are non-negligible sets
$U'_n\subseteq V(H_n)$ that make the edge relation almost stable in the induced sequence
$(H_n\rest_{U'_n})_{n\in\NN}$. A similar theorem is also obtained for arbitrary universal theories
(Theorem~\ref{thm:stableconvsequniversal}).

For a countable graph (or model) $G$, this dichotomy (Theorems~\ref{thm:stablecountablegraph}
and~\ref{thm:stablecountablemodel}) takes the following form: there exists a set $U\subseteq\NN_+$
and an increasing sequence $(n_\ell)_{\ell\in\NN}$ of positive integers along which $(G\rest_{U\cap
  [n_\ell]})_{\ell\in\NN}$ is almost ``uniform'' and $U\cap [n_\ell]$ is non-negligible in
$[n_\ell]$ if and only if there exists a set $U'\subseteq\NN$ and an increasing sequence
$(n'_\ell)_{\ell\in\NN}$ of positive integers along which $(G\rest_{U'\cap [n'_\ell]})_{\ell\in\NN}$
is almost stable and $U'\cap [n'_\ell]$ is non-negligible in $[n'_\ell]$.

From a general point of view, the theorems above on convergent sequences may be understood as
characterizing existence of a linear-sized subset which is an almost clique or almost empty
graph. This characterization is in terms of almost-absence of a certain finite structure (but notice
that we remain agnostic about some of the edges; this is in some ways very satisfying as it
coincides with the major structural property mentioned above, stability, in model theory). Focusing
on the case of graphs, by analogy to the usual \Erdos--Hajnal Conjecture, this begs the question:
does there exist a family of finite graphs (with no agnostic edges) whose absence precisely
characterizes the existence of a linear-sized almost clique or almost anti-clique in a convergent
sequence?

In the last, substantial section of the paper, Section~\ref{sec:AEHPgraphs}, we prove that the
answer is yes and provide a complete characterization (Theorem~\ref{thm:AEHPgraph}) of this family
as all induced subgraphs of some recursive blow-up of the $4$-cycle (see
Definition~\ref{def:C4omega} and Figure~\ref{fig:C4ell}). This doubly unexpected characterization
(its existence and the nature of the forbidden family were both a surprise) sheds a quite different
light on the proofs above.

Using the characterization above, we can show (Theorem~\ref{thm:AEHP->EHP}) that this approximate
\Erdos--Hajnal property ($\AEHP$) implies the usual \Erdos--Hajnal property ($\EHP$). Namely, if a
hereditary class of graphs $\cG$ is such that every convergent sequence of graphs in $\cG$ has a
linear-sized almost clique or almost anti-clique, then there exists a constant $c_\cG > 0$ such that
every graph of $\cG$ on $n$ vertices has a clique or anti-clique of size $n^{c_\cG}$. Our proof is
somewhat indirect and crucially relies on the characterization of $\AEHP$ and in the concluding
section, we ask if a more direct construction is possible.

These points are further explored in the text once the details of the proofs are available for
discussion and commentary.

We conclude the introduction with some comments on the model theoretic notion of stability, which has
been mentioned several times above. We emphasize that \emph{it is not necessary to be familiar with
  stability either to read or to appreciate the present paper}, but let us explain this remark. It
was very interesting to us to discover in the course of writing this paper that stability appears in
a characteristic sense in some of our main results, despite investigating a priori unrelated
questions. Stability is not only a key concept in Shelah's classification theory~\cite{She90} but
has had recent applications in finite combinatorics via the stable regularity lemma and stable
Ramsey's Theorem~\cite{MS14} (see also~\cite{AFP18,MS21}) and learning theory via Littlestone
dimension~\cite{ALMM19,BLM20}. Stability has several equivalent definitions in the language of model
theory. The one we will use is combinatorial and says essentially that no relation forms a
half-graph with respect to any partition of the variables (see
Definition~\ref{def:almoststabletheon} below), which may be thought of as ensuring a certain kind of
symmetry for relations. For model theorists, we note that our framework only requires working with
stability of specific formulas, not of all formulas, and we will note in the text where ideas from
stable regularity will play a role in guiding certain proofs. This said, because of the new context,
the present proofs require working by hand with the combinatorial definitions and building up
everything necessary from scratch. To the extent that stability appears, it is as a
characterization: its appearance is fully justified, so to speak, by the proofs in either direction.
Moreover, the interaction of convergence and stability in the present context seems to point to
something interesting and new about stability's effect, which was certainly not explored in the
usual infinitary context. As a result, not only is the paper self-contained in this respect, but
also readers encountering stability for the first time in the present context may bring new
understanding by taking these results as a starting point.

\medskip

The paper is organized as follows. In Section~\ref{sec:notation} we establish some basic notation
that will be used throughout the text. The remainder of the paper can be thematically divided into
three main parts.

\begin{itemize}
\item In the first part, Sections~\ref{sec:graphons} and~\ref{sec:consgraphs}, we study the
  interaction of large almost cliques or almost anti-cliques with large almost stable sets. More
  specifically, in Section~\ref{sec:graphons}, we prove the stability dichotomy theorem for
  graphons, Theorem~\ref{thm:stablegraphon}. In Section~\ref{sec:consgraphs}, we prove the finite
  and countable versions of the stability dichotomy theorem for graphs,
  Theorems~\ref{thm:stableconvseqgraph} and~\ref{thm:stablecountablegraph},
  respectively. Sections~\ref{sec:graphons} and~\ref{sec:consgraphs} also contain a gentle
  introduction to the concepts of the theory of graphons needed for the results.

\item In the second part, Sections~\ref{sec:universaltheories} and~\ref{sec:consmodels}, we extend
  the stability results of the first part to universal theories in finite relational languages: in
  Section~\ref{sec:universaltheories} we prove the stability dichotomy theorem for arbitrary
  universal theories, Theorem~\ref{thm:stabletheon} and in Section~\ref{sec:consmodels} we prove its
  finite and countable versions, Theorems~\ref{thm:stableconvsequniversal}
  and~\ref{thm:stablecountablemodel}, respectively. These sections also contain a gentle
  introduction to the concepts of the theory of theons needed for the results.

\item In the third part, we revisit the characterization of the first part, in the pure language of
  hereditary classes. More specifically, in Section~\ref{sec:AEHP} we define and prove basic
  properties of the approximate \Erdos--Hajnal property ($\AEHP$), an analogue of the usual
  \Erdos--Hajnal property requiring linear-sized almost uniform sets in convergent sequences of
  models (rather than polynomial-sized uniform sets in a single model). In
  Section~\ref{sec:AEHPgraphs} we characterize universal theories of graphs with $\AEHP$ as
  precisely the theories that forbid some induced subgraph of some recursive blow-up of the
  $4$-cycle (Theorem~\ref{thm:AEHPgraph}) and show that the approximate \Erdos--Hajnal property
  implies the usual \Erdos--Hajnal property for graphs (Theorem~\ref{thm:AEHP->EHP}).
\end{itemize}

Although the parts can probably be read in any order (provided the readers are willing to assume the
referenced results from previous sections), reading in the given order may provide the most insight
into the interconnected aspects of the emerging theory.

We conclude in Section~\ref{sec:concl} with remarks and open questions.

\section{Some notation}
\label{sec:notation}

Throughout the text, we will use the notation $\NN\df\{0,1,\ldots\}$ for the non-negative integers
and $\NN_+\df\NN\setminus\{0\}$ for the positive integers. We also let $[n]\df\{1,\ldots,n\}$ and
$(n)_m\df n(n-1)\cdots(n-m+1)$. The usage of the arrow $\rightarrowtail$ for a function will always
presume the function to be injective.  We let $2^V\df\{A\subseteq V\}$ be the set of all the subsets
of $V$, let $\binom{V}{\ell}\df\{A\subseteq V\mid\lvert A\rvert=\ell\}$. We also let $r(V)$ be the
set of all \emph{finite non-empty} subsets of $V$ and $r(V,\ell)\df\{A\in r(V)\mid\lvert A\rvert\leq
\ell\}$ be the set of all non-empty subsets of $V$ of size at most $\ell$. Given a function
$\alpha\function{V}{W}$, we may use the notation $\alpha_v$ for $\alpha(v)$ when convenient.

Given an injection $\alpha\injection{U}{V}$ and a set $X$, we let $\alpha^*\function{X^V}{X^U}$ be
the contra-variantly defined ``projection'' given by $\alpha^*(x)_u = x_{\alpha(u)}$. We will be
using these ``projections'' both in the situation where we are interested in the coordinates of some
point $x\in X^V$ that are indexed by elements of $\im(\alpha)$ and since $\alpha\injection{U}{V}$
induces an injection $\alpha\injection{r(U)}{r(V)}$ (denoted by abuse with the same letter), this in
turn gives the projection $\alpha^*\function{X^{r(V)}}{X^{r(U)}}$ that allows us to inspect
coordinates of $x\in X^{r(V)}$ that are indexed by non-empty subsets of $\im(\alpha)\subseteq V$.

We will be frequently abusing notation by identifying $[n]$ with $n$, e.g., we will use $r(n,\ell)$
as a shorthand for $r([n],\ell)$. Random variables will always be typed in
$\rn{math\ bold\ face}$. We denote by $S_V$ the group of bijections $V\rightarrowtail V$ so that
$S_n$ is the group of permutations on $n$ elements.

We denote the complete graph on $n$ vertices by $K_n$ and the empty graph on $n$ vertices by
$\overline{K}_n$. We use the terms ``anti-clique'', ``empty graph'' and ``independent set''
interchangeably.

\section{Almost cliques or anti-cliques in graphons}
\label{sec:graphons}

In this section, we state and prove the stability dichotomy theorem for graphons,
Theorem~\ref{thm:stablegraphon}. Along the way, we give a gentle introduction to the concepts of
graphon theory that we will be using (we refer the reader to~\cite{Lov12} for a more thorough
introduction to the theory). Let us also remark that some of the techniques of limit theory can be
traced back to way before the development of graphons at least as far as~\cite{DF81}.

Given finite graphs $G$ and $H$, let $\Tind(G,H)$ be the set of all graph embeddings of $G$ in $H$
(i.e., injective functions $f\injection{V(G)}{V(H)}$ that preserve edges and non-edges) and let
\begin{align*}
  \tind(G,H)
  & \df
  \frac{\lvert\Tind(G,H)\rvert}{(\lvert H\rvert)_{\lvert G\rvert}}
\end{align*}
be the normalized number of embeddings of $G$ and $H$; this is sometimes called the \emph{labeled
  (induced) density} of $G$ in $H$. The ``labeled'' here is to differentiate from the
\emph{(induced) density} of $G$ in $H$, which is the normalized number of induced subgraphs of $H$
that are isomorphic to $G$ given by
\begin{align}
  p(G,H)
  & \df
  \frac{\lvert\{U\subseteq V(H) \mid G\rest_U\cong H\}\rvert}{\binom{\lvert H\rvert}{\lvert G\rvert}}
  =
  \frac{\lvert G\rvert!}{\lvert\Aut(G)\rvert}\cdot\tind(G,H),
  \label{eq:ptind}
\end{align}
where $\Aut(G)$ is the group of automorphisms of $G$. We denote by $\rho$ the edge graph, so that
$p(\rho,H)=\tind(\rho,H)$ denotes the edge density of $H$.

A sequence $(H_n)_{n\in\NN}$ of finite graphs is called \emph{convergent} if it is \emph{increasing}
in the sense that for every $n\in\NN$, $\lvert H_n\rvert < \lvert H_{n+1}\rvert$ and if for every
finite graph $G$, the limit $\lim_{n\to\infty} p(G,H_n)$ exists. An alternative way of seeing
convergence is that each graph $H$ defines a point $p(\place,H)\in [0,1]^{\cM}$, where $\cM$ is the
set of all finite graphs up to isomorphism and convergence of an increasing sequence
$(H_n)_{n\in\NN}$ is simply convergence of the corresponding sequence $(p(\place,H_n))_{n\in\NN}$
with respect to the product topology of $[0,1]^{\cM}$. Since $\cM$ is countable, $[0,1]^{\cM}$ is
metrizable and since it is compact, it follows that any increasing sequence of finite graphs has a
convergent subsequence.

The main point of graphon theory is that convergent sequences can be encoded by a \emph{geometric}
limit object in which limits of labeled densities can be naturally computed. Formally, given an
atomless standard probability space $\Omega=(X,\cA,\mu)$, a \emph{graphon} over $\Omega$ is a
function $W\function{X\times X}{[0,1]}$ that is symmetric and is measurable with respect to the
completion of the product $\sigma$-algebra $\cA\otimes\cA$ with respect to the product measure
$\mu\otimes\mu$. Typically, we take $\Omega$ as $[0,1]$ equipped with the Lebesgue measure $\lambda$
over Borel sets, in which case we say ``graphon over $[0,1]$'' or simply ``graphon'' (which is
simply a symmetric Lebesgue measurable function $[0,1]^2\to[0,1]$). The intuition is that a graphon
over $\Omega=(X,\cA,\mu)$ is a graph with vertex set $X$ in which edges can have fractional values
and $W(x,y)$ should be interpreted as the ``probability'' that $x$ and $y$ are adjacent. With this
intuition in mind, the \emph{labeled (induced) density} of a graph $G$ in a graphon $W$ is naturally
defined as
\begin{align}\label{eq:tindGW}
  \tind(G,W)
  & \df
  \int_{X^{V(G)}}
  \prod_{\{v,w\}\in E(G)} W(x_v,x_w)
  \prod_{\{v,w\}\in E(\overline{G})} (1 - W(x_v,x_w))
  \ d\mu(x),
\end{align}
where $E(G)\df\{\{v,w\}\in\binom{V(G)}{2} \mid G\vDash E(v,w)\}$ is the edge set of $G$ and
$\overline{G}$ is the complement graph of $G$. We also define $p(G,W)\df\lvert
G\rvert!\cdot\tind(G,W)/\lvert\Aut(G)\rvert$ in analogy with~\eqref{eq:ptind} and we let $\phi_W\df
p(\place,W)$. We say that $W$ is a limit of a convergent sequence $(H_n)_{n\in\NN}$ if
$\lim_{n\to\infty} p(G,H_n) = \phi_W(G)$ for every finite graph $G$.

The following theorem, sometimes referred to as Existence Theorem for graphons, is the main theorem of graphon
theory.
\begin{theorem}[Lov\'{a}sz--Szegedy~\protect{\cite{LS06}}]
  Let $\Omega$ be an atomless standard probability space. If $(H_n)_{n\in\NN}$ is a convergent
  sequence of graphs, then there exists a graphon $W$ over $\Omega$ that is a limit of
  $(H_n)_{n\in\NN}$. Conversely, every graphon is a limit of a convergent sequence of graphs.
\end{theorem}
It is trivial that two convergent sequences can converge to the same limit graphon as only the tail
behavior of the convergent sequences matters and changes to $o(\lvert H_n\rvert^2)$ edges do not
affect densities. On the other side, more than one graphon can represent the limit of the same
convergent sequence. For example, any graphon $W\function{[0,1]^2}{[0,1]}$ over $[0,1]$ represents
the same limit as $W'$ given by $W'(x,y) = W(2x\bmod 1, 2y\bmod 1)$ (see
Figure~\ref{fig:halfgraphon} for an example). The Uniqueness Theorem for graphons~\cite{BCL10} (see
also~\cite[Theorem~13.10]{Lov12}) characterizes when two graphons represent the same limit using
measure-preserving functions.

\begin{figure}[htb]
  \input{halfgraphon}
\end{figure}

\medskip

For our theorems, we will be considering subgraphons, which are a limit world generalization of the notion of
induced subgraph, but only by non-negligible sets (i.e., sets of linear size).

\begin{definition}\label{def:subgraphon}
  Given a graphon $W$ over $\Omega$, a \emph{(positive measure) subgraphon} $W'$ of $W$ is a graphon
  over a space $\Omega'$ such that there exist a sequence $(H_n)_{n\in\NN}$ converging to $W$ and
  sets $U_n\subseteq V(H_n)$ such that $\lim_{n\to\infty} \lvert U_n\rvert/\lvert H_n\rvert > 0$ and
  $(H_n\rest_{U_n})_{n\in\NN}$ converges to $W'$. When we want to be more specific, we say that $W'$
  is a subgraphon of $W$ of measure $c\in(0,1]$, if the condition above holds with
    $\lim_{n\to\infty} \lvert U_n\rvert/\lvert H_n\rvert = c$.
\end{definition}

Na\"{i}vely, one might think that each subgraphon of a graphon $W$ could be represented as
$W\rest_{U\times U}$ for some positive measure set $U\subseteq X$ over the space
$\Omega'\df(U,\cA\rest_U,\mu\rest_U)$, where $\cA_U\df\{A\cap U\mid A\in\cA\}$ and
$\mu_U(A)\df\mu(A)/\mu(U)$. There are two problems with this na\"{i}ve definition. The first is only
technical: $\Omega'$ is not necessarily a standard probability space, but this can be addressed by
conditioning the measure rather than restricting the space by using the space $\Omega_U =
(X,\cA,\mu_U)$, where $\mu_U(A) = \mu(A\cap U)/\mu(U)$. The second is more serious: not every limit
of a sequence of the aforementioned form $(H_n\rest_{U_n})_{n\in\NN}$ is necessarily encoded this
way. However, the next lemma says that this description is not too far from correct, we only need to
``rescale'' the underlying measure by a weight function.

\begin{lemma}\label{lem:subgraphon}
  Let $W$ be a graphon over $\Omega=(X,\cA,\mu)$, let $W'$ be another graphon and let $c > 0$. The
  following are equivalent.
  \begin{enumerate}
  \item There exist a convergent sequence of graphs $(H_n)_{n\in\NN}$ converging to $W$ and sets
    $U_n\subseteq V(H_n)$ with $\lim_{n\to\infty} \lvert U_n\rvert/\lvert H_n\rvert = c$ such that
    $(H_n\rest_{U_n})_{n\in\NN}$ converges to $W'$, that is, $W'$ is a subgraphon of $W$ of measure
    $c$.
  \item There exists a measurable function $f\function{X}{[0,1]}$ with $\int_X f\ d\mu = c$ such
    that $\phi_{W'} = \phi_{W_f}$, where $W_f$ is the graphon over the the space
    $\Omega_f\df(X,\cA,\mu_f)$ defined by
    \begin{align*}
      \mu_f(A) & \df \frac{\int_A f(x) \ d\mu(x)}{c},\\
      W_f(x,y) & \df W(x,y).
    \end{align*}
  \end{enumerate}
\end{lemma}

We defer the proof of this lemma as it is a particular case of the more general Lemma~\ref{lem:limitsubobject}.

\medskip

As we mentioned before, one way of interpreting a graphon $W$ is as a measurable ``graph'' over
$\Omega$, except that $W(x,y)$ is the ``probability'' that $x$ and $y$ are adjacent. Under this
interpretation, $\{0,1\}$-valued graphons are simply measurable graphs and we can reinterpret the
labeled density formula~\eqref{eq:tindGW} as follows. The set of \emph{labeled (induced) copies} of
a finite graph $G$ in a graphon $W$ over $\Omega=(X,\cA,\mu)$ is the set
\begin{align*}
  \Tind(G,W)
  & \df
  \begin{multlined}[t]
    \biggl\{(x,y)\in X^{V(G)}\times [0,1)^{\binom{V(G)}{2}}
    \;\bigg\vert\;
    \\
    \forall\{v,w\}\in \binom{V(G)}{2}, (y_{\{v,w\}} < W(x_v,x_w) \tot \{v,w\}\in E(G))
    \biggr\}.
  \end{multlined}
\end{align*}
Under this definition, we have $\tind(G,W) =
(\mu^{V(G)}\otimes\lambda^{\binom{V(G)}{2}})(\Tind(G,W))$.

Note also that if $W$ is a $\{0,1\}$-valued graphon and we interpret it as simply a measurable graph
on $[0,1]$, whenever all coordinates of $x$ are distinct, we have $(x,y)\in\Tind(G,W)$ if and only
if $x$ is an embedding of $G$ in $W$.

In the same way that the usual (dense setting) Graph Removal Lemma~\cite{RS78,EFR86} (see
also~\cite[Lemma~11.64 and Theorems~15.24 and~15.25]{Lov12}) says that we can change a negligible
fraction of edges to remove graphs that have negligible density, the following graphon version says
that $\tind(G,W) = 0$ can be turned into $\Tind(G,W)$ morally empty by changing $W$ only in a
zero-measure set (see also Theorem~\ref{thm:theonremoval} for the general case). In fact,
Elek--Szegedy showed~\cite[Theorem~1]{ES12} that the finite version of the Removal Lemma follows
from a connection of limit theory via ultraproducts that we will see later.

\begin{theorem}[Graphon Removal Lemma~\protect{\cite[Theorem~1]{Pet13}}]\label{thm:graphonremoval}
  If $W$ is a graphon over $\Omega=(X,\cA,\mu)$, then there exists a graphon $W'$ over $\Omega$ such
  that $W = W'$ a.e.\ and for every finite graph $G$ such that $\tind(G,W) = 0$, we have
  $\Tind(G,W')\subseteq\cD_{V(G)}$, where
  \begin{align}\label{eq:diagonalgraphon}
    \cD_{V(G)}
    & \df
    \left\{(x,y)\in X^{V(G)}\times [0,1)^{\binom{V(G)}{2}}
    \;\middle\vert\;
    \exists v,w\in V(G), (v\neq w\land x_v = x_w)
    \right\}
  \end{align}
  denotes the diagonal set with respect to the $x$ variables.

  Furthermore, if $W$ is $\{0,1\}$-valued, then $W'$ can also be taken to be $\{0,1\}$-valued.
\end{theorem}

The final concept needed to state our graphon dichotomy theorem is that of an almost stable graphon
defined below.

\begin{definition}\label{def:almoststablegraphon}
  Recall that a \emph{half-graph of order $n$} in a graph $G$ (see Figure~\ref{fig:halfgraph}) is
  pair of sequences $(x_1,\ldots,x_n)$ and $(y_1,\ldots,y_n)$ of vertices of $G$ such that
  $\{x_i,y_j\}\in E(G)$ if and only if $i\leq j$.

  We say that a \emph{tree of height $n$} in a graph $G$ (see Figure~\ref{fig:tree}) is pair of
  sequences $(x_\sigma \mid \sigma\in\{0,1\}^n)$ and $(y_\tau \mid m\in\{0,1,\ldots,n-1\},
  \tau\in\{0,1\}^m)$ such that for every $\sigma = (\sigma_i)_{i=1}^n\in\{0,1\}^n$ and every $m <
  n$, $\{x_\sigma,y_{\sigma\rest_{[m]}}\}\in E(G)$ if and only if $\sigma_{m+1} = 1$ (trees of
  height $n$ are also known under several different names in the literature).

  Recall also that a graph is called \emph{$n$-stable} (or more formally, its edge relation is
  $n$-stable) if it does \emph{not} contain any half-graphs of order $n$.

  A graphon $W$ is \emph{almost stable} if there exists $n\in\NN$ such that every finite graph $G$
  containing a half-graph of order $n$ satisfies $p(G,W) = 0$.
\end{definition}

Our use of stability will be mainly to provide a bound on the height of trees, since
by~\cite[Lemma~6.7.9]{Hod93}, an $n$-stable graph does not contain any trees of height $2^{n+2}-2$
and, conversely, if a graph does not contain a tree of height $n$, then it is $(2^{n+2}-2)$-stable.

\begin{figure}[htb]
  \input{halfgraph}
\end{figure}

\begin{figure}[htb]
  \input{tree}
\end{figure}

\begin{theorem}\label{thm:stablegraphon}
  The following are equivalent for a graphon $W$ over $\Omega=(X,\cA,\mu)$.
  \begin{enumerate}
  \item $W$ contains a subgraphon that is either constant equal to $1$ or constant equal to $0$.%
    \label{thm:stablegraphon:uniformsubgraphon}
  \item There exists a positive measure set $U\subseteq X$ such that either $W\rest_{U\times U} = 1$
    a.e.\ or $W\rest_{U\times U} = 0$ a.e.%
    \label{thm:stablegraphon:uniformsquare}
  \item $W$ contains an almost stable subgraphon.%
    \label{thm:stablegraphon:stablesubgraphon}
  \end{enumerate}
\end{theorem}

\begin{discussion}\label{dsc:subgraphon}
  As we noted before, not every subgraphon of $W$ is of the form $W\rest_{U\times U}$ and thus the
  equivalence of items~\ref{thm:stablegraphon:uniformsubgraphon}
  and~\ref{thm:stablegraphon:uniformsquare} is not trivial.

  However, if $P$ is property of graphons that is closed under taking subgraphons, then a graphon
  $W$ has a subgraphon satisfying $P$ if and only if it has a positive measure $U\subseteq X$ such
  that $W\rest_{U\times U}$ satisfies $P$. The backward implication is obvious, and the forward
  implication can be seen easily from Lemma~\ref{lem:subgraphon}: if the subgraphon $W'$ satisfying
  $P$ corresponds to a measurable function $f$ of positive integral, then for some $\epsilon > 0$,
  the set $U_\epsilon\df\{x\in X \mid f(x) > \epsilon\}$ has positive measure and
  $W\rest_{U_\epsilon\times U_\epsilon}$ satisfies $P$ as it is a subgraphon of $W'$.

  As we will see in Theorem~\ref{thm:stableconvseqgraph}, the importance of
  item~\ref{thm:stablegraphon:uniformsquare} is that if $(H_n)_{n\in\NN}$ converges to $W$, then
  subgraphons of form $W\rest_{U\times U}$ can be pulled back to $(H_n\rest_{U_n})_{n\in\NN}$
  without changing the sequence $(H_n)_{n\in\NN}$.
\end{discussion}

The main ingredient to prove this theorem is the following lemma whose main idea can be seen as a
graphon analogue of the construction of $\epsilon$-good sets in~\cite{MS14,MS21}, but with $\epsilon
= 0$ and is much easier for measure theoretic reasons.

\begin{lemma}\label{lem:stablegraphon}
  Let $W$ be an almost stable graphon over $\Omega=(X,\cA,\mu)$. Then there exists a positive
  measure set $U\subseteq X$ such that either $W\rest_{U\times U} = 1$ a.e.\ or $W\rest_{U\times U}
  = 0$ a.e.
\end{lemma}

\begin{proof}
  By~\cite[Theorem~4.1]{LS10}, we know that $W$ is $\{0,1\}$-valued almost everywhere, so we can
  change it in a zero-measure set so that it is $\{0,1\}$-valued\footnote{As we will see in
    Theorem~\ref{thm:stabletheon} for general universal theories, this step is not actually
    necessary, but it simplifies the proof in the graphon case.}.
  
  By Theorem~\ref{thm:graphonremoval}, we can further replace $W$ with a $\{0,1\}$-valued graphon
  $W'$ such that there exists $n\in\NN$ such that every finite graph $G$ containing a half-graph of
  order $n$ satisfies $\Tind(G,W')\subseteq\cD_{V(G)}$, that is, $W'$ as a measurable graph over
  $\Omega$ is $n$-stable (except for potential half-graphs that collide vertices).

  For $x\in X$, let $N_{W'}(x)\df\{y\in X \mid W'(x,y) = 1\}$ denote the ``neighborhood'' of $x$ in
  $W'$ and let $X'$ be the set of $x\in X$ such that $N_{W'}(x)$ is measurable with respect to the
  completion $(\cA',\mu')$ of $(\cA,\mu)$. Fubini's Theorem gives $\mu'(X')=1$.

  We now construct sequences $(X_\sigma)_\sigma$ and $(y_\sigma)_\sigma$ indexed by finite strings
  over $\{0,1\}$ inductively in the length of $\sigma$ as follows.
  \begin{enumerate}[label={\arabic*.}]
  \item Set $X_\varnothing\df X'$.
  \item Given $X_\sigma$, if there exists $z\in X'$ such that $0 < \mu'(N_{W'}(z)\cap X_\sigma) <
    \mu'(X_\sigma)$, then set $y_\sigma\df z$, $X_{\sigma 1}\df X_\sigma\cap N_{W'}(z)$ and $X_{\sigma 0}\df
    X_\sigma\setminus N_{W'}(z)$; otherwise stop the construction.
  \end{enumerate}

  Let also $Y$ be the (countable) set of all $y_\sigma$ that get defined in the construction
  above. By induction, it follows that if $X_\sigma$ is defined for every $\sigma\in\{0,1\}^t$ of a
  fixed length $t$, then $\{X_\sigma \mid \sigma\in\{0,1\}^t\}$ forms a measurable partition of $X'$
  into sets of positive measure (hence non-empty). Furthermore, if $x_\sigma\in X_\sigma\setminus Y$
  ($\sigma\in\{0,1\}^t$) then $(x_\sigma\mid\sigma\in\{0,1\}^t)$ and $\{y_\tau \mid
  m\in\{0,1,\ldots,t-1\},\tau\in\{0,1\}^m\}$ form a tree of height $t$ in $W'$.

  By~\cite[Lemma~6.7.9]{Hod93}, we know that $n$-stable graphs do not contain trees of height
  $2^{n+2}-2$, so the construction above must stop otherwise it would produce an off-diagonal copy
  of a graph containing a half-graph of order $n$ (the fact that the copy is off-diagonal follows
  since no $x_\sigma$ is equal to any $y_\tau$). Thus there exists $\widetilde{\sigma}$ such that
  for every $z\in X'$, we have $\mu'(N_{W'}(z)\cap
  X_{\widetilde{\sigma}})/\mu'(X_{\widetilde{\sigma}})\in\{0,1\}$. Let
  \begin{align*}
    Z_i
    & \df
    \left\{z\in X_{\widetilde{\sigma}} \;\middle\vert\;
    \frac{\mu'(N_{W'}(z)\cap X_{\widetilde{\sigma}})}{\mu'(X_{\widetilde{\sigma}})} = i
    \right\} \quad (i\in\{0,1\})
  \end{align*}
  and since $X_{\widetilde{\sigma}} = Z_0\cup Z_1$ and $\mu'(X_{\widetilde{\sigma}}) > 0$, there
  exists $i_0\in\{0,1\}$ such that $\mu'(Z_{i_0}) > 0$. Finally, taking $U\in\cA$ such that
  $\mu'(Z_{i_0}\symdiff U) = 0$ gives $\mu(U) = \mu'(Z_{i_0}) > 0$ and $W\rest_{U\times U} =
  W'\rest_{U\times U} = i_0$ a.e.
\end{proof}

\begin{discussion}\label{dsc:goodsets}
  As we mentioned before, the proof of Lemma~\ref{lem:stablegraphon} can be seen as the construction
  of a \emph{$0$-good} set $X_{\widetilde{\sigma}}$ in the graphon $W$, that is, a positive measure
  set $U\subseteq X$ such that almost every $x\in X$ is either adjacent to almost all of $U$ or
  almost none of $U$ in the sense that
  \begin{align*}
    \frac{1}{\mu(U)}\int_U W(x,y)\ d\mu(y) & \in \{0,1\}.
  \end{align*}

  Another important concept in~\cite{MS14,MS21} that can be generalized to graphons is that of
  excellent sets. Let us say that a \emph{$0$-excellent} set in a graphon $W$ is a $0$-good
  set\footnote{In the finite case, we do not need to explicitly require excellent sets to be good as
    the goodness property follows from the excellent property when $V$ is a single vertex (which is
    necessarily a good set in the finite). However, in the limit, a single vertex is not good as it
    does not have positive measure.} $U$ such that for every $0$-good set $V$ we either have almost
  all edges between $U$ and $V$ or we have almost no edges between $U$ and $V$ in the sense that
  \begin{align*}
    \frac{1}{\mu(U\times V)}\int_{U\times V} W(x,y)\ d\mu(x,y) & \in \{0,1\}.
  \end{align*}

  Under this definition, it is easy to generalize the proof of Lemma~\ref{lem:stablegraphon} to
  prove that every $0$-good set $U$ in an almost stable graphon $W$ contains some $0$-excellent set:
  one can simply repeat the inductive construction by starting with $X_\varnothing\df U$ an use
  $0$-good sets $Y_\sigma$ for the internal nodes instead of single vertices $y_\sigma$. Composing
  these two and with a transfinite induction, it then follows that if $W$ is an almost stable
  graphon over $\Omega=(X,\cA,\mu)$, then there exists a countable partition $(U_i)_{i\in I}$ of $X$
  into positive measure sets such that for each $i,j\in I$, there exists $b_{i,j}\in\{0,1\}$ such
  that $W\rest_{U_i\times U_j}\equiv b_{i,j}$ a.e., that is, $W$ is a $\{0,1\}$-valued ``countable
  step-graphon''. This can be seen as a $0$-error version of the stable regularity
  lemma~\cite{MS14,MS21} in the limit.
\end{discussion}

Let us now show the stability dichotomy theorem for graphons.

\begin{proofof}{Theorem~\ref{thm:stablegraphon}}
  The
  implication~\ref{thm:stablegraphon:uniformsubgraphon}$\implies$\ref{thm:stablegraphon:stablesubgraphon}
  is trivial as the constant $0$ and constant $1$ graphons are almost stable.

  \medskip

  For the
  implication~\ref{thm:stablegraphon:uniformsquare}$\implies$\ref{thm:stablegraphon:uniformsubgraphon},
  using Lemma~\ref{lem:subgraphon} with the indicator function $f\df\One_U$ of $U$, we obtain a subgraphon
  $W_f$ that is either a.e.\ equal to $1$ or a.e.\ equal to $0$.

  \medskip

  For the final
  implication~\ref{thm:stablegraphon:stablesubgraphon}$\implies$\ref{thm:stablegraphon:uniformsquare},
  if $W'$ is an almost stable subgraphon of $W$, then by Lemma~\ref{lem:subgraphon} there exists
  $f\function{X}{[0,1]}$ with $\int_X f\ d\mu > 0$ such that $\phi_{W'} = \phi_{W_f}$ for the
  graphon $W_f$ over $\Omega_f=(X,\cA,\mu_f)$ given by $W_f(x,y)\df W(x,y)$. Let $V\df\{x\in X \mid
  f(x) > 0\}$, let $g\df\One_V$ be the indicator function of $V$ and consider the subgraphon $W_g$
  of $W$ over the space $\Omega_g = (X,\cA,\mu_g)$ corresponding to $g$ via
  Lemma~\ref{lem:subgraphon}.

  We claim that $W_g$ is almost stable. This is a standard measure theoretic trick: for $\epsilon >
  0$ and a finite graph $G$, let
  \begin{align*}
    \Tind^\epsilon(G,W_f)
    & \df
    \{(x,y)\in\Tind(G,W_f) \mid \forall v\in V(G), f(x_v) > \epsilon\},
    \\
    \Tind^\epsilon(G,W_g)
    & \df
    \{(x,y)\in\Tind(G,W_g) \mid \forall v\in V(G), f(x_v) > \epsilon\},
  \end{align*}
  then it is easy to see that
  \begin{multline*}
    (\mu_f^{V(G)}\otimes\lambda^{\binom{V(G)}{2}})(\Tind^\epsilon(G,W_f))
    \\
    \geq
    \left(\epsilon\cdot\frac{\mu(V)}{\int_X f\ d\mu}\right)^{\lvert G\rvert}
    (\mu_g^{V(G)}\otimes\lambda^{\binom{V(G)}{2}})(\Tind^\epsilon(G,W_g)).
  \end{multline*}
  On the other hand, since we have $\Tind(G,W_f) = \bigcup_{n\in\NN_+}\Tind^{1/n}(G,W_f)$
  $\mu_f$-a.e.\ and $\Tind(G,W_g) = \bigcup_{n\in\NN_+}\Tind^{1/n}(G,W_g)$ $\mu_g$-a.e., it follows
  that
  \begin{align*}
    \forall G, (\tind(G,W_f) = 0 \implies\tind(G,W_g) = 0)
  \end{align*}
  and since $W_f$ is almost stable, we get that $W_g$ is almost stable.

  By Lemma~\ref{lem:stablegraphon}, there exists a measurable set $\widetilde{U}\subseteq X$ such
  that $\mu_g(\widetilde{U}) > 0$ and either $W_g\rest_{\widetilde{U}\times\widetilde{U}} = 1$
  $\mu_g$-a.e.\ or $W_g\rest_{\widetilde{U}\times\widetilde{U}} = 0$ $\mu_g$-a.e. The result now
  follows by setting $U\df\widetilde{U}\cap V$ since $\mu_g(A) = \mu(A\cap V)/\mu(V)$ for every
  measurable set $A\subseteq X$.
\end{proofof}

\begin{remark}
  Recall that the set $U$ produced by Lemma~\ref{lem:stablegraphon} actually has a stronger property
  than simply almost clique or almost anti-clique, namely, it is a $0$-good set. By tracking down
  the application of this lemma in the proof of Theorem~\ref{thm:stablegraphon} above, we conclude
  that if a graphon $W$ contains some almost stable subgraphon $W'$ of measure $c$, then it contains
  positive measure sets $U$ and $V$ such that $\mu(V)\geq c$, $W\rest_{V\times V}$ is almost stable
  and $U$ is a $0$-good set in $W\rest_{V\times V}$.
\end{remark}

A natural question that arises is whether it is possible for a graphon to not contain any almost
stable subgraphon. A trivial example is obtained by considering quasirandom graphons: for $p\in
[0,1]$, let $W_p$ be the constant $p$ graphon. It is not hard to see from Lemma~\ref{lem:subgraphon}
that $W_p$ is the only subgraphon of $W_p$. In fact, the content of one of the original graph
quasirandomness equivalences~\cite[$P_1\iff P_4$]{CGW89} (see also~\cite[Theorem~3.4]{SS97}) is
precisely that these are the only graphons with this property. Since $W_p$ is not almost stable if
$0 < p < 1$ (as $\tind(G,W_p) = p^{\lvert E(G)\rvert}(1-p)^{\lvert E(\overline{G})\rvert} > 0$ for
every finite graph $G$), it follows that none of its subgraphons are either.

One can then ask if this is not an artifact of the fact that $W_p$ has fractional values, that is,
could it be that $\{0,1\}$-valued graphons must necessarily contain some almost stable subgraphon?
The next example answers this in the negative. We will also show in
Lemma~\ref{lem:phiC4notrivialsubobject} that the recursive blow-up of $C_4$ (see
Definition~\ref{def:recC4}) is another such example.

\begin{example}\label{ex:quasirandompermuton}
  Let $\Omega$ be the space $[0,1]^2$ equipped with the $2$-dimensional Lebesgue measure over Borel
  sets and consider the graphon $W$ over $\Omega$ given by
  \begin{align*}
    W((x_1,x_2),(y_1,y_2))
    & \df
    \begin{dcases*}
      1, & if $x_1 < y_1\tot x_2 < y_2$ and $x_1,x_2,y_1,y_2$ are distinct;\\
      0, & otherwise
    \end{dcases*}
  \end{align*}
  Clearly $W$ is $\{0,1\}$-valued.

  We claim that for every positive measure set $U\subseteq [0,1]^2$, $W\rest_{U\times U}$ is not
  a.e.\ constant. By Theorem~\ref{thm:stablegraphon}, this in particular means that $W$ does not
  have any almost stable subgraphon.
  
  Note that if $W\rest_{U\times U} = 1$ a.e., then for every $n\in\NN$, we must have $\tind(K_n,W)
  \geq\lambda(U)^n$. On the other hand, if $W\rest_{U\times U} = 0$ a.e., then for every $n\in\NN$,
  we must have $\tind(\overline{K}_n,W)\geq\lambda(U)^n$. In fact, by~\cite[Theorem~6]{CKP21}, for
  any graphon $W'$, we have
  \begin{align*}
    \sup\{\lambda(U) \mid W'\rest_{U\times U} = 1\text{ a.e.}\}
    & =
    \lim_{n\to\infty} \tind(K_n,W')^{1/n},
    \\
    \sup\{\lambda(U) \mid W'\rest_{U\times U} = 0\text{ a.e.}\}
    & =
    \lim_{n\to\infty} \tind(\overline{K}_n,W')^{1/n}.
  \end{align*}

  However, it is easy to see that
  \begin{align*}
    \tind(K_n,W) = \tind(\overline{K}_n,W) = \frac{1}{n!}
  \end{align*}
  as $\tind(K_n,W)$ is the probability that the relative order of the coordinates of $\rn{x}$
  matches that of the coordinates of $\rn{y}$ when both are picked independently and uniformly in
  $[0,1]^n$ and $\tind(\overline{K}_n,W)$ is that these relative orders are the precise inverses of
  each other (a more detailed explanation will be given in
  Example~\ref{ex:quasirandompermuton2}). Since for every $c > 0$, there exists $n\in\NN$ such that
  $1/n! < c^n$, the claim follows.

  To visualize $W$, we can consider the standard measure-isomorphism $F$ modulo $0$ from $[0,1]^2$
  to $[0,1]$ that maps $(w,z) = (0.w_1w_2\cdots,0.z_1z_2\cdots)$ to $0.w_1z_1w_2z_2\cdots$ using the
  binary expansions of the numbers $w$ and $z$. The graphon $W'$ over $[0,1]$ given indirectly by
  $W'(F(x),F(y))\df W(x,y)$ then represents the same limit as $W$, see
  Figure~\ref{fig:agreementsgraphon}.
\end{example}

\begin{figure}[htb]
  \input{agreementsgraphon}
\end{figure}

\begin{discussion}
  Another way of seeing Example~\ref{ex:quasirandompermuton} is as a ``recursive half-graphon'' that
  does not contain any almost stable graphon: we start by splitting the space $[0,1]^2$ into two
  parts $A_0\df[0,1]\times[0,1/2]$ and $A_1\df[0,1]\times(1/2,1]$ and put a half-graphon (see
    Figure~\ref{fig:halfgraphon}) between these two parts by setting $W((x_1,y_1),
    (x_2,y_2))\df\One[x_1 < x_2]$ for every $(x_1,y_1)\in A_0$ and every $(x_2,y_2)\in A_1$. We then
    split each of the halves in two and proceed recursively splitting the space along the dyadics in
    the second coordinate. It is easy to see that this recursive construction gives the graphon of
    Example~\ref{ex:quasirandompermuton}, which intuitively has half-graphons within every
    subgraphon. We will see in Example~\ref{ex:quasirandompermuton2} that another way of
    interpreting this graphon is as the graphon of agreements of the quasirandom permuton.
\end{discussion}

\section{Consequences for finite graphs}
\label{sec:consgraphs}

The objective of this section is to transfer Theorem~\ref{thm:stablegraphon} to the finite
world. There are several different ways that one can construct different geometric limit objects
that encode convergent sequences, each of the different approaches brings to light new connections
between the finite and the infinite. The approach of \Lovasz--Szegedy~\cite{LS06} (see
also~\cite{Lov12}) relied on \Szemeredi's Regularity Lemma~\cite{Sze78} and the graph cut-norm, the
approach of Diaconis--Janson~\cite{DJ08} uses the theory of exchangeable arrays (see~\cite{Kal05}
for more on this theory), the approach of Elek--Szegedy~\cite{ES12} uses ultraproducts and, more
recently, the approach of \Dolezal--\Grebik--\Hladky--Rocha--\Rozhon~\cite{DGHRR21} uses weak${}^*$
convergence when we think of the space of graphons as $L^\infty(\Omega^2)$.

To transfer Theorem~\ref{thm:stablegraphon}, the ultraproduct method Elek--Szegedy~\cite{ES12} (and
its generalization by Aroskar--Cummings~\cite{AC14}) will be of particular importance as it allows
pulling back properties from the infinite to convergent sequences via \Los's Theorem for
ultraproducts. We describe this method informally here and defer formal definitions to
Appendix~\ref{sec:ultraproduct} (we refer the combinatorially oriented reader to~\cite[\S 2.7]{ES12}
for an application-oriented introduction to (countable) ultrafilters and ultraproducts and
to~\cite[Chapter~4]{CK90} for a more thorough introduction).

Consider a sequence of graphs $(H_n)_{n\in\NN}$ of increasing sizes and let $V_n\df V(H_n)$. Note
that for each graph $G$, the set $\Tind(G,H_n)$ of embeddings of $G$ in $H_n$ can be seen as a
subset of $V_n^{V(G)}$ and we have
\begin{align*}
  \tind(G,H_n) & = \mu_n^{V(G)}(\Tind(G,H_n)) + o_G(1),
\end{align*}
where $\mu_n^{V(G)}$ is the normalized counting measure on $V_n^{V(G)}$ given by $\mu_n^{V(G)}(A) =
\lvert A\rvert/\lvert H_n\rvert^{\lvert G\rvert}$ and the error term $o_G(1)$ goes to $0$ as
$n\to\infty$ for each fixed $G$ and accounts for the fact that the normalization in $\tind$ is
$(\lvert H_n\rvert)_{\lvert G\rvert}$ instead of $\lvert H_n\rvert^{\lvert G\rvert}$. We then
consider the ultraproduct $H\df\prod_{n\in\NN} H_n/\cD$ and note that \Los's Theorem
(see~\cite[Theorem~9.5.1]{Hod93}) implies that the set of embeddings $\Tind(G,H)$ of $G$ in $H$ is
an internal set of $\prod_{n\in\NN} V_n^{V(G)}/\cD$, namely, we have $\Tind(G,H) =
\prod_{n\in\NN}\Tind(G,H_n)/\cD$. Going one step further if $\mu^{V(G)}$ is the Loeb measure
corresponding to $(\mu_n^{V(G)})_{n\in\NN}$, then we have
\begin{align*}
  \mu^{V(G)}(\Tind(G,H))
  & =
  \mu^{V(G)}\left(\prod_{n\in\NN}\Tind(G,H_n)/\cD\right)
  \\
  & =
  \lim_{n\to\cD} \mu_n^{V(G)}(\Tind(G,H_n))
  =
  \lim_{n\to\cD} \tind(G,H_n).
\end{align*}
If the sequence $(H_n)_{n\in\NN}$ is convergent, this ultralimit must be equal to the actual limit
$\lim_{n\to\infty} \tind(G,H_n)$; this means that the ultraproduct $\prod_{n\in\NN} H_n/\cD$ along
with the Loeb measures $\mu^U$ for each finite set $U$ successfully encode the ``limit'' of the
sequence $(H_n)_{n\in\NN}$. The ``problem'' with this encoding is that equipping $\prod_{n\in\NN}
V_n^U/\cD$ with the Loeb measure $\mu^U$ gives a probability space that is far from being standard,
namely it is non-separable. Moreover, if $\sigma(U)$ is the $\sigma$-algebra of $\mu^U$, then for
$U_1,U_2$ disjoint and non-empty, $\sigma(U_1\cup U_2)$ contains many more sets than the completion
of the product $\sigma$-algebra $\sigma(U_1)\otimes\sigma(U_2)$ with respect to the product measure
$\mu^{U_1}\otimes\mu^{U_2}$ (even though Fubini's Theorem still holds: the $\mu^U$-measure of a set
in $\sigma(U)$ can be computed via iterated integrals with respect to $\mu^{U_1}$ and $\mu^{U_2}$,
see Theorem~\ref{thm:FubiniLoeb}).

To address this issue, Elek and~Szegedy encode this ultraproduct probability space in the
probability space $[0,1]^{r(k)}$ equipped with the Lebesgue measure (recall that
$r(k)\df\{A\subseteq[k]\mid A\neq\varnothing\}$); this encoding is done via separable realizations,
which can be seen as a structured version of Maharam's Theorem~\cite{Mah42}. Informally, a
\emph{separable realization of order $k\in\NN_+$} is a measure-preserving function
$\Theta\function{\prod_{n\in\NN} V_n^k/\cD}{[0,1]^{r(k)}}$ that preserves enough structure of the
probability space so that:
\begin{enumerate}
\item For each $m\in[k]$, there exists a \emph{restriction of $\Theta$ of order $m$}, that is, a
  separable realization $\Theta_m\function{\prod_{n\in\NN} V_n^m/\cD}{[0,1]^{r(m)}}$ of order $m$
  such that the diagram
  \begin{equation*}
    \begin{tikzcd}
      \prod_{n\in\NN} V_n^k/\cD \arrow[r, "\Theta"]\arrow[d, "\alpha^*"'] &
      {[0,1]}^{r(k)} \arrow[d, "\alpha^*"]
      \\
      \prod_{n\in\NN} V_n^m/\cD \arrow[r, "\Theta_m"] &
      {[0,1]}^{r(m)}
    \end{tikzcd}
  \end{equation*}
  commutes for every injection $\alpha\injection{[m]}{[k]}$ (recall that $\alpha^*$ is the
  ``projection'' given by $\alpha^*(x)_u = x_{\alpha(u)}$).
\item For every $m\geq k$, there exists a \emph{lifting of $\Theta$ of order $m$}, that is, a
  measure-preserving $\Theta_m\function{\prod_{n\in\NN} V_n^m/\cD}{[0,1]^{r(m,k)}}$ such that the
  diagram
  \begin{equation*}
    \begin{tikzcd}
      \prod_{n\in\NN} V_n^m/\cD \arrow[r, "\Theta_m"]\arrow[d, "\alpha^*"'] &
      {[0,1]}^{r(m,k)} \arrow[d, "\alpha^*"]
      \\
      \prod_{n\in\NN} V_n^k/\cD \arrow[r, "\Theta"] &
      {[0,1]}^{r(k)}
    \end{tikzcd}
  \end{equation*}
  commutes for every injection $\alpha\injection{[k]}{[m]}$.
\end{enumerate}

The main theorem of the ultraproduct method for hypergraphs is then the following, which we state
only for the graph case (see also Theorem~\ref{thm:theonultraproduct} for the general case).

\begin{theorem}[Elek--Szegedy~\protect{\cite{ES12}}]\label{thm:graphonultraproduct}
  For every sequence of graphs of increasing sizes $(H_n)_{n\in\NN}$ and every non-principal
  ultrafilter $\cD$ over $\NN$, there exists a separable realization
  $\Theta\function{\prod_{n\in\NN} V(H_n)^2/\cD}{[0,1]^{r(2)}}$ of order $2$ and a measurable set
  $\cN\subseteq [0,1]^{r(2)}$ such that
  \begin{align*}
    \mu^2\left(\Theta^{-1}(\cN)\symdiff\prod_{n\in\NN} E^{H_n}/\cD\right) & = 0,
  \end{align*}
  where $E^{H_n}\df\{(v,w)\in V(H_n)^2 \mid H_n\vDash E(v,w)\}$ is the set of edges of $H_n$ as
  ordered pairs.
\end{theorem}

With the theorem above, one can then define the graphon
\begin{align*}
  W(x,y) & \df \lambda(\{z\in [0,1] \mid (x,y,z)\in\cN\})
\end{align*}
and from the properties of restrictions and liftings of the separable realization $\Theta$, it
follows that if $G$ is a finite graph with $V(G) = [m]$, then
\begin{align*}
  \mu^m\left(\Tind\left(G,\prod_{n\in\NN} H_n/\cD\right)\right)
  & =
  \mu^m(\Theta_m^{-1}(\Tind(G,\cN)))
  =
  \lambda(\Tind(G,\cN))
  \\
  & =
  \lambda(\Tind(G,W))
  =
  \tind(G,W),
\end{align*}
where
\begin{align*}
  \Tind(G,\cN)
  \df
  \biggl\{
  x\in [0,1]^{r(m,2)}
  &
  \mathrel{\bigg\vert}
  \forall\{v,w\}\in\binom{V(G)}{2},
  \\
  & (\{v,w\}\in E(G)\tot (x_{\{v\}},x_{\{w\}},x_{\{v,w\}})\in\cN)
  \biggr\}.
\end{align*}
This completes the translation of convergent sequences into graphons.

\medskip

In the next theorem, we explore this connection by pulling back the almost clique or almost empty graphon
$W\rest_{U\times U}$ provided by Theorem~\ref{thm:stablegraphon} through the separable realization of
Theorem~\ref{thm:graphonultraproduct} to the ultraproduct and producing a linear-sized almost clique or almost
empty graph in the convergent sequence. Naturally, we need the analogue of almost stability for increasing
sequences of graphs.

\begin{definition}
  We say that an increasing sequence of graphs $(H_n)_{n\in\NN}$ is \emph{almost stable} if there
  exists $\ell\in\NN$ such that every finite graph $G$ containing a half-graph of order $\ell$
  satisfies $\lim_{n\to\infty} p(G,H_n) = 0$.
\end{definition}

The next theorem is the stability dichotomy for convergent sequences of graphs, which in plain
English says that a convergent sequence of graphs $(H_n)_{n\in\NN}$ contains a \emph{sequence} of
linear-sized induced subgraphs that is either an almost clique or an almost anti-clique if and only
if it contains a \emph{subsequence} of linear-sized induced subgraphs that is almost stable. A
posteriori, it is clear that these conditions are also equivalent to $(H_n)_{n\in\NN}$ containing a
\emph{subsequence} of linear-sized induced subgraphs that is either an almost clique or an almost
anti-clique and equivalent to $(H_n)_{n\in\NN}$ containing a \emph{sequence} of linear-sized induced
subgraphs that is almost stable.

\begin{theorem}\label{thm:stableconvseqgraph}
  The following are equivalent for a convergent sequence of graphs $(H_n)_{n\in\NN}$.
  \begin{enumerate}
  \item There exist $c > 0$ and sets $U_n\subseteq V(H_n)$ such that $\lvert U_n\rvert\geq
    c\cdot\lvert H_n\rvert$ for every $n\in\NN$ and $\lim_{n\to\infty}
    p(\rho,H_n\rest_{U_n})\in\{0,1\}$.%
    \label{thm:stableconvseqgraph:uniform}
  \item There exist a subsequence $(H_{n_\ell})_{\ell\in\NN}$ of $(H_n)_{n\in\NN}$ and sets
    $U_{n_\ell}\subseteq V(H_{n_\ell})$ such that $\limsup_{\ell\to\infty} \lvert U_{n_\ell}\rvert /
    \lvert H_{n_\ell}\rvert > 0$ and $(H_{n_\ell}\rest_{U_{n_\ell}})_{\ell\in\NN}$ is almost
    stable.%
    \label{thm:stableconvseqgraph:stable}
  \end{enumerate}
\end{theorem}

\begin{proof}
  The implication~\ref{thm:stableconvseqgraph:uniform}$\implies$\ref{thm:stableconvseqgraph:stable} is
  trivial as $\lim_{n\to\infty} p(\rho,H_n\rest_{U_n})\in\{0,1\}$ implies $(H_n\rest_{U_n})_{n\in\NN}$ is
  almost stable.

  \medskip

  For the implication~\ref{thm:stableconvseqgraph:stable}$\implies$\ref{thm:stableconvseqgraph:uniform}, fix
  any graphon $W$ over some space $\Omega=(X,\cA,\mu)$ that is a limit of $(H_n)_{n\in\NN}$ and note that
  since this sequence is convergent, every subsequence of $(H_n)_{n\in\NN}$ converges to $W$. By hypothesis,
  possibly passing to a further subsequence $(H_{m_\ell})_{\ell\in\NN}$, there are sets $U_{m_\ell}\subseteq
  V(H_{m_\ell})$ with $\lim_{\ell\to\infty} \lvert U_{m_\ell}\rvert/\lvert H_{m_\ell}\rvert > 0$ such that
  $(H_{m_\ell}\rest_{U_{m_\ell}})_{\ell\in\NN}$ is both almost stable and convergent and if $\widehat{W}$ is a
  limit graphon of this sequence, then it is a stable subgraphon of $W$. By Theorem~\ref{thm:stablegraphon},
  we conclude that there exists a positive measure set $U\subseteq X$ such that $W\rest_{U\times U} = 1$
  a.e.\ or $W\rest_{U\times U} = 0$ a.e.; we let $b\in\{0,1\}$ be the a.e.\ value of $W\rest_{U\times U}$.

  Let $c\df\mu(U)$. We claim now that if $W'$ is another graphon over some space $\Omega'=(X',\cA',\mu')$ that
  is a limit of $(H_n)_{n\in\NN}$, then there exists a positive measure set $U'\subseteq X'$ such that
  $W'\rest_{U'\times U'} = b$ a.e. and $\mu'(U')\geq c$.

  This is completely trivial from the Graphon Uniqueness Theorem~\cite{BCL10} (see
  also~\cite[Theorem~13.10]{Lov12}), but we offer an ad hoc proof here: by
  Lemma~\ref{lem:subgraphon} applied to $W$ and the indicator function $\One_U$, there exist a
  sequence of graphs $(H'_n)_{n\in\NN}$ converging to $W$ and sets $U'_n\subseteq V(H'_n)$ with
  $\lim_{n\to\infty} \lvert U'_n\rvert/\lvert H'_n\rvert = c$ and $(H_n'\rest_{U'_n})_{n\in\NN}$
  converging to the constant $b$ graphon. Since $(H'_n)_{n\in\NN}$ also converges to $W'$, by the
  same lemma, we get a measurable function $f\function{X'}{[0,1]}$ with $\int_{X'} f\ d\mu' = c$
  such that the graphon $W_f$ over $\Omega_f$ given by $W_f(x,y) = W'(x,y)$ is $\mu'_f$-a.e.\ equal
  to $b$. Taking $U'\df\{x\in X' \mid f(x)>0\}$ gives $\mu'(U')\geq c$ and $W'\rest_{U'\times U'} =
  b$ $\mu'$-a.e.

  Therefore, we have shown that there exists $c > 0$ such that every graphon $W$ over some space
  $\Omega=(X,\cA,\mu)$ that is a limit of $(H_n)_{n\in\NN}$ has a measurable set $U\subseteq X$ such
  that $\mu(U)\geq c$ and $W\rest_{U\times U} = b$ a.e.

  For each $n\in\NN$, let $U_n^c\subseteq V(H_n)$ be a set that minimizes $\lvert
  p(\rho,H_n\rest_{U_n}) - b\rvert$ over all possible sets $U_n\subseteq V(H_n)$ such that $\lvert
  U_n\rvert\geq (c/2)\cdot\lvert H_n\rvert$. To conclude the proof, it is sufficient to show that
  $\lim_{n\to\infty} p(\rho,H_n\rest_{U_n^c}) = b$. Suppose not. Then there exists a subsequence
  $(H_{m_\ell})_{\ell\in\NN}$ of $(H_n)_{n\in\NN}$ such that $\lim_{\ell\to\infty} \lvert
  p(\rho,H_{m_\ell}\rest_{U^c_{m_\ell}}) - b\rvert > 0$ and by possibly passing to a further
  subsequence, we can also assume that $(H_{m_\ell}\rest_{U^c_{m_\ell}})_{\ell\in\NN}$ is
  convergent.

  We now let $H\df\prod_{\ell\in\NN} H_{m_\ell}/\cD$ for some non-principal ultrafilter $\cD$ over
  $\NN$ and let $\Theta\function{\prod_{\ell\in\NN} V(H_{m_\ell})^2/\cD}{[0,1]^{r(2)}}$ be a
  separable realization of order $2$ and $\cN$ be as in Theorem~\ref{thm:graphonultraproduct}. We
  also let $\Theta_1$ be a restriction of $\Theta$ of order $1$, let $W'$ be the graphon over
  $[0,1]$ defined by
  \begin{align*}
    W'(x,y) & \df \lambda(\{z\in[0,1] \mid (x,y,z)\in\cN\})
  \end{align*}
  and per our previous claim, let $U'\subseteq[0,1]$ be such that $\lambda(U')\geq c$ and
  $W'\rest_{U'\times U'} = b$ a.e. We define further $\widehat{U}\df\Theta_1^{-1}(U')\subseteq
  \prod_{\ell\in\NN} V(H_{m_\ell})/\cD$ and since $\Theta_1$ is measure-preserving, it follows that
  $\mu^1(\widehat{U})\geq c$ for the Loeb measure $\mu^1$.

  Consider now the graph $H\rest_{\widehat{U}}$ and note that $E^{H\rest_{\widehat{U}}} =
  \Theta^{-1}(\cN\cap(U'\times U'\times[0,1]))$ a.e.\ and since $W'\rest_{U'\times U'} = b$ a.e., it
  follows that $\mu^2(E^{H\rest_{\widehat{U}}}) = b\cdot \mu^2(\widehat{U}\times\widehat{U}) =
  b\cdot\mu^1(\widehat{U})^2$, where the last equality follows from Fubini's Theorem for Loeb
  measures, Theorem~\ref{thm:FubiniLoeb}. Let now $U\df \prod_{\ell\in\NN} U_\ell/\cD$ be an
  internal set such that $\mu^1(U\symdiff\widehat{U})=0$. By Fubini's Theorem again, it follows that
  $\mu^2(E^{H\rest_U}) = b\cdot\mu^1(U)^2$ so we must have
  \begin{align*}
    \lim_{\ell\to\cD} \frac{\lvert U_\ell\rvert}{\lvert H_{m_\ell}\rvert} & = \mu^1(U) \geq c,\\
    \lim_{\ell\to\cD} p(\rho,H_{m_\ell}\rest_{U_\ell}) & = \frac{\mu^2(E^{H\rest_U})}{\mu^1(U)^2} = b.
  \end{align*}

  However, this is a contradiction because it implies that along some subsequence we have $\lvert
  U_\ell\rvert/\lvert H_{m_\ell}\rvert\geq c/2$ and $p(\rho,H_{m_\ell}\rest_{U_\ell})\to b$
  contradicting the fact that the former implies $\lvert p(\rho,H_{m_\ell}\rest_{U_\ell}) -
  b\rvert\geq\lvert p(\rho,H_{m_\ell}\rest_{U^c_{m_\ell}}) - b\rvert$ and we have $\lvert
  p(\rho,H_{m_\ell}\rest_{U^c_{m_\ell}}) - b\rvert\not\to 0$.
\end{proof}

\begin{discussion}\label{dsc:convergence}
  Note that the convergence condition in Theorem~\ref{thm:stableconvseqgraph} is necessary for a
  very simple reason: if we take a sequence of increasing graphs that alternates between complete
  graphs (say, when $n$ is even) and empty graphs (say, when $n$ is odd), it is clearly not
  convergent and any linear-sized induced subgraph also alternates between almost clique or almost
  anti-clique.
\end{discussion}

\begin{discussion}\label{dsc:disjunioncliques}
  Na\"{i}vely, one might conjecture that if the sequence $(H_n)_{n\in\NN}$ itself is almost stable
  and we know the order of its stability, say, we know that $\lim_{n\to\infty} p(G,H_n) = 0$ for
  every finite graph $G$ containing a half-graph of order $\ell$, then one would be able to know
  bounds on the relative size $c$ of the sets $U_n$ depending only on $\ell$. However, this is not
  the case since if $H_{n,m}$ is the disjoint union of $m$ cliques of size $n$, then for each fixed
  $m\in\NN_+$, the sequence $(H_{n,m})_{n\in\NN}$ is convergent, does not contain any half-graphs of
  order $2$ and the maximum asymptotic relative size of an almost clique or almost anti-clique is
  $1/m$.

  This also shows the necessity of requiring almost cliques or almost anti-cliques as opposed to
  cliques or anti-cliques: the diagonal sequence $(H_{n,n})_{n\in\NN}$ is convergent but the largest
  cliques or anti-cliques in $H_{n,n}$ have size $n = \sqrt{\lvert H_{n,n}\rvert}$. However, the
  edge density in the sequence $(H_{n,n})_{n\in\NN}$ itself goes to zero so it is an almost
  anti-clique.
\end{discussion}

\begin{discussion}
  A posteriori, the example of Discussion~\ref{dsc:convergence} shows that we cannot get
  Theorem~\ref{thm:stableconvseqgraph} by simply applying the removal lemma followed by the stable
  regularity lemma~\cite{MS14,AFP18,MS21} to each of the $H_{n_\ell}\rest_{U_{n_\ell}}$ with a
  precision $\epsilon_{n_\ell} > 0$ as such argument does not use the required property of
  convergence of the sequence $(H_n)_{n\in\NN}$ in any way. The reason why the stable regularity
  lemma is not enough is that when applied to $\epsilon > 0$, it provides some $c(\epsilon) > 0$
  such that every sufficiently large $H$ has some set $U\subseteq V(H)$ of size at least
  $c(\epsilon)\cdot\lvert H\rvert$ that has edge density either at least $1-\epsilon$ or at most
  $\epsilon$. However, to obtain the almost clique or almost anti-clique, we need to make
  $\epsilon_{n_\ell}\to 0$ which also destroys our guaranteed lower bound on the relative size of
  the sets: $c(\epsilon_{n_\ell})\to 0$.
\end{discussion}

We now proceed to transfer the stability dichotomy to countable graphs.

\begin{theorem}\label{thm:stablecountablegraph}
  The following are equivalent for a graph $G$ with $V(G)=\NN_+$.
  \begin{enumerate}
  \item There exist a set $U\subseteq\NN_+$ and an increasing sequence $(n_\ell)_{\ell\in\NN}$ of
    positive integers such that $\lim_{\ell\to\infty} p(\rho,G\rest_{U\cap[n_\ell]})\in\{0,1\}$ and
    $\lim_{\ell\to\infty} \lvert U\cap [n_\ell]\rvert/n_\ell > 0$.%
    \label{thm:stablecountablegraph:uniform}
  \item There exist a set $U\subseteq\NN_+$ and an increasing sequence $(n_\ell)_{\ell\in\NN}$ of
    positive integers such that $(G\rest_{U\cap[n_\ell]})_{\ell\in\NN}$ is almost stable and
    $\lim_{\ell\to\infty} \lvert U\cap [n_\ell]\rvert/n_\ell > 0$.%
    \label{thm:stablecountablegraph:stable}
  \end{enumerate}
\end{theorem}

\begin{proof}
  The
  implication~\ref{thm:stablecountablegraph:uniform}$\implies$\ref{thm:stablecountablegraph:stable}
  is trivial as $\lim_{\ell\to\infty} p(\rho,G\rest_{U\cap[n_\ell]})\in\{0,1\}$ implies
  $(G\rest_{U\cap[n_\ell]})_{\ell\in\NN}$ is almost stable.

  \medskip

  For the
  implication~\ref{thm:stablecountablegraph:stable}$\implies$\ref{thm:stablecountablegraph:uniform},
  by possibly passing to a subsequence of $(n_\ell)_{\ell\in\NN}$, we may further assume that
  $(G\rest_{[n_\ell]})_{\ell\in\NN}$ is convergent, so by Theorem~\ref{thm:stableconvseqgraph},
  there exist $c > 0$ and sets $U_\ell\subseteq [n_\ell]$ such that $\lvert U_\ell\rvert\geq c\cdot
  n_\ell$ for every $\ell\in\NN$ and $\lim_{\ell\to\infty} p(\rho,G\rest_{U_\ell})\in\{0,1\}$.

  Define then the sequence $(m_t)_{t\in\NN}$ recursively by
  \begin{align*}
    m_0 & \df n_0, &
    m_{t+1} & \df \min\{n_\ell \mid \ell\in\NN\land n_\ell\geq 2^t\cdot m_t\}
  \end{align*}
  and for each $t\in\NN$, let $\ell_t\in\NN$ be such that $m_t = n_{\ell_t}$. Let also
  \begin{align*}
    U & \df \bigcup_{t\in\NN} U_{\ell_t}\cap ([m_t]\setminus [m_{t-1}]),
  \end{align*}
  where $m_{-1}\df 0$.

  Note that
  \begin{align*}
    \lvert(U\cap [m_t])\symdiff U_{\ell_t}\rvert
    & \leq
    m_{t-1}
    \leq
    2^{-t+1}\cdot\lvert U_{\ell_t}\rvert,
  \end{align*}
  hence $\lim_{t\to\infty} \lvert U\cap [m_t]\rvert/\lvert U_{\ell_t}\rvert = 1$, which implies that
  \begin{align*}
    \liminf_{t\to\infty} \frac{\lvert U\cap[m_t]\rvert}{m_t}
    & =
    \liminf_{t\to\infty}
    \frac{\lvert U_{\ell_t}\rvert}{n_{\ell_t}}
    \geq
    c
    >
    0,
  \end{align*}
  and
  \begin{align*}
    \lim_{t\to\infty} p(\rho,G\rest_{U\cap[m_t]})
    & =
    \lim_{t\to\infty}
    p(\rho,G\rest_{U_{\ell_t}})
    \in
    \{0,1\},
  \end{align*}
  completing the proof.
\end{proof}

\begin{discussion}\label{dsc:countable}
  One might na\"{i}vely hope that in the countable case one would be able to produce an almost
  clique or almost anti-clique $U$ of positive density (as opposed to positive upper density as in
  Theorem~\ref{thm:stablecountablegraph}), but a simple counter-example shows this is not possible:
  if $G$ is the graph over $\NN_+$ with edge set
  \begin{align*}
    E(G)
    & \df
    \{\{v,w\}
    \mid
    \lfloor\sqrt{\log_2(v)}\rfloor\equiv\lfloor\sqrt{\log_2(w)}\rfloor\equiv 0 \pmod{2}\}
  \end{align*}
  then $G$ is stable (as it is a union of cliques) and does not have any positive density almost
  clique or anti-clique simply because for each $\epsilon > 0$ and each $n_0\in\NN_+$, there exist
  $n,n'\geq n_0$ such that the edge density of the marginals $G\rest_{[n]}$ and $G\rest_{[n']}$ are
  at most $\epsilon$ away from $0$ and $1$, respectively.
\end{discussion}

\section{Trivial sub-objects in theons}
\label{sec:universaltheories}

In this section, we state and prove the stability dichotomy theorem for theons, Theorem~\ref{thm:stabletheon},
which is a generalization of Theorem~\ref{thm:stablegraphon} of Section~\ref{sec:graphons} for universal
theories over finite relational languages. Before we do so, let us give a gentle introduction to the theories
of flag algebras~\cite{Raz07} and theons~\cite{CR20a}, which generalize the theory of graphons to universal
theories.

First, given finite models $M$ and $N$ of a universal theory $T$ over a finite relational language $\cL$, we
let $\Tind(M,N)$ be the set of all model embeddings of $M$ in $N$ (i.e., injective functions
$f\injection{V(M)}{V(N)}$ that preserve all relations and negations of relations) and let
\begin{align*}
  \tind(M,N)
  & \df
  \frac{\lvert\Tind(M,N)\rvert}{(\lvert N\rvert)_{\lvert M\rvert}}
\end{align*}
be the normalized number of embeddings of $M$ in $N$, called the \emph{labeled (induced) density} of
$M$ in $N$. We also define the \emph{(induced) density} of $M$ in $N$ as the normalized number of
induced submodels of $N$ that are isomorphic to $M$ given by
\begin{align*}
  p(M,N)
  & \df
  \frac{\lvert\{U\subseteq V(N) \mid N\rest_U\cong M\}\rvert}{\binom{\lvert N\rvert}{\lvert M\rvert}}
  =
  \frac{\lvert M\rvert!}{\lvert\Aut(M)\rvert}\cdot\tind(M,N),
\end{align*}
where $\Aut(M)$ is the group of automorphisms of $M$. For each $n\in\NN$, we let $\cM_n[T]$ be the
set of models of $T$ of size $n$ up to isomorphism and we let\footnote{In the framework of limits,
  it is very convenient to assume that the vertex set of a structure/model can be empty and thus
  $\cM_0[T]$ is included in this union.\label{ftnt:emptyvertexset}}
$\cM[T]\df\bigcup_{n\in\NN}\cM_n[T]$. Note that the fact that $T$ is a universal theory implies that
$\cM[T]$ is closed under substructures, which in turn implies that for every $N\in\cM[T]$ and every
$n\leq\lvert N\rvert$, we have $\sum_{M\in\cM_n[T]} p(M,N) = 1$. For a set $V$, we also let
$\cK_V[T]$ be the set of models of $T$ whose vertex set is $V$ (we do \emph{not} factor isomorphisms
for $\cK_V[T]$).

Recall that for universal theories $T_1$ and $T_2$ over finite relational languages $\cL_1$ and
$\cL_2$, respectively, an \emph{open interpretation} (also known under the name \emph{definition})
from $T_1$ to $T_2$ is a function $I$ (denoted $I\interpret{T_1}{T_2}$) that maps each predicate
symbol $P\in\cL_1$ to an open (i.e., quantifier-free) formula $I(P)(x_1,\ldots,x_{k(P)})$ of
$\cL_2$, where $k(P)$ is the \emph{arity} of $P$ and such that for each axiom
$\forall\vec{x},F(\vec{x})$ of $T_1$, we have $T_2\vdash\forall\vec{x}, I(F)(\vec{x})$ when we
declare $I$ to commute with logical connectives. An open interpretation $I\interpret{T_1}{T_2}$
contra-variantly naturally defines maps $\cK_V[T_2]\to\cK_V[T_1]$ for each set $V$ given by
$(I(M)\vDash P(\vec{x}))\iff (M\vDash I(P)(\vec{x}))$ for each $P\in\cL_1$.

Two open interpretations $I_1,I_2\interpret{T_1}{T_2}$ are \emph{equivalent} if
$T_2\vdash\forall\vec{x},(I_1(P)(\vec{x})\tot I_2(P)(\vec{x}))$ for every predicate symbol
$P\in\cL_1$. Equivalently, the open interpretations $I_1,I_2\interpret{T_1}{T_2}$ are equivalent if
they induce the same maps $\cK_V[T_2]\to\cK_V[T_1]$ for every set $V$ (in fact, it is enough to know
that this is true for $V=[k]$, where $k$ is the maximum arity of a predicate of $T_1$). We let
$\cat{Int}$ be the category of universal theories in finite relational languages whose morphisms are
open interpretations up to equivalence. Note that if $I\interpret{T_1}{T_2}$ is an isomorphism of
$\cat{Int}$, then $p(I(M),I(N)) = p(M,N)$ for every $M,N\in\cM[T_2]$, which means that isomorphic
theories are \emph{indistinguishable} for the purposes of densities of submodels. Isomorphisms in
the category $\cat{Int}$ are also known under the name \emph{interdefinitions}.

It will be more convenient to work with \emph{canonical theories}, which are theories in which every
relation contains only injective tuples, that is, theories that entail
\begin{align}\label{eq:canonical}
  \forall x_1,\ldots,x_{k(P)}, &
  \left(\bigvee_{1\leq i < j\leq k(P)} x_i = x_j\to \neg P(x_1,\ldots,x_{k(P)})\right)
\end{align}
for every predicate symbol $P$. By~\cite[Theorem~2.3]{CR20a} (see also~\cite[\S 2.2]{AC14}), every
universal theory is isomorphic (in~$\cat{Int}$) to a canonical theory. From this point forward,
unless explicitly mentioned otherwise, \emph{all theories are assumed to be canonical}. For a finite
relational language, we let $T_\cL$ be the \emph{pure canonical theory over $\cL$}, that is, the
theory whose axioms are precisely the ones in~\eqref{eq:canonical} for each $P\in\cL$; the models of
$T_\cL$ are sometimes referred to as \emph{canonical structures in $\cL$}.

The notion of convergence is now defined analogously to the graph case: a sequence $(N_n)_{n\in\NN}$
of finite models of a canonical theory $T$ is called \emph{convergent} if it is \emph{increasing} in
the sense that for every $n\in\NN$, $\lvert N_n\rvert < \lvert N_{n+1}\rvert$ and if for every
$M\in\cM[T]$, the limit $\lim_{n\to\infty} p(M,N_n)$ exists. Again, another way of seeing this is as
convergence in the (compact and metrizable) product topology of $[0,1]^{\cM[T]}$ of the sequence
$(p(\place,N_n))_{n\in\NN}$.

The simplest way of encoding the limit of a convergent sequence $(N_n)_{n\in\NN}$ is
syntactically/algebraically by defining $\phi\in [0,1]^{\cM[T]}$ by $\phi(M)\df\lim_{n\to\infty}
p(M,N_n)$. The theory of flag algebras then describes which points of $[0,1]^{\cM[T]}$ can arise as
limits of convergent sequences. It turns out that this description boils down to some polynomial
restrictions plus a positivity condition. Namely, let $\RR\cM[T]$ be the vector space of formal
$\RR$-linear combinations of elements of $\cM[T]$. We then extend each $\phi\in [0,1]^{\cM[T]}$
linearly to a function $\phi\function{\RR\cM[T]}{\RR}$ (which we denote by abuse with the same
letter) as
\begin{align*}
  \phi\left(\sum_{M\in\cM[T]} c_M M\right) & \df \sum_{M\in\cM[T]} c_M \phi(M).
\end{align*}
Let $\cK[T]$ be the linear subspace of $\RR\cM[T]$ spanned by elements of the form
\begin{align*}
  M - \sum_{N\in\cM_n[T]} p(M,N)N
\end{align*}
for $n\geq\lvert M\rvert$ and let $\cA[T]\df\RR\cM[T]/\cK[T]$. It is straightforward to see that if
$\phi\df \lim_{n\to\infty} p(\place,N_n)$ for some convergent sequence $(N_n)_{n\in\NN}$, then
$\cK[T]\subseteq\ker(\phi)$, which means that we can think of $\phi$ instead as a linear map
$\phi\function{\cA[T]}{\RR}$ by factoring out $\cK[T]$. It turns out that $\cA[T]$ becomes an
$\RR$-algebra when equipped with the (bilinear) product operation defined by
\begin{align*}
  M_1\cdot M_2 & \df \sum_{M\in\cM_n[T]} p(M_1,M_2;M) M,
\end{align*}
for $n\geq\lvert M_1\rvert + \lvert M_2\rvert$, where
\begin{multline*}
  p(M_1,M_2;M)
  \\
  \df
  \frac{%
    \lvert\{(U_1,U_2)\in 2^{V(M)}\times 2^{V(M)}
    \mid M\rest_{U_1}\cong M_1\land M\rest_{U_2}\cong M_2
    \land U_1\cap U_2 = \varnothing
    \}\rvert
  }{%
    \binom{\lvert M\rvert}{\lvert M_1\rvert}\binom{\lvert M\rvert - \lvert M_1\rvert}{\lvert M_2\rvert}
  },
\end{multline*}
and the unit of $\cA[T]$ is the equivalence class of the element $\sum_{M\in\cM_n[T]} M$ for any
given $n\in\NN$. Furthermore, any $\phi$ coming from a convergent sequence respects this product
operation, in other words, $\phi$ is necessarily in the set $\Hom(\cA[T],\RR)$ of $\RR$-algebra
homomorphisms from $\cA[T]$ to $\RR$. In fact, by letting
\begin{align*}
  \HomT{T} & \df \{\phi\in\Hom(\cA[T],\RR) \mid \forall M\in\cM[T], \phi(M)\geq 0\}
\end{align*}
be the set of \emph{positive homomorphisms}, any $\phi$ coming from a convergent sequence is
necessarily a positive homomorphism. The main theorem below of flag algebra theory (sometimes
referred to as Existence Theorem for flag algebras) says that in fact the set $\HomT{T}$ is
precisely the set of all limits of convergent sequences.

\begin{theorem}[Lov\'{a}sz--Szegedy~\protect{\cite{LS06}}, Razborov~\protect{\cite{Raz07}}]
  Let $T$ be a universal theory in a finite relational language.

  If $(N_n)_{n\in\NN}$ is a convergent sequence of finite models of $T$, then there exists
  $\phi\in\HomT{T}$ such that $\lim_{n\to\infty} p(M,N_n)=\phi(M)$ for every
  $M\in\cM[T]$. Conversely, if $\phi\in\HomT{T}$, then there exists a convergent sequence
  $(N_n)_{n\in\NN}$ of $T$ such that $\lim_{n\to\infty} p(M,N_n)=\phi(M)$ for every $M\in\cM[T]$.
\end{theorem}

Note that because of the minimalist nature of the flag algebraic description, uniqueness here is
obvious: $\phi_1,\phi_2\in\HomT{T}$ represent the limit of the same convergent sequence
$(N_n)_{n\in\NN}$ if and only if $\phi_1=\phi_2$. For this reason, it is very convenient to use the
set $\HomT{T}$ when talking about limits of finite models of the theory $T$.

\medskip

For a semantic/geometric description of the limit objects we use the theory of theons~\cite{CR20a},
which generalizes the theory of graphons to describe limits of finite models of canonical theories.

Given an atomless standard probability space $\Omega=(X,\cA,\mu)$ and a set $V$, let
\begin{align*}
  \cE_V(\Omega) & \df X^{r(V)}
\end{align*}
and equip it with the completion of the product measure, which by abuse we also denote by $\mu$. We
also define the \emph{diagonal set} as (cf.~Equation~\eqref{eq:diagonalgraphon})
\begin{align*}
  \cD_V(\Omega) & \df \{x\in\cE_V(\Omega) \mid \exists v,w\in V, (v\neq w\land x_{\{v\}} = x_{\{w\}})\}.
\end{align*}
Clearly, the diagonal has zero-measure (and this is precisely the reason why we need to work with
canonical theories so that no information is lost). Again, we will typically take $\Omega$ to be
$[0,1]$ equipped with the Lebesgue measure over Borel sets and in this case, we will omit $\Omega$
from the notation.

We will also be abusing the notation slightly by identifying the spaces $\cE_V(\Omega\times\Omega)$
and $\cE_V(\Omega)\times\cE_V(\Omega)$ naturally via the correspondence
$\cE_V(\Omega\times\Omega)\ni x\tot (y,z)\in\cE_V(\Omega)\times\cE_V(\Omega)$ given by $y_A\df
(x_A)_1$ and $z_A\df (x_A)_2$.

For a predicate symbol $P$, a \emph{$P$-on} over $\Omega$ is a measurable subset of
$\cE_{k(P)}(\Omega)$, where $k(P)$ is the arity of $P$. An \emph{Euclidean structure} in a finite
relational language $\cL$ over $\Omega$ is a function $\cN$ that maps each predicate symbol
$P\in\cL$ to a $P$-on $\cN_P\subseteq\cE_{k(P)}(\Omega)$.

Analogously to the way that solution sets are defined, given an open formula $F(x_1,\ldots,x_n)$ in
$\cL$ and an Euclidean structure $\cN$ in $\cL$ over $\Omega$, the \emph{truth set}
$T(F,\cN)\subseteq\cE_n(\Omega)$ of $F$ is defined by
\begin{enumerate}
\item $T(x_i = x_i,\cN) \df \cE_n(\Omega)$.
\item $T(x_i = x_j,\cN) \df \varnothing$, if $i\neq j$.%
  \label{it:Tnoninjeq}
\item $T(P(x_{i_1},\ldots,x_{i_{k(P)}}),\cN) \df \varnothing$, if $i\function{[k(P)]}{[n]}$ is not
  injective.%
  \label{it:Tnoninj}
\item $T(P(x_{i_1},\ldots,x_{i_{k(P)}}),\cN) \df (i^*)^{-1}(\cN_P)$, if $i\injection{[k(P)]}{[n]}$
  is injective (recall that $i^*\function{\cE_n(\Omega)}{\cE_{k(P)}(\Omega)}$ is given by
  $i^*(x)_A\df x_{i(A)}$).
\item $T(\place,\cN)$ commutes with logical connectives (so, e.g., $T(\neg F,\cN)\df
  \cE_n(\Omega)\setminus T(F,\cN)$ and $T(F_1\land F_2,\cN)\df T(F_1,\cN)\cap T(F_2,\cN)$).
\end{enumerate}

One might complain that items~\ref{it:Tnoninjeq} and~\ref{it:Tnoninj} above should not be defined as
the empty set but rather as particular subsets of the diagonal $\cD_n(\Omega)$, but since all
information on the diagonal will be lost regardless, the definition uses the empty set for
simplicity.

Truth sets allow us to define the set of copies of a finite canonical structure $M$ as follows: if
$V(M)\df [n]$, then the \emph{open diagram} $\Dopen(M)(x_1,\ldots,x_n)$ of $M$ is the conjunction of
all formulas of the form
\begin{align*}
  x_i\neq x_j & & \text{with } i\neq j,\\
  P(x_{i_1},\ldots,x_{i_k}) & & \text{with } M\vDash P(i_1,\ldots,i_k),\\
  \neg P(x_{i_1},\ldots,x_{i_k}) & & \text{with } M\vDash \neg P(i_1,\ldots,i_k).
\end{align*}
Equivalently, it is the open formula that completely encodes the quantifier-free type (over the
empty set) of the tuple $(1,\ldots,n)$ in $M$ (recall that the language is finite). The set of
\emph{labeled (induced) copies} of $M$ in $\cN$ is defined as $\Tind(M,\cN)\df T(\Dopen(M),\cN)$. If
the vertex set $V(M)$ of $M$ is not $[n]$, then we simply relabel its vertices with a bijection
$\alpha\injection{V(M)}{[n]}$, where $n\df\lvert M\rvert$ to get a canonical structure $N$ with
vertex set $[n]$ such that
\begin{align*}
   (M\vDash P(v_1,\ldots,v_n)) & \iff (N\vDash P(\alpha(v_1),\ldots,\alpha(v_{k(P)})))
\end{align*}
and define $\Tind(M,\cN)\df \alpha^*(\Tind(N,\cN))\subseteq\cE_{V(M)}(\Omega)$ (it is easy to see
that this does not depend on the choice of $\alpha$). The \emph{labeled (induced) density} and the
\emph{(induced) density} of $M$ in $\cN$ are then defined respectively as
\begin{align*}
  \tind(M,\cN) & \df \mu(\Tind(M,\cN)), &
  \phi_\cN(M) & \df \frac{\lvert M\rvert!}{\lvert\Aut(M)\rvert}\cdot\tind(M,\cN).
\end{align*}

Finally, given a canonical theory $T$ over $\cL$ and an Euclidean structure $\cN$ in $\cL$ over
$\Omega$, we say that $\cN$ is a \emph{(weak) $T$-on} if every $\cL$-structure $M$ that is
\emph{not} a model of $T$ satisfies $\tind(M,\cN) = 0$ and we say that $\cN$ is a \emph{strong
  $T$-on} if every $\cL$-structure $M$ that is \emph{not} a model of $T$ satisfies
$\Tind(M,\cN)\subseteq\cD_{V(M)}(\Omega)$. We say that a weak or strong $T$-on $\cN$ is a limit of a
convergent sequence $(N_n)_{n\in\NN}$ of models of $T$ if $\lim_{n\to\infty} p(M,N_n) = \phi_\cN(M)$
for every model $M$ of $T$ (see Theorem~\ref{thm:substclosed} below for an equivalent definition in
terms of axioms of $T$).

The main theorem of the theory of theons is naturally the Existence Theorem for theons below.
\begin{theorem}[\protect{\cite[Theorem~3.4]{CR20a}}, see also~\protect{\cite[\S 3.1]{AC14}}]
  Let $T$ be a canonical universal theory in a finite relational language and $\Omega$ be an
  atomless probability space. If $(N_n)_{n\in\NN}$ is a convergent sequence of models of $T$, then
  there exists a $T$-on $\cN$ over $\Omega$ that is a limit of $(N_n)_{n\in\NN}$. Conversely, every
  $T$-on over $\Omega$ is a limit of a convergent sequence of models of $T$.
\end{theorem}

\begin{remark}\label{rmk:TGraphons}
  If $\TGraph$ is the theory of graphs, then a $\TGraph$-on $\cN$ is not exactly the same object as
  a graphon $W$, but there is a (not one-to-one) correspondence preserving densities of finite
  graphs given by
  \begin{align*}
    W_\cN(x,y) & \df \lambda(\{z\in[0,1] \mid (x,y,z)\in\cN\}),\\
    \cN & \df \{x\in\cE_2 \mid x_{\{1,2\}} < W(x_{\{1\}},x_{\{2\}})\}.
  \end{align*}
\end{remark}

Just as the Graphon Removal Lemma, Theorem~\ref{thm:graphonremoval}, allows us to remove graphs of
density zero from a graphon by only changing a zero-measure set, the Induced Euclidean Removal Lemma
below does the same for theons.
\begin{theorem}[Induced Euclidean Removal Lemma~\protect{\cite[Theorem~3.3]{CR20a}}]\label{thm:theonremoval}
  If $\cN$ is a $T$-on over $\Omega=(X,\cA,\mu)$, then there exists a strong $T$-on $\cN'$ over
  $\Omega$ such that $\mu(\cN_P\symdiff\cN'_P)=0$ for every predicate symbol $P$.
\end{theorem}

\begin{remark}\label{rmk:theonremoval}
  Theorem~\ref{thm:theonremoval} above can also be used to ensure that all open formulas that are
  false a.e.\ in $\cN$ become false everywhere off-diagonal in $\cN'$. Namely, given a $T$-on $\cN$
  over $\Omega=(X,\cA,\mu)$, we let $\Th(\phi_\cN)$ be the canonical theory whose finite models are
  precisely those $M$ such that $\phi_\cN(M) > 0$. Note that $\cN$ is also a (weak)
  $\Th(\phi_\cN)$-on, so we can apply Theorem~\ref{thm:theonremoval} above to get a strong
  $\Th(\phi_\cN)$-on $\cN'$ whose peons differ from those of $\cN$ only by zero-measure sets. If
  $\mu(T(F,\cN)) = 0$ for some open formula $F(x_1,\ldots,x_n)$, then for any $\cL$-structure $M$
  with $V(M)=[n]$ and $M\vDash F(1,\ldots,n)$, we must have $\tind(M,\cN)=0$ and thus
  $\Tind(M,\cN')\subseteq\cD_n(\Omega)$, which in turn implies $T(F,\cN')\subseteq\cD_n(\Omega)$.
\end{remark}

As expected from the graphon case, the same convergent sequence can converge to different theons and
this is completely characterized by the Theon Uniqueness Theorem~\cite[Theorems~3.9 and~3.11 and
  Proposition~7.7]{CR20a}, which has a very technical statement. Fortunately, we will only need a
consequence of it concerning open interpretations, Proposition~\ref{prop:interpretationlifting}
below. But before we state it, we need some preliminary definitions and properties.

First, open interpretations behave naturally with respect to convergence: it is not hard to see that
if $I\interpret{T_1}{T_2}$ is an open interpretation and $(N_n)_{n\in\NN}$ is a convergent sequence
of models of $T_2$, then $(I(N_n))_{n\in\NN}$ is a convergent sequence of models of $T_1$. It turns
out that there are natural operations that encode this operation for limit objects. Namely, for flag
algebras, Razborov~\cite[Theorem~2.6]{Raz07} showed that the linear map
$\pi^I\function{\cA[T_1]}{\cA[T_2]}$ given by
\begin{align*}
  \pi^I(M) & \df \sum\{M'\in\cM_{\lvert M\rvert}[T_2] \mid I(M')\cong M\}
\end{align*}
is an $\RR$-algebra homomorphism and if $\phi$ is the limit of $(N_n)_{n\in\NN}$, then the
composition $\phi^I\df \phi\comp\pi^I\in\HomT{T_1}$ is the limit of $(I(N_n))_{n\in\NN}$. For
theons~\cite[Remark~6]{CR20a}, if $\cN$ is a $T_2$-on that is the limit of $(N_n)_{n\in\NN}$, then
the $T_1$-on $I(\cN)$ defined via truth sets by $I(\cN)_P\df T(P,\cN)$ for every predicate symbol
$P$ is the limit of $(I(N_n))_{n\in\NN}$. We can combine these results neatly as $\phi_{I(\cN)} =
\phi_\cN^I$, or in plain English, the limit encoded by $I(\cN)$ is the same as the interpreted limit
of $\cN$ via $I$.

One natural question that arises is whether theons can be lifted through open interpretations in the
following sense: if $\cN$ is a $T_1$-on and $\phi\in\HomT{T_2}$ is such that $\phi^I = \phi_\cN$ for
some open interpretation $I\interpret{T_1}{T_2}$, then is there a $T_2$-on $\cH$ such that $\phi_\cH
= \phi$ and $I(\cH) = \cN$ a.e.? In plain English, if $\cN$ encodes the limit $\phi^I$, then is it
of the form $\cN = I(\cH)$ a.e.\ for some limit $\cH$ encoding $\phi$?

While the answer to this question is no (see~\cite[Example~45]{CR20a}), the following proposition
says that if we allow ourselves to add ``dummy variables'', the answer becomes yes.

\begin{proposition}[\protect{\cite[Proposition~4.3]{CR20b}}]\label{prop:interpretationlifting}
  Let $I\interpret{T_1}{T_2}$ be an open interpretation, let $\phi\in\HomT{T_2}$ and let $\cN$ be a
  $T_1$-on over $\Omega$ such that $\phi^I = \phi_\cN$. Then there exists a $T_2$-on $\cH$ over
  $\Omega\times\Omega$ such that $\phi_\cH = \phi$ and $I(\cH)_P = \cN_P\times\cE_{k(P)}(\Omega)$
  a.e.\ for every predicate symbol $P$ in the language of $T_1$.
\end{proposition}

\medskip

We now define limit sub-objects in analogy to subgraphons.

\begin{definition}
  Given a limit object $\phi\in\HomT{T}$, a \emph{(positive measure) limit sub-object} of $\phi$ is
  a limit object $\psi\in\HomT{T}$ such that there exists a sequence $(N_n)_{n\in\NN}$ converging to
  $\phi$ and sets $U_n\subseteq V(H_n)$ such that $(N_n\rest_{U_n})_{n\in\NN}$ converges to $\psi$
  and $\lim_{n\to\infty} \lvert U_n\rvert/\lvert N_n\rvert > 0$; when we want to be more specific,
  for $c\df \lim_{n\to\infty} \lvert U_n\rvert/\lvert N_n\rvert > 0$ we say that $\psi$ is a measure
  $c$ limit sub-object of $\phi$.
\end{definition}

Similarly to the graphon case, if $\cN$ is a $T$-on over $\Omega=(X,\cA,\mu)$ with $\phi_\cN=\phi$,
then not every sub-object of $\phi$ can be represented by conditioning the vertex variables
$x_{\{v\}}$ to be in some positive measure set $U\subseteq X$.

More precisely, given a $T$-on $\cN$ over a space $\Omega=(X,\cA,\mu)$ and a positive measure set
$U\subseteq X$, we let $\mu_U$ be the measure over $(X,\cA)$ defined by $\mu_U(A)\df \mu(A\cap
U)/\mu(U)$ and for a measure-isomorphism $F$ modulo $0$ from $\Omega_U\df(X,\cA,\mu_U)$ to $\Omega$,
we let $\cN\rest^F_U$ be the $T$-on over $\Omega_U$ defined by
\begin{align*}
  (\cN\rest^F_U)_P & \df \{x\in\cE_{k(P)}(\Omega_U) \mid x^F\in\cN_P\},
\end{align*}
where
\begin{align}\label{eq:xF}
  x^F_A & \df
  \begin{dcases*}
    x_A, & if $\lvert A\rvert = 1$;\\
    F(x_A), & if $\lvert A\rvert\geq 2$.
  \end{dcases*}
\end{align}

Under this definition, not every sub-object $\psi$ of $\phi_\cN$ is of the form
$\phi_{\cN\rest_U^F}$ for some choice of $(U,F)$ as above. However, just as in the graphon case,
this description is not far from correct, we only need to ``rescale'' the underlying measure by a
weight function.

\begin{lemma}\label{lem:limitsubobject}
  Let $\cN$ be a $T$-on over $\Omega=(X,\cA,\mu)$, let $c > 0$ and let $\psi\in\HomT{T}$. The
  following are equivalent.
  \begin{enumerate}
  \item There exist a convergent sequence $(N_n)_{n\in\NN}$ converging to $\phi_\cN$ and sets
    $U_n\subseteq V(N_n)$ with $\lim_{n\to\infty} \lvert U_n\rvert/\lvert N_n\rvert = c$ such that
    $(N_n\rest_{U_n})_{n\in\NN}$ converges to $\psi$, that is, $\psi$ is a measure $c$ limit
    sub-object of $\phi_\cN$.%
    \label{lem:limitsubobject:limitsubobject}
  \item There exists a measurable function $f\function{X}{[0,1]}$ with $\int_X f\ d\mu = c$ such
    that for the space $\Omega_f\df(X,\cA,\mu_f)$ defined by
    \begin{align}\label{eq:limitsubobject:muf}
      \mu_f(A) & \df \frac{\int_A f(x) \ d\mu(x)}{c},
    \end{align}
    there exists a measure-isomorphism $F$ modulo $0$ from $\Omega_f$ to $\Omega$ such that $\psi =
    \phi_{\cN\rest_f^F}$ for the $T$-on $\cN\rest_f^F$ over the space $\Omega_f$ defined by
    \begin{align}\label{eq:limitsubobject:cNrestfF}
      (\cN\rest_f^F)_P & \df \{x\in\cE_{k(P)}(\Omega_f) \mid x^F\in\cN_P\},
    \end{align}
    where $x^F\in\cE_{k(P)}(\Omega)$ is given by~\eqref{eq:xF}.%
    \label{lem:limitsubobject:rescaleexistsF}
  \item Item~\ref{lem:limitsubobject:rescaleexistsF} holds for every measure-isomorphism $F$ modulo
    $0$ from $\Omega_f$ to $\Omega$.%
    \label{lem:limitsubobject:rescaleforallF}
  \end{enumerate}
\end{lemma}

\begin{proof}
  The implication~\ref{lem:limitsubobject:rescaleforallF}$\implies$\ref{lem:limitsubobject:rescaleexistsF} is
  trivial.

  \medskip

  For the other implications, we will use the operator $\pi^{(U,I)}$ of the theory of flag
  algebras~\cite[Theorem~2.6]{Raz07}. Let $\widehat{T}$ be the theory obtained from $T$ by
  augmenting it with a unary predicate symbol $U$ and for each $n\in\NN$, let
  \begin{align*}
    \cM_n^U[\widehat{T}] & \df \{M\in\cM_n[\widehat{T}] \mid M\vDash\forall x, U(x)\}
  \end{align*}
  be the set of all models of $\widehat{T}$ of size $n$ in which all vertices satisfy $U$. Let
  $u\df\sum_{M\in\cM_1^U[\widehat{T}]} M$ and let $\cA_u[\widehat{T}]$ be the localization of
  $\cA[\widehat{T}]$ with respect to the multiplicative system $\{u^n \mid n\in\NN\}$. Finally, let
  $I\interpret{T}{\widehat{T}}$ be the structure-erasing interpretation that acts identically on
  $T$. By~\cite[Theorem~2.6]{Raz07}, the linear map
  $\pi^{(U,I)}\function{\cA[T]}{\cA_u[\widehat{T}]}$ given by
  \begin{align*}
    \pi^{(U,I)}(M) & \df \frac{M^U}{u^{\lvert M\rvert}},
  \end{align*}
  where $M^U\in\cM_{\lvert M\rvert}[\widehat{T}]$ is the model of $\widehat{T}$ obtained from $M$ by declaring
  all its vertices to satisfy $U$ is an $\RR$-algebra homomorphism. The intuition is that if
  $\psi\in\HomT{\widehat{T}}$ is such that $\psi(u) > 0$, then $\psi$ has a non-negligible fraction of
  ``vertices'' satisfying $U$ and the composition $\psi\comp\pi^{(U,I)}\in\HomT{T}$ is the limit object of $T$
  induced by the ``vertices'' of $\psi$ satisfying $U$ (hence the need for the localization).

  \medskip

  Let us prove the
  implication~\ref{lem:limitsubobject:limitsubobject}$\implies$\ref{lem:limitsubobject:rescaleforallF}.

  For each $n\in\NN$, let $\widehat{N}_n$ be the model of $\widehat{T}$ obtained from $N_n$ by
  declaring the predicate symbol $U$ to be true exactly in the set $U_n$ and by possibly passing to
  a subsequence, we may suppose that $(\widehat{N}_n)_{n\in\NN}$ converges to some homomorphism
  $\xi\in\HomT{\widehat{T}}$. Note that since $\lim_{n\to\infty} \lvert U_n\rvert/\lvert N_n\vert =
  c$, we have $\xi(u) = c$. Note further that since $(N_n)_{n\in\NN}$ and
  $(N_n\rest_{U_n})_{n\in\NN}$ converge to $\phi_\cN$ and $\psi$, respectively and $I(\widehat{N}_n)
  = N_n$, we must have $\phi_\cN = \xi^I$ and $\psi = \xi\comp\pi^{(U,I)}$. By
  Proposition~\ref{prop:interpretationlifting}, there exists a $\widehat{T}$-on $\widehat{\cN}$ over
  $\Omega\times\Omega$ such that $\phi_{\widehat{\cN}} = \xi$ and $I(\widehat{\cN})_P =
  \cN_P\times\cE_{k(P)}(\Omega)$ a.e., for every predicate symbol $P$ in the language of $T$. Since
  $\widehat{\cN}_U\subseteq\cE_1(\Omega)\times\cE_1(\Omega)\cong X\times X$, we can define the
  function $f\function{X}{[0,1]}$ by
  \begin{align}\label{eq:fx}
    f(x) & \df \mu(\{y\in X \mid (x,y)\in\widehat{\cN}_U\})
  \end{align}
  (defining it arbitrarily when the set above is not measurable) and Fubini's Theorem ensures that
  $f$ is measurable.

  Note also that
  \begin{align}\label{eq:intfc}
    c
    & =
    \xi(u)
    =
    \sum_{M\in\cM_1^U[\widehat{T}]} \tind(M,\widehat{\cN})
    =
    \mu(T(U,\widehat{\cN}))
    =
    \mu(\widehat{\cN}_U)
    =
    \int_X f\ d\mu.
  \end{align}

  Define $\Omega_f\df(X,\cA,\mu_f)$ with $\mu_f$ given by~\eqref{eq:limitsubobject:muf} and the
  $T$-on $\cN\rest_f^F$ by~\eqref{eq:limitsubobject:cNrestfF} for an arbitrary measure-isomorphism
  $F$ modulo $0$ from $\Omega_f$ to $\Omega$ and note that for every $M\in\cM_n[T]$, we have
  \begin{equation}\label{eq:xipiUI}
    \begin{aligned}
      \phi_{\cN\rest_f^F}(M)
      & =
      \frac{\lvert M\rvert!}{\lvert\Aut(M)\rvert}\cdot
      \tind(M,\cN\rest_f^F)
      =
      \frac{\lvert M^U\rvert!}{\lvert\Aut(M^U)\rvert}\cdot
      \frac{\tind(M^U,\widehat{\cN})}{\mu(\widehat{\cN}_U)^n}
      \\
      & =
      (\xi\comp\pi^{(U,I)})(M)
      =
      \psi(M),
    \end{aligned}
  \end{equation}
  so $\phi_{\cN\rest_f^F} = \psi$ as required.

  \medskip

  For the final
  implication~\ref{lem:limitsubobject:rescaleexistsF}$\implies$\ref{lem:limitsubobject:limitsubobject},
  we define the $\widehat{T}$-on $\widehat{\cN}$ over $\Omega\times\Omega$ from $\cN$ by letting
  $\widehat{\cN}_P\df\cN_P\times\cE_{k(P)}(\Omega)$ for every predicate symbol of $T$ and letting
  $\widehat{\cN}_U\subseteq\cE_1(\Omega)\times\cE_1(\Omega)$ be any measurable set such
  that~\eqref{eq:fx} holds. By also letting $\xi\df\phi_{\widehat{\cN}}$, we can deduce the
  equalities in~\eqref{eq:xipiUI} in a different order:
  \begin{align*}
    \psi(M)
    & =
    \phi_{\cN\rest_f^F}(M)
    =
    \frac{\lvert M\rvert!}{\lvert\Aut(M)\rvert}\cdot
    \tind(M,\cN\rest_f^F)
    \\
    & =
    \frac{\lvert M^U\rvert!}{\lvert\Aut(M^U)\rvert}\cdot
    \frac{\tind(M^U,\widehat{\cN})}{\mu(\widehat{\cN}_U)^n}
    =
    (\xi\comp\pi^{(U,I)})(M).
  \end{align*}
  Similarly, the equalities in~\eqref{eq:intfc} also hold deduced in a different order:
  \begin{align*}
    c
    & =
    \int_X f\ d\mu
    =
    \mu(\widehat{\cN}_U)
    =
    \mu(T(U,\widehat{\cN}))
    =
    \xi(u).
  \end{align*}

  Finally, we let $(\widehat{N}_n)_{n\in\NN}$ be a sequence of models of $\widehat{T}$ converging to
  $\widehat{\cN}$, let $N_n\df I(\widehat{N}_n)$ and $U_n\df U^{\widehat{N}_n}\df\{v\in
  V(\widehat{N}_n) \mid \widehat{N}_n\vDash U(v)\}$ and note that $\lim_{n\to\infty} \lvert
  U_n\rvert/\lvert N_n\rvert = \xi(u) = c$ and for every $M\in\cM[T]$, we have
  \begin{align*}
    \lim_{n\to\infty} p(M,N_n\rest_{U_n})
    & =
    \lim_{n\to\infty} p(M^U,\widehat{N}_n)
    \cdot\left(\frac{\lvert\widehat{N}_n\rvert}{\lvert U_n\rvert}\right)^{\lvert M\rvert}
    =
    \frac{\xi(M^U)}{\xi(u^{\lvert M\rvert})}
    \\
    & =
    (\xi\comp\pi^{(U,I)})(M)
    =
    \psi(M),
  \end{align*}
  concluding the proof.
\end{proof}

From this theorem, Lemma~\ref{lem:subgraphon} on subgraphons follows trivially.

\begin{proofof}{Lemma~\ref{lem:subgraphon}}
  Follows directly from Lemma~\ref{lem:limitsubobject} via the correspondence between $\TGraph$-ons
  and graphons of Remark~\ref{rmk:TGraphons}.
\end{proofof}

For general universal theories, the role of complete or empty graphons (i.e., $W$ constant equal to
$0$ or $1$) is played by trivial limits defined below.

\begin{definition}
  A limit $\phi\in\HomT{T}$ is called \emph{trivial} if there exists a $T$-on $\cN$ with $\phi_\cN =
  \phi$ and each $P$-on $\cN_P$ either has measure $0$ or $1$. Equivalently, a limit
  $\phi\in\HomT{T}$ is trivial if and only if it is of the form $\phi = \psi^I$ for some open
  interpretation $I\interpret{T}{T_0}$ and the unique $\psi\in\HomT{T_0}$, where $T_0$ is the
  trivial universal theory, that is, the theory over the empty language without any axioms.
\end{definition}

Before we can finally state the stability dichotomy theorem for limits of arbitrary universal
theories, we also need to define stability in this more general setting.

\begin{definition}\label{def:almoststabletheon}
  Recall that for a formula $F(\vec{x},\vec{y})$ with a particular partition of its free variables
  into two parts $\vec{x}$ and $\vec{y}$, a \emph{half-graph of order $n$ with respect to
    $F(\vec{x},\vec{y})$} in a structure $M$ is a pair of sequences $(\vec{a}_1,\ldots,\vec{a}_n)$
  and $(\vec{b}_1,\ldots,\vec{b}_n)$ of tuples of vertices of $M$ with $\lvert\vec{a}_i\rvert =
  \lvert\vec{x}\rvert$, $\lvert\vec{b}_i\rvert = \lvert\vec{y}\rvert$ and such that $M\vDash
  F(\vec{a}_i,\vec{b}_j)$ if and only if $i\leq j$. A \emph{tree of height $n$ with respect to
    $F(\vec{x},\vec{y})$} in a structure $M$ is a pair of sequences $(\vec{a}_\sigma \mid
  \sigma\in\{0,1\}^n)$ and $(\vec{b}_\tau \mid m\in\{0,1,\ldots,n-1\}, \tau\in\{0,1\}^m)$ such that
  $\lvert\vec{a}_\sigma\rvert=\lvert\vec{x}\rvert$, $\lvert\vec{b}_\tau\rvert = \lvert\vec{y}\rvert$
  and for every $\sigma = (\sigma_i)_{i=1}^n\in\{0,1\}^n$ and every $m < n$, $M\vDash
  F(x_\sigma,y_{\sigma\rest_{[m]}})$ if and only if $\sigma_{m+1} = 1$.

  We say that $F(\vec{x},\vec{y})$ is \emph{almost stable} in a limit $\phi\in\HomT{T}$ if there
  exists $n\in\NN$ such that every finite model $M$ of $T$ containing a half-graph of order $n$ with
  respect to $F(\vec{x},\vec{y})$ satisfies $\phi(M) = 0$. Equivalently, letting
  \begin{align}\label{eq:halfgraphformula}
    H_{n,F}(\vec{x}_1,\ldots,\vec{x}_n,\vec{y}_1,\ldots,\vec{y}_n)
    & \df
    \bigwedge_{1\leq i\leq j\leq n} F(\vec{x}_i,\vec{y}_j)\land
    \bigwedge_{1\leq j < i\leq n} \neg F(\vec{x}_i,\vec{y}_j)
  \end{align}
  be the formula encoding a half-graph of order $n$ with respect to $F(\vec{x},\vec{y})$, the
  formula $F(\vec{x},\vec{y})$ is almost stable in $\phi\in\HomT{T}$ if there exists $n\in\NN$ such
  that for every (not-necessarily injective) substitution $H$ of the variables of the formula
  $H_{n,F}$, the set $T(H,\cN)$ has measure $0$ for some (equivalently, every) $T$-on $\cN$ such
  that $\phi=\phi_\cN$.

  It will also be convenient to define a weak version of almost stability: we say that
  $F(\vec{x},\vec{y})$ is \emph{almost weakly stable} in $\phi\in\HomT{T}$ if there exists $n\in\NN$
  such that $T(H_{n,F},\cN)$ has measure $0$ for some (equivalently, every) $T$-on $\cN$ such that
  $\phi=\phi_\cN$. Thus the difference between stability and weak stability is whether the tuples of
  the half-graph are allowed to repeat vertices or not.
\end{definition}

Our stability dichotomy theorem for limits of universal theories will be particularly concerned with
the case when $F(\vec{x},\vec{y})$ is a \emph{$(1,k(P)-1)$-split} of a predicate symbol $P$ (whose
arity $k(P)$ is at least $2$), that is, we have
\begin{align*}
  F(x,y_1,\ldots,y_{k(P)-1}) & \df P(y_1,\ldots,y_{i-1},x,y_i,y_{i+1}\ldots,y_{k(P)-1})
\end{align*}
for some $i\in[k(P)]$.

\begin{theorem}\label{thm:stabletheon}
  The following are equivalent for a $T$-on $\cN$ over a space $\Omega=(X,\cA,\mu)$.
  \begin{enumerate}
  \item $\phi_\cN$ contains a trivial sub-object $\psi$.%
    \label{thm:stabletheon:trivialsubobject}
  \item There exists a positive measure $U\subseteq X$ such that for every measure-isomorphism $F$
    modulo $0$ from $\Omega_U$ to $\Omega$, the sub-object $\phi_{\cN\rest_U^F}$ is trivial.%
    \label{thm:stabletheon:trivialrestriction}
  \item $\phi_\cN$ contains a sub-object $\psi$ in which every $(1,k(P)-1)$-split of every predicate
    symbol $P$ is almost weakly stable.%
    \label{thm:stabletheon:stablesubobject}
  \end{enumerate}
\end{theorem}

The same observations of Discussion~\ref{dsc:subgraphon} can be made here: the equivalence between
items~\ref{thm:stabletheon:trivialsubobject} and~\ref{thm:stabletheon:trivialrestriction} is not
immediate since not every sub-object is of the form that appears in the latter item. However, by an
argument analogous to that in Discussion~\ref{dsc:subgraphon}, if $P$ is a property of limits that
is closed under sub-objects, then $\phi_\cN$ has a sub-object satisfying $P$ if and only if there
exists $U\subseteq X$ such that $\phi_{\cN\rest_U^F}$ satisfies $P$ for every measure-isomorphism
$F$ modulo $0$ from $\Omega_U$ to $\Omega$.

Naturally, the main ingredient to prove the theorem above is a generalization of
Lemma~\ref{lem:stablegraphon} for theons.

\begin{lemma}\label{lem:stabletheon}
  Let $\cN$ be a $T$-on over a space $\Omega=(X,\cA,\mu)$ such that every $(1,k(P)-1)$-split of
  every predicate symbol $P$ is almost weakly stable in $\phi_\cN$. Then there exists a positive
  measure set $U\subseteq X$ such that for every measure-isomorphism $F$ modulo $0$ from $\Omega_U$
  to $\Omega$, the sub-object $\phi_{\cN\rest_U^F}$ is trivial.
\end{lemma}

\begin{proof}
  In this proof, we will work with measurability with respect to the $\sigma$-algebra corresponding
  to the completion of the measure $\mu$. Note that the result still follows for the original
  $\sigma$-algebra by simply changing the final set $U$ in a zero-measure set. Note also that it is
  enough to show the existence of some $U$ and $F$ such that $\phi_{\cN\rest_U^F}$ is trivial as if
  $F'$ is any other measure-isomorphism modulo $0$ from $\Omega_U$ to $\Omega$, then
  $\phi_{\cN\rest_U^{F'}}$ is also trivial.
  
  Let $\cL$ be the language of $T$. The proof is by induction in the sum $\sum_{P\in\cL} k(P)$ of
  the arities of the predicate symbols.

  If the language $\cL$ is empty, the result is trivial.

  Suppose then that $\cL$ is non-empty, let $P\in\cL$ and let $k\df k(P)$ be its arity.

  If $P$ is a unary predicate, then applying the result inductively for $\cL\setminus\{P\}$, we get a
  measurable set $U'\subseteq X$ with $\mu(U') > 0$ such that for every $Q\in\cL\setminus\{P\}$, we have
  $\mu_{U'}((\cN\rest_{U'}^F)_Q)\in\{0,1\}$. Since $\mu(U')>0$, at least one of $\cN_P\cap U'$ or
  $U'\setminus\cN_P$ has positive $\mu$-measure, so letting $U$ be any of these having positive measure gives
  the desired result.

  Suppose now that $k\geq 2$. By Theorem~\ref{thm:theonremoval} and Remark~\ref{rmk:theonremoval}, we can
  replace $\cN$ with a $T$-on such that there exists $n\in\NN$ such that every $(1,k-1)$-split $S(x,\vec{y})$
  of $P$ satisfies $T(H_{n,S},\cN)\subseteq\cD_{nk}(\Omega)$ (here we are abusing the notation a bit by saying
  that the variables of $H_{n,S}$ are indexed by $[nk]$).

  Let us consider the natural $(1,k-1)$-split of $P$ given by $P(x,\vec{y})\df P(x,y_1,\ldots,y_{k-1})$, which
  we will denote simply by $P$ and for convenience of notation, let $V_k\df[k]\setminus\{1\}$.

  For each $a\in\cE_1(\Omega)\cong X$ and
  $b\in\cE_{V_k}(\Omega)\setminus\cD_{V_k}(\Omega)$, let
  \begin{align*}
    \cN_P(a,b) & \df \{c\in X^{r(k)\setminus (r(1)\cup r(V_k))} \mid (a,b,c)\in\cN_P\setminus\cD_k(\Omega)\}
  \end{align*}
  be the set of points that complete $(a,b)$ to a point of $\cN_P\setminus\cD_k(\Omega)$ (note that
  if $a\in\{b_{\{1\}},\ldots,b_{\{k-1\}}\}$, then $\cN_P(a,b)$ is immediately empty). Define further
  \begin{align*}
    N^1_P(b) & \df \{a\in\cE_1(\Omega) \mid \mu(\cN_P(a,b)) = 1\},\\
    N^0_P(b) & \df \{a\in\cE_1(\Omega) \mid \mu(\cN_P(a,b)) = 0\}
  \end{align*}
  and let $B_0$ be the set of all $b\in\cE_{V_k}(\Omega)\setminus\cD_{V_k}(\Omega)$ such that both
  $N^1_P(b)$ and $N^0_P(b)$ are measurable. Fubini's Theorem gives $\mu(B_0)=1$.

  Given a finite collection of points $\{b_1,\ldots,b_t\}\subseteq B_0$, let
  \begin{align*}
    C(\{b_1,\ldots,b_t\})
    & \df
    \{(b_i)_{\{v\}} \mid i\in[t], v\in V_k\}
    \subseteq
    X
  \end{align*}
  be the set of coordinates of the $b_i$ that are indexed by singletons and let
  \begin{align*}
    B(\{b_1,\ldots,b_t\})
    & \df
    \{b\in B_0 \mid \forall v\in V_k, b_{\{v\}}\notin C(\{b_1,\ldots,b_t\})\}
    \subseteq
    \cE_{V_k}(\Omega)
  \end{align*}
  be the set of $b\in B_0$ whose coordinates indexed by a singletons do not appear in the set
  $C(\{b_1,\ldots,b_t\})$. We also let $B(\varnothing)\df B_0$. Note that
  $\mu(B(\{b_1,\ldots,b_t\})) = \mu(B_0) = 1$.

  We now construct sequences $(A_\sigma)_\sigma$ and $(b_\sigma)_\sigma$ indexed by finite strings
  over $\{0,1\}$ inductively in the length-lexicographic order $\leqLL$ as follows.
  \begin{enumerate}[label={\arabic*.}]
  \item Set $A_\varnothing\df X$.
  \item For a string $\sigma$, given $A_\sigma$ and $b_\tau$ for all $\tau\lessLL\sigma$, if there
    exist $j\in\{0,1\}$ and $b\in B(\{b_\tau \mid \tau\lessLL\sigma\})$ such that $0 <
    \mu(N^j_P(b)\cap A_\sigma) < \mu(A_\sigma)$, then set
    \begin{align*}
      b_\sigma & \df b,\\
      A_{\sigma j} & \df (A_\sigma\cap N^j_P(b))\setminus C(\{b_\tau \mid \tau\leqLL\sigma\}),\\
      A_{\sigma (1-j)} & \df A_\sigma\setminus (N^j_P(b)\cup C(\{b_\tau \mid \tau\leqLL\sigma\}));
    \end{align*}
    otherwise, stop the construction.
  \end{enumerate}

  By induction in the construction, it follows that if $A_\sigma$ is defined for every
  $\sigma\in\{0,1\}^t$ of a fixed length $t$, then $\{A_\sigma \mid \sigma\in\{0,1\}^t\}$ is a
  collection of pairwise disjoint sets of positive measure whose union has measure $1$ (hence each
  of these sets is non-empty). Furthermore, if $a_\sigma\in A_\sigma$ ($\sigma\in\{0,1\}^t$), then
  we can find a tree of height $t$ in $\cN$ as follows. Let
  \begin{align*}
    U & \df \{0,1\}^t\cup\bigcup_{m=0}^{t-1} (\{0,1\}^m\times V_k)
  \end{align*}
  and define $x\in\cE_U(\Omega)$ as follows.
  \begin{enumerate}[label={\alph*.}, ref={(\alph*)}]
  \item For each $\sigma\in\{0,1\}^t$, let $x_{\{\sigma\}}\df a_\sigma$.%
    \label{it:adef}
  \item For each $m\in\{0,1,\ldots,t-1\}$, each $\tau\in\{0,1\}^m$ and each $V\in r(V_k)$, let
    \begin{align*}
      x_{\{(\tau,v) \mid v\in V\}} & \df (b_\tau)_V,
    \end{align*}
    that is, for the injection $\alpha_\tau\injection{V_k}{U}$ given by $\alpha_\tau(v)\df(\tau,v)$,
    we have $\alpha_\tau^*(x) = b_\tau$.%
    \label{it:bdef}
  \item For each $\sigma\in\{0,1\}^t$ and each $m\in\{0,1,\ldots,t-1\}$, let
    \begin{align*}
      c_{\sigma,m} & \in
      \begin{dcases*}
        \cN_P(a_\sigma,b_{\sigma\rest_{[m]}}), & if $\sigma_{m+1} = 1$,\\
        X^{r(k)\setminus (r(1)\cup r(V_k))}\setminus\cN_P(a_\sigma,b_{\sigma\rest_{[m]}}),
        & if $\sigma_{m+1} = 0$,
      \end{dcases*}
    \end{align*}
    and for each $V\in r(V_k)$, let
    \begin{align*}
      x_{\{\sigma\}\cup\{(\sigma\rest_{[m]},v) \mid v\in V\}} & \df (c_{\sigma,m})_{\{1\}\cup V}.
    \end{align*}
    \label{it:cdef}
  \item Define all other coordinates of $x$ arbitrarily.
  \end{enumerate}

  Let us make some observations about this construction. First, it is straightforward to check that
  no coordinate of $x$ is defined more than once. Second, by induction in the construction, it
  follows that each $A_\sigma\subseteq A_{\sigma\rest_{[m+1]}}$; this means that the element
  $c_{\sigma,m}$ of item~\ref{it:cdef} is guaranteed to exist from the definition of
  $A_{\sigma\rest_{[m+1]}}$. Third, all coordinates of $x$ that are indexed by singletons are
  defined in items~\ref{it:adef} and~\ref{it:bdef} and since these coordinates must be either
  $a_\sigma$ for different $\sigma\in\{0,1\}^t$ or coordinates indexed by singletons of some
  $b_\tau$ for some $\tau\in\{0,1\}^m$ with $m\in\{0,1,\ldots,t-1\}$, it follows by construction
  that they must all be distinct, that is, we must have $x\notin\cD_U(\Omega)$. Finally, if
  $G_t(\vec{x},\vec{y})$ is the formula encoding a tree of height $t$ with respect to
  $P(x,\vec{y})$, that is, we have
  \begin{align*}
    G_t(\vec{x},\vec{y})
    & \df
    \bigwedge_{\sigma\in\{0,1\}^t} \bigwedge_{m=0}^{t-1} \neg^{1-\sigma_{m+1}} P(x_\sigma,\vec{y}_{\sigma\rest_{[m]}}),
  \end{align*}
  then $x\in T(G,\cN)$ after an appropriate \emph{bijective} relabeling of variables.

  From~\cite[Lemma~6.7.9]{Hod93}, we know that if a model has a tree of height $2^{n+2}-2$ with
  respect to $P(x,\vec{y})$, then it must have a half-graph of order $n$, where the $\vec{a}$ and
  $\vec{b}$ parts of the half-graph are picked from the $\vec{a}$ and $\vec{b}$ parts of the tree,
  respectively, with all of them distinct. In particular, if $t\geq 2^{n+2}-2$, then there exists an
  injection $\alpha\injection{[nk]}{U}$ such that $\alpha^*(x)\in T(H_{n,P},\cN)$.

  Since $x\notin\cD_U(\Omega)$ and $T(H_{n,P},\cN)\subseteq\cD_{nk}(\Omega)$, the construction must
  stop before constructing all $A_\sigma$ with $\lvert\sigma\rvert = 2^{n+1}-2$.

  Let then $\widetilde{\sigma}$ be the last string considered by the construction and let $A\df
  A_{\widetilde{\sigma}}$ and $B\df B(\{b_\tau \mid \tau\lessLL\widetilde{\sigma}\})$. We know that
  for every $b\in B$, we have
  \begin{align*}
    \frac{\mu(N^1_P(b)\cap A)}{\mu(A)}
    & =
    1 - \frac{\mu(N^0_P(b)\cap A)}{\mu(A)}
    \in\{0,1\}.
  \end{align*}

  We now setup our induction: let $\cL'\df(\cL\setminus\{P\})\cup\{P'\}$, where $P'$ is a new
  predicate symbol of arity $k(P')\df k-1$, let $\widetilde{F}$ be a measure-isomorphism modulo $0$
  from $\Omega_A$ to $\Omega$ and define the $T_{\cL'}$-on $\cN'$ over $\Omega_A$ by letting
  $\cN'_Q\df(\cN\rest_A^{\widetilde{F}})_Q$ for every $Q\in\cL'\setminus\{P'\}$ and letting
  \begin{align*}
    \cN'_{P'}
    & \df
    \left\{\iota^*(x) \;\middle\vert\;
    x^{\widetilde{F}}\in B\land
    \frac{\mu(N^1_P(x^{\widetilde{F}})\cap A)}{\mu(A)}=1
    \right\},
  \end{align*}
  where $\iota\injection{[k-1]}{V_k}$ is the relabeling $\iota(v) = v+1$ and $x^{\widetilde{F}}$ is
  given by~\eqref{eq:xF}.

  We claim that every $(1,k(Q)-1)$-split of a predicate symbol $Q\in\cL'$ is almost weakly stable in
  $\phi_{\cN'}$. For $Q\neq P'$ this obviously follows from the same property for $\phi_\cN$ as
  $\mu_A$ is absolutely continuous with respect to $\mu$. For $P'$, if $S'(x,\vec{y})\df
  P'(y_1,\ldots,y_{i_0-1},x,y_{i_0},\ldots,y_{k-2})$ ($i_0\in[k-1]$) is a $(1,k-2)$-split of $P'$,
  then we want to show that $T(H_{n,S'},\cN')$ has measure zero. To prove this, let $S(x,\vec{y})$
  be the $(1,k-1)$-split of $P$ given by $P(y_1,\ldots,y_{i_0},x,y_{i_0+1},\ldots,y_{k-1})$ and let
  us change the indexing of the variables of the formulas $H_{n,S'}$ and $H_{n,S}$ as follows:
  \begin{align*}
    H_{n,S'}(x_{i,j} \mid i\in V_k, j\in[n])
    & \df
    \begin{multlined}[t]
      \bigwedge_{1\leq j_1\leq j_2\leq n}
      P'(x_{\beta_{j_1,j_2}(1)},\ldots,x_{\beta_{j_1,j_2}(k-1)})
      \\
      \land
      \bigwedge_{1\leq j_2 < j_1\leq n}
      \neg P'(x_{\beta_{j_1,j_2}(1)},\ldots,x_{\beta_{j_1,j_2}(k-1)}),
    \end{multlined}
    \\
    H_{n,S}(x_{i,j} \mid i\in [k], j\in[n])
    & \df
    \begin{multlined}[t]
      \bigwedge_{1\leq j_1\leq j_2\leq n}
      P(x_{\gamma_{j_1,j_2}(1)},\ldots,x_{\gamma_{j_1,j_2}(k)})
      \\
      \land
      \bigwedge_{1\leq j_2 < j_1\leq n}
      \neg P(x_{\gamma_{j_1,j_2}(1)},\ldots,x_{\gamma_{j_1,j_2}(k)}),
    \end{multlined}
  \end{align*}
  where the injections $\beta_{j_1,j_2}\injection{[k-1]}{V_k\times[n]}$ and
  $\gamma_{j_1,j_2}\injection{[k]}{[k]\times[n]}$ are given by
  \begin{align*}
    \beta_{j_1,j_2}(v) & \df
    \begin{dcases*}
      (v+1,j_1), & if $v\neq i_0$,\\
      (v+1,j_2), & if $v = i_0$,
    \end{dcases*}
    &
    \gamma_{j_1,j_2}(v) & \df
    \begin{dcases*}
      (v,j_1), & if $v\neq i_0+1$,\\
      (v,j_2), & if $v = i_0+1$.
    \end{dcases*}
  \end{align*}

  Suppose now that $z\in T(H_{n,S'},\cN')\setminus\cD_{V_k\times [n]}(\Omega_A)$ is such that all of
  its coordinates indexed by singletons are in $A$ and for every injection
  $\alpha\injection{V_k}{V_k\times [n]}$, we have $\alpha^*(z)\in B$. Then we can define a point
  $\widehat{z}\in T(H_{n,S},\cN)\setminus\cD_{[k]\times [n]}(\Omega)$ as follows.
  \begin{enumerate}[label={\Alph*.}, ref={(\Alph*)}]
  \item For each $i\in V_k$ and each $j\in[n]$, let $\widehat{z}_{\{(i,j)\}}\df z_{\{(i,j)\}}$.
  \item For each $V\subseteq V_k\times [n]$ with $\lvert V\rvert\geq 2$, let $\widehat{z}_V\df
    \widetilde{F}(z_V)$.
  \item For each $j_1,j_2\in[n]$ with $j_1\leq j_2$, since $\beta_{j_1,j_2}^*(z)\in\cN'_{P'}$, the
    definition of $\cN'_{P'}$ implies that the point
    $w\df(\beta_{j_1,j_2}\comp\iota^{-1})^*(z)^{\widetilde{F}}\in\cE_{V_k}(\Omega)$ satisfies
    $\mu(N^1_P(w)\cap A) = \mu(A) > 0$, so we can let $\widehat{z}_{\{(1,j_1)\}}\in N^1_P(w)\cap A$
    be different from all coordinates defined so far and define the coordinates $\widehat{z}_V$ with
    $\{(1,j_1)\}\subsetneq V\subseteq\im(\gamma_{j_1,j_2})$ based on a point in
    $\cN_P(\widehat{z}_{\{(1,j_1)\}},w)$ so that $\gamma_{j_1,j_2}^*(\widehat{z})\in\cN_P$.%
    \label{it:P'true}
  \item Analogously, for each $j_1,j_2\in [n]$ with $j_2 < j_1$, since
    $\beta_{j_1,j_2}^*(z)\notin\cN'_{P'}$, the definition of $\cN'_{P'}$ implies that the point
    $w\df(\beta_{j_1,j_2}\comp\iota^{-1})^*(z)^{\widetilde{F}}\in\cE_{V_k}(\Omega)$ satisfies
    \begin{align*}
      \mu(N^0_P(w)\cap A)
      =
      \mu(A) - \mu(N^1_P(w)\cap A)
      =
      \mu(A)
      >
      0,
    \end{align*}
    so we can let $\widehat{z}_{\{(1,j_1)\}}\in N^0_P(w)\cap A$ be different from all coordinates
    defined so far and define the coordinates $\widetilde{z}_V$ with $\{(1,j_1)\}\subsetneq
    V\subseteq\im(\gamma_{j_1,j_2})$ based on a point in the complement of
    $\cN_P(\widehat{z}_{\{(1,j_1)\}},w)$ so that $\gamma_{j_1,j_2}^*(\widehat{z})\notin\cN_P$.%
    \label{it:P'false}
  \item Finally, we define all other coordinates of $\widehat{z}$ arbitrarily.
  \end{enumerate}

  Since in items~\ref{it:P'true} and~\ref{it:P'false} we ensured that coordinates were not repeated,
  we get $\widetilde{z}\in T(H_{n,S},\cN)\setminus\cD_{[k]\times[n]}(\Omega)$, a contradiction. Thus
  $T(H_{n,S'},\cN')$ has measure zero.

  Therefore, the $(1,k-2)$-split $S'$ of $P'$ is almost weakly stable.

  \medskip

  By inductive hypothesis, it follows that there exists a measurable set $U\subseteq A$
  with $\mu_A(U) > 0$ and a measure-isomorphism $F'$ modulo $0$ from $\Omega_U$ to
  $\Omega_A$ such that $\phi_{\cN'\rest_U^{F'}}$ is trivial. It follows from the
  definition of $\cN'$ that $\phi_{\cN\rest_U^{\widetilde{F}\comp F'}}$ is trivial, completing the proof.
\end{proof}

The proof of Theorem~\ref{thm:stabletheon} from Lemma~\ref{lem:stabletheon} below is analogous to
its graphon counterpart Theorem~\ref{thm:stablegraphon} from Lemma~\ref{lem:stablegraphon}.

\begin{proofof}{Theorem~\ref{thm:stabletheon}}
  The implication~\ref{thm:stabletheon:trivialrestriction}$\implies$\ref{thm:stabletheon:trivialsubobject} is
  trivial and the
  implication~\ref{thm:stabletheon:trivialsubobject}$\implies$\ref{thm:stabletheon:stablesubobject} follows
  since (all splits of) all open formulas are almost stable in a trivial sub-object.

  \medskip

  For the final
  implication~\ref{thm:stabletheon:stablesubobject}$\implies$\ref{thm:stabletheon:trivialrestriction},
  again it is enough to show the existence of some $U$ and $F$ such that $\phi_{\cN\rest_U^F}$ is
  trivial as if $F'$ is any other measure-isomorphism modulo $0$ from $\Omega_U$ to $\Omega$, then
  $\phi_{\cN\rest_U^{F'}}$ is also trivial.

  Let then $\psi$ be a sub-object in which every $(1,k(P)-1)$-split of every predicate symbol is
  almost weakly stable. By Lemma~\ref{lem:limitsubobject}, there exists $f\function{X}{[0,1]}$ with
  $\int_X f\ d\mu > 0$ such that for any measure-isomorphism $F$ module $0$ from $\Omega_f$ to
  $\Omega$, we have $\psi = \phi_{\cN\rest_f^F}$.

  Let $V\df\{x\in X \mid f(x) > 0\}$, let $\widetilde{F}$ be any measure-isomorphism modulo $0$ from
  $\Omega_V$ to $\Omega$ and consider the sub-object
  $\widetilde{\psi}\df\phi_{\cN\rest_V^{\widetilde{F}}}$ of $\phi_\cN$.

  The same measure theoretic trick of Theorem~\ref{thm:stablegraphon} gives that every $(1,k(P)-1)$
  split of every predicate symbol $P$ is almost weakly stable in $\widetilde{\psi}$: given one such
  split $S(\vec{x},y)$, since it is almost weakly stable in $\psi$, we know that there exists
  $n\in\NN$ such that $\mu_f(T(H_{n,S},\cN\rest_f^F)) = 0$ for the half-graph formula $H_{n,S}$
  of~\eqref{eq:halfgraphformula}. For each $\epsilon > 0$, let
  \begin{align*}
    T^\epsilon(H_{n,S},\cN\rest_f^F)
    & \df
    \{x\in\Tind(H_{n,S},\cN\rest_f^F) \mid \forall v, f(x_{\{v\}}) > \epsilon\},
    \\
    T^\epsilon(H_{n,S},\cN\rest_V^{\widetilde{F}})
    & \df
    \{x\in\Tind(H_{n,S},\cN\rest_V^{\widetilde{F}}) \mid \forall v, f(x_{\{v\}}) > \epsilon\}.
  \end{align*}
  Then it follows that
  \begin{align*}
    \mu_f(T^\epsilon(H_{n,S},\cN\rest_f^F))
    & \geq
    \left(\epsilon\cdot\frac{\mu(V)}{\int_X f\ d\mu}\right)^{n\cdot k(P)}
    \mu_V(T^\epsilon(H_{n,S},\cN\rest_V^{\widetilde{F}}))
  \end{align*}
  and since we have $T(H_{n,S},\cN\rest_f^F) = \bigcup_{m\in\NN_+} T^{1/m}(H_{n,S},\cN\rest_f^F)$
  and $T(H_{n,S},\cN\rest_V^{\widetilde{F}}) = \bigcup_{m\in\NN_+}
  T^{1/m}(H_{n,S},\cN\rest_V^{\widetilde{F}})$, it follows that
  $\mu_V(T(H_{n,S},\cN\rest_V^{\widetilde{F}})) = 0$, so $S$ is almost weakly stable in
  $\widetilde{\psi}$.

  By Lemma~\ref{lem:stabletheon}, there exists a measurable set $U'\subseteq X$ such that $\mu_V(U')
  > 0$ and the sub-object $\phi_\cH$ is trivial, where $\cH\df
  (\cN\rest_V^{\widetilde{F}})\rest_{U'}^{F'}$ for any given measure-isomorphism $F'$ modulo $0$
  from $\Omega_{U'}$ to $\Omega_V$. The result now follows by letting $U\df U'\cap V$ and using the
  measure-isomorphism $\widetilde{F}\comp F'$ modulo $0$ from $\Omega_U$ to $\Omega$.
\end{proofof}

Let us now revisit Example~\ref{ex:quasirandompermuton}.

\begin{example}\label{ex:quasirandompermuton2}
  An alternative way of constructing the $\{0,1\}$-valued graphon of
  Example~\ref{ex:quasirandompermuton} that does not have any linear-sized almost clique or almost
  anti-clique is as follows.

  Given a permutation $\sigma\in S_n$ and a set $U\subseteq [n]$, the \emph{subpermutation} induced
  by $U$ is the unique permutation $\tau\in S_{\lvert U\rvert}$ such that for every $i,j\in[\lvert
    U\rvert]$, we have $\tau(i) < \tau(j) \iff \tau(\iota_U(i)) < \tau(\iota_U(j))$, where
  $\iota_U\function{[\lvert U\rvert]}{[n]}$ is the unique increasing function with $\im(\iota_U) =
  U$; equivalently, we have $\tau = \iota_{\sigma(U)}^{-1}\comp\sigma\comp\iota_U$.

  Consider now the theory $\TPerm\df\TLinOrder\cup\TLinOrder$, where $\TLinOrder$ is the theory of
  (strict) linear orders, that is, $\TPerm$ is the theory of two linear orders on the same base
  set. There is a natural correspondence between $S_n$ and $\cM_n[\TPerm]$ in which $\sigma\in S_n$
  corresponds to the model $M_\sigma\in\cM_n[\TPerm]$, in which the first order $\prec_1$ is the
  natural order on $[n]$ and the second order $\prec_2$ is given by $i\prec_2 j \iff
  \sigma^{-1}(i)\prec_2\sigma^{-1}(j)$. Furthermore, under this correspondence, subpermutations
  correspond to submodels (up to isomorphism).

  It is straightforward to check that if $\rn{\sigma_n}$ is distributed uniformly at random in
  $S_n$, then with probability $1$ the sequence $(\rn{\sigma_n})_{n\in\NN}$ is convergent (as models
  of $\TPerm$) and it converges to the $\TPerm$-on $\cH$ over $[0,1]^2$ given by
  \begin{align*}
    \cH_{\prec_i} & \df \{x\in\cE_2([0,1]^2) \mid \pi_i(x_{\{1\}}) < \pi_i(x_{\{2\}})\} \qquad (i\in[2]),
  \end{align*}
  where $\pi_i\function{[0,1]^2}{[0,1]}$ is the projection onto the $i$th coordinate. The limit
  $\phi_\cH$ is called the \emph{quasirandom permuton} and it is easy to see that $\phi_\cH(\sigma)
  = 1/\lvert\sigma\rvert!$ for every permutation $\sigma$.

  Consider then the open interpretation $I\interpret{\TGraph}{\TPerm}$ corresponding to the
  construction of the \emph{graph of agreements of a permutation} given by
  \begin{align*}
    I(E)(x_1,x_2) & \df x_1\neq x_2\land (x_1\prec_1 x_2 \tot x_1\prec_2 x_2).
  \end{align*}
  The interpreted $\TGraph$-on $I(\cH)$ over $[0,1]^2$ is then given by
  \begin{align*}
    I(\cH)_E
    & =
    \{(x,y)\in\cE_2\times\cE_2 \mid x_{\{1\}} < x_{\{2\}} \tot y_{\{1\}} < y_{\{2\}}\},
  \end{align*}
  which means that $\phi_\cH^I = \phi_{I(\cH)}$ is precisely the limit $\phi_W$ encoded by the graphon of
  Example~\ref{ex:quasirandompermuton}.

  Since a clique (anti-clique, resp.) in a graph $G$ of agreements of a permutation $\sigma$
  corresponds to an increasing (decreasing, resp.) sequence in $\sigma$, we have
  \begin{align*}
    \phi_\cH^I(\overline{K_n}) = \phi_\cH^I(K_n) & = \frac{1}{n!}.
  \end{align*}
\end{example}

\section{Consequences for finite models}
\label{sec:consmodels}

In this section, we transfer Theorem~\ref{thm:stabletheon} to the finite world just as we did in
Section~\ref{sec:consgraphs} for the theory of graphs. To do so, we will use the generalization of the
ultraproduct method of Elek--Szegedy~\cite{ES12} by Aroskar--Cummings~\cite{AC14} below.

\begin{theorem}[Elek--Szegedy~\protect{\cite{ES12}}, Aroskar--Cummings~\protect{\cite{AC14}}]\label{thm:theonultraproduct}
  Let $T$ be a canonical universal theory in a finite relational language $\cL$, let
  $(N_n)_{n\in\NN}$ be a convergent sequence of models of $T$, let $\cD$ be a non-principal
  ultrafilter over $\NN$ and let $k\in\NN_+$ be such that $k(P)\leq k$ for every $P\in\cL$.

  Then there exists a separable realization $\Theta\function{\prod_{n\in\NN} V(N_n)^k/\cD}{\cE_k}$
  of order $k$ and measurable sets $\cN_P\subseteq\cE_{k(P)}$ for each $P\in\cL$ such that
  \begin{align*}
    \mu^{k(P)}\left(\Theta^{-1}_{k(P)}(\cN_P)\symdiff\prod_{n\in\NN} P^{N_n}/\cD\right) & = 0
  \end{align*}
  for every $P\in\cL$, every restriction $\Theta_{k(P)}$ of $\Theta$ of order $k(P)$ and where
  $\mu^{k(P)}$ is the Loeb measure on $\prod_{n\in\NN} V(N_n)^{k(P)}/\cD$.
\end{theorem}

In plain English, the theorem above says that the ultraproduct construction is the pre-image of the
$T$-on $\cN$ under the separable realization $\Theta$, except for a zero-measure error. The
properties of restrictions and liftings of separable realizations then imply that for an open
formula $F(x_1,\ldots,x_m)$, we have
\begin{align*}
  \mu^m\left(F\left(\prod_{n\in\NN} N_n / \cD\right)\right)
  & =
  \mu^m(\Theta_m^{-1}(T(F,\cN)))
  =
  \lambda(T(F,\cN)),
\end{align*}
and thus $(N_n)_{n\in\NN}$ converges to $\phi_\cN$.

\medskip

Just as in the graph case, we can use this connection to pull back the set yielding a trivial sub-object in
the theon world through the separable realization and produce a linear-sized ``almost trivial'' submodel in
the convergent sequence. For this we make the following natural definitions.

\begin{definition}
  An increasing sequence $(N_n)_{n\in\NN}$ of structures in a language $\cL$ is \emph{almost
    trivial} if for every $P\in\cL$, we have $\lim_{n\to\infty} \lvert P^{N_n}\rvert / \lvert
  N_n\rvert^{k(P)} \in \{0,1\}$, i.e., either all but $o(\lvert N_n\rvert^{k(P)})$ amount of
  $k(P)$-tuples satisfy $P$ or at most an $o(\lvert N_n\rvert^{k(P)})$ amount of $k(P)$-tuples
  satisfy $P$.

  We say that $F(\vec{x},\vec{y})$ is \emph{almost stable} in $(N_n)_{n\in\NN}$ if there exists
  $m\in\NN$ such that $\lim_{n\to\infty} \lvert H_{m,F}(N_n)\rvert / \lvert
  N_n\rvert^{\lvert\vec{x}\rvert + \lvert\vec{y}\rvert} = 0$, i.e., only an $o(\lvert
  N_n\rvert^{\lvert\vec{x}\rvert + \lvert\vec{y}\rvert})$ amount of tuples satisfy the formula
  $H_{n,F}$.
\end{definition}

Note that the notion of almost stability for convergent sequences corresponds to almost \emph{weak}
stability in the limit; this is because solutions of $H_{m,F}$ that repeat variables can only
account for at most $O(\lvert N_n\rvert^{\lvert\vec{x}\rvert + \lvert\vec{y}\rvert - 1})$ tuples.

\begin{theorem}\label{thm:stableconvsequniversal}
  The following are equivalent for a convergent sequence $(N_n)_{n\in\NN}$ of models of a universal
  theory $T$ in a finite relational language $\cL$.
  \begin{enumerate}
  \item There exist $c > 0$ and sets $U_n\subseteq V(N_n)$ such that $\lvert U_n\rvert\geq c\lvert
    N_n\rvert$ for every $n\in\NN$ and $(N_n\rest_{U_n})_{n\in\NN}$ is almost trivial.%
    \label{thm:stableconvsequniversal:trivial}
  \item There exist a subsequence $(N_{n_\ell})_{\ell\in\NN}$ of $(N_n)_{n\in\NN}$ and sets
    $U_{n_\ell}\subseteq V(N_{n_\ell})$ such that $\limsup_{\ell\to\infty} \lvert
    U_{n_\ell}\rvert/\lvert N_{n_\ell}\rvert > 0$ and every $(1,k(P)-1)$-split of every predicate
    symbol $P\in\cL$ is almost stable in $(N_{n_\ell}\rest_{U_{n_\ell}})_{\ell\in\NN}$.%
    \label{thm:stableconvsequniversal:stable}
  \end{enumerate}
\end{theorem}

\begin{proof}
  The
  implication~\ref{thm:stableconvsequniversal:trivial}$\implies$\ref{thm:stableconvsequniversal:stable}
  follows since (all splits of) all open formulas are almost stable in an almost trivial sequence.

  \medskip

  For the
  implication~\ref{thm:stableconvsequniversal:stable}$\implies$\ref{thm:stableconvsequniversal:trivial},
  let $\phi\in\HomT{T}$ be the limit of $(N_n)_{n\in\NN}$. By hypothesis and possibly passing to a
  further subsequence $(N_{m_\ell})_{\ell\in\NN}$, there exist sets $U_{m_\ell}\subseteq
  V(N_{m_\ell})$ with $\lim_{\ell\to\infty}\lvert U_{m_\ell}\rvert/\lvert N_{m_\ell}\rvert > 0$ such
  that the sequence $(N_{m_\ell}\rest_{U_{m_\ell}})_{\ell\in\NN}$ is convergent and any
  $(1,k(P)-1)$-split of any predicate symbol $P\in\cL$ is almost stable in it. Since
  $(N_{m_\ell})_{\ell\in\NN}$ also converges to $\phi$, by Lemma~\ref{lem:limitsubobject}, it
  follows that $\phi$ contains a sub-object in which every $(1,k(P)-1)$-split of every predicate
  symbol $P\in\cL$ is almost stable.

  Let then $\cH$ be a $T$-on over some space $\Omega=(X,\cA,\mu)$ with $\phi = \phi_\cH$. By
  Theorem~\ref{thm:stabletheon}, there exists a positive measure $U\subseteq X$ such that
  $\psi\df\phi_{\cH\rest_U^F}$ is trivial for every measure-isomorphism $F$ modulo $0$ from
  $\Omega_U$ to $\Omega$.

  Let $c\df\mu(U) > 0$. We claim that for every $T$-on $\cH'$ over some space
  $\Omega'=(X',\cA',\mu')$ with $\phi = \phi_{\cH'}$, there exists a measurable set $U'\subseteq X'$
  such that $\phi_{\cH'\rest_{U'}^{F'}} = \psi$ for every measure-isomorphism $F'$ modulo $0$ from
  $\Omega_{U'}$ to $\Omega$ and $\mu'(U')\geq c$.

  This is completely trivial from the Theon Uniqueness Theorem~\cite[Theorems~3.9 and~3.11 and
    Proposition~7.7]{CR20a}, but an ad hoc proof analogous to the one in
  Theorem~\ref{thm:stableconvseqgraph} can be obtained from Lemma~\ref{lem:limitsubobject}: by this
  lemma applied to $\cH$, there exists a sequence $(N'_n)_{n\in\NN}$ converging to $\phi$ and sets
  $U'_n\subseteq V(N'_n)$ with $\lim_{n\to\infty} \lvert U'_n\rvert/\lvert N'_n\rvert = c$ and
  $(N'_n\rest_{U'_n})_{n\in\NN}$ converging to $\psi$. Applying this lemma again to $\cH'$, it
  follows that for some measurable function $f\function{X'}{[0,1]}$ with $\int_{X'} f\ d\mu' = c$
  and every measure-isomorphism $F$ modulo $0$ from $\Omega_f$ to $\Omega$, we have
  $\phi_{\cH'\rest_f^F}=\psi$. Taking $U'\df\{x\in X' \mid f(x) > 0\}$ gives $\mu'(U')\geq c$ and
  $\phi_{\cH'\rest_{U'}^{F'}}=\psi$ for every measure-isomorphism $F'$ modulo $0$ from $\Omega_{U'}$
  to $\Omega$ (note that the rescaling $f$ does not change densities of submodels within $U'$
  because $\psi$ is trivial) completing the proof of the claim.

  \smallskip

  Since $\psi=\phi_{\cH\rest_U^F}$ is trivial, for each $P\in\cL$, we know that $b_P\df
  \mu_U(\cH\rest_U^F)\in\{0,1\}$. For each $n\in\NN$, let $U_n^c\subseteq V(N_n)$ be a set that
  minimizes the quantity
  \begin{align*}
    d_n(U_n)
    \df
    \sum_{P\in\cL}
    \left\lvert
    \frac{\lvert P^{N_n\rest_{U_n}}\rvert}{\lvert U_n\rvert^{k(P)}} - b_P
    \right\rvert
  \end{align*}
  over all possible sets $U_n\subseteq V(N_n)$ with $\lvert U_n\rvert\geq (c/2)\cdot\lvert
  N_n\rvert$. To conclude the proof, it is sufficient to show that $\lim_{n\to\infty} d_n(U_n^c) =
  0$. Suppose not. Then there exists a subsequence $(N_{m_\ell})_{\ell\in\NN}$ of $(N_n)_{n\in\NN}$
  such that $\lim_{\ell\to\infty} d_{m_\ell}(U_{m_\ell}^c) > 0$ and by possibly passing to a further
  subsequence, we can also assume that $(N_{m_\ell}\rest_{U_{m_\ell}^c})_{\ell\in\NN}$ is
  convergent.

  We now let $N\df\prod_{\ell\in\NN} N_{m_\ell}/\cD$ for some non-principal ultrafilter $\cD$ over
  $\NN$, let $k\geq k(P)$ for every $P\in\cL$, let $\Theta\function{\prod_{\ell\in\NN}
    V(N_{m_\ell})^k}{\cE_k}$ and $\cN$ be as in Theorem~\ref{thm:theonultraproduct} and per our
  previous claim, there exists a measurable set $U'\subseteq [0,1]$ with $\lambda(U')\geq c$ and
  $\phi_{\cN\rest_{U'}^{F'}}=\psi$ for every measure-isomorphism $F'$ modulo $0$ from
  $([0,1],\lambda_{U'})$ to $([0,1],\lambda)$. Define further $\widehat{U}\df
  \Theta_1^{-1}(U')\subseteq\prod_{\ell\in\NN} V(N_{m_\ell})/\cD$, where $\Theta_1$ is a restriction
  of $\Theta$ or order $1$, and note that since $\Theta_1$ is measure-preserving, we have
  $\mu^1(\widehat{U})\geq c$ for the Loeb measure $\mu^1$.

  Consider now the model $N\rest_{\widehat{U}}$ and note that for every predicate symbol $P\in\cL$
  and every restriction $\Theta_{k(P)}$ of $\Theta$ of order $k(P)$, we have
  \begin{align*}
    P^{N\rest_{\widehat{U}}}
    & =
    \Theta_{k(P)}^{-1}(\{x\in\cN_P \mid\forall v\in [k(P)], x_{\{v\}}\in U'\})
    \quad\text{a.e.}
  \end{align*}
  and since $\lambda_{U'}((\cN\rest_{U'}^{F'})_P) = b_P$, it follows that
  \begin{align*}
    \mu^{k(P)}(P^{N\rest_{\widehat{U}}})
    & =
    b_P\cdot\mu^{k(P)}(\widehat{U}^{k(P)})
    =
    b_P\cdot\mu^1(\widehat{U})^{k(P)},
  \end{align*}
  where the last equality follows from Fubini's Theorem for Loeb measures,
  Theorem~\ref{thm:FubiniLoeb}.
  
  Let now $U\df\prod_{\ell\in\NN} U_\ell/\cD$ be an internal set such that
  $\mu^1(U\symdiff\widehat{U}) = 0$. By Fubini's Theorem again, it follows that
  $\mu^{k(P)}(P^{N\rest_U}) = b_P\cdot\mu^1(U)^{k(P)}$, so we must have
  \begin{align*}
    \lim_{\ell\to\cD} \frac{\lvert U_\ell\rvert}{\lvert N_{m_\ell}\rvert}
    & =
    \mu^1(U)
    \geq
    c,
    \\
    \lim_{\ell\to\cD} \frac{\lvert P^{N_{m_\ell}\rest_{U_\ell}}\rvert}{\lvert U_\ell\rvert^{k(P)}}
    & =
    \frac{\mu^{k(P)}(P^{N\rest_U})}{\mu^1(U)^{k(P)}}
    =
    b_P
    \qquad
    (P\in\cL).
  \end{align*}

  However, this is a contradiction because it implies that along some subsequence we have $\lvert
  U_\ell\rvert/\lvert N_{m_\ell}\rvert\geq c/2$ and $d_{m_\ell}(U_\ell)\to 0$, contradicting the
  fact that the former implies $d_{m_\ell}(U_{m_\ell}^c)\leq d_{m_\ell}(U_\ell)$ and we have
  $d_{m_\ell}(U_{m_\ell}^c)\not\to 0$.
\end{proof}

Finally, with an argument similar to that of Theorem~\ref{thm:stablecountablegraph}, we can prove a
stability dichotomy for countable models.

\begin{theorem}\label{thm:stablecountablemodel}
  Let $T$ be a universal theory in a finite relational language $\cL$. The following for a countable
  model $N$ of $T$ with $V(N)=\NN_+$.
  \begin{enumerate}
  \item There exist a set $U\subseteq\NN_+$ and an increasing sequence $(n_\ell)_{\ell\in\NN}$ of
    positive integers such that $(N\rest_{U\cap [n_\ell]})_{\ell\in\NN}$ is almost trivial and
    $\lim_{\ell\to\infty} \lvert U\cap [n_\ell]\rvert/n_\ell > 0$.%
    \label{thm:stablecountablemodel:trivial}
  \item There exist a set $U\subseteq\NN_+$ and an increasing sequence $(n_\ell)_{\ell\in\NN}$ of
    positive integers such that every $(1,k(P)-1)$-split of every predicate symbol $P\in\cL$ is
    almost stable in $(N\rest_{U\cap[n_\ell]})_{\ell\in\NN}$ and $\lim_{\ell\to\infty} \lvert U\cap
    [n_\ell]\rvert/n_\ell > 0$.%
    \label{thm:stablecountablemodel:stable}
  \end{enumerate}
\end{theorem}

\begin{proof}
  The implication~\ref{thm:stablecountablemodel:trivial}$\implies$\ref{thm:stablecountablemodel:stable} is
  trivial as (all splits of) all open formulas are almost stable in almost trivial sequences.

  \medskip

  For the
  implication~\ref{thm:stablecountablemodel:stable}$\implies$\ref{thm:stablecountablemodel:trivial},
  by possibly passing to a subsequence of $(n_\ell)_{\ell\in\NN}$, we may further assume that
  $(N\rest_{[n_\ell]})_{n\in\NN}$ is convergent, so by Theorem~\ref{thm:stableconvsequniversal},
  there exist $c > 0$ and sets $U_\ell\subseteq [n_\ell]$ such that $\lvert U_\ell\rvert \geq c\cdot
  n_\ell$ for every $\ell\in\NN$ and $(N\rest_{U_\ell})_{\ell\in\NN}$ is almost trivial.

  We then use the same recursive definition from Theorem~\ref{thm:stablecountablegraph} of the
  sequence $(m_t)_{t\in\NN}$ by
  \begin{align*}
    m_0 & \df n_0, & m_{t+1} & \df \min\{n_\ell \mid \ell\in\NN\land n_\ell\geq 2^t\cdot m_t\}
  \end{align*}
  and for each $t\in\NN$, let $\ell_t\in\NN$ be such that $m_t = n_{\ell_t}$.

  Finally, letting
  \begin{align*}
    U & \df \bigcup_{t\in\NN} U_{\ell_t}\cap ([m_t]\setminus [m_{t-1}]),
  \end{align*}
  where $m_{-1}\df 0$ gives the result by an argument similar to that of
  Theorem~\ref{thm:stablecountablegraph}.
\end{proof}

\section{The approximate \Erdos--Hajnal property}
\label{sec:AEHP}

In this section, we study more systematically the approximate \Erdos--Hajnal property defined below.

\begin{definition}
  We say that a universal theory $T$ has \emph{approximate \Erdos--Hajnal property} ($\AEHP$) if
  every limit $\phi\in\HomT{T}$ contains a trivial sub-object.
\end{definition}

By
Theorem~\ref{thm:stabletheon}\ref{thm:stabletheon:trivialsubobject}$\iff$\ref{thm:stabletheon:trivialrestriction},
we have $T\in\AEHP$ if and only if every $T$-on $\cN$ over some space $\Omega=(X,\cA,\mu)$ has a
positive measure $U\subseteq X$ such that $\phi_{\cN\rest_U^F}$ is trivial for some (equivalently,
every) measure-isomorphism $F$ modulo $0$ from $\Omega_U$ to $\Omega$. See also
Discussion~\ref{dsc:AEHPconvseq} for an equivalent formulation in terms of convergent sequences.

\begin{discussion}
  Before we proceed to showing basic properties of $\AEHP$, let us note that
  Examples~\ref{ex:quasirandompermuton} and~\ref{ex:quasirandompermuton2} already bring to light a
  curious difference between the usual \Erdos--Hajnal Conjecture and its approximate version.

  For the usual version, a perfect graph $G$ of size $n$ is guaranteed to contain either a clique or
  an anti-clique of size at least $\sqrt{n}$. This is because of the trivial bound
  $\alpha(G)\chi(G)\geq n$ involving the independence and chromatic numbers of $G$ and the fact that
  the chromatic and clique numbers of $G$ are the same. On the other hand, the stable Ramsey
  Theorem~\cite{MS14,MS21} only guarantees that any stable graph on $n$ vertices contains a clique
  or anti-clique of size $n^c$ for some fixed $c\in(0,1)$ that depends only on the largest order of
  a half-graph of $G$. More generally, the \Erdos--Hajnal Conjecture is believed to be
  true~\cite{FPS19} for hereditary graphs that whose neighborhoods of vertices have bounded
  Vapnik--Chervonenkis dimension~\cite{VC71} (these are also known as classes with NIP, i.e.,
  without the independence property, in model theory).

  However, for the approximate version, Theorem~\ref{thm:stableconvseqgraph} implies that any
  convergent sequence of stable graphs contains a linear-sized almost clique or almost anti-clique,
  but since every graph of agreements of a permutation is a perfect graph,
  Example~\ref{ex:quasirandompermuton2} says that there exists a convergent sequence of perfect
  graphs without any linear-sized almost clique or anti-clique.

  Furthermore, it is easy to see that the theory $T$ of graphs of agreements of permutations has
  NIP, i.e., neighborhoods of vertices have bounded VC~dimension: this is because any hereditary class of
  graphs without NIP is required to have at least $2^{\Omega(n^2)}$ different graphs
  with vertex set $[n]$ and $T$ has at most $(n!)^2$ different graphs with vertex set $[n]$ (as
  models of $\TPerm$ over $[n]$ consist of two linear orders on $[n]$).

  This means that neither perfection nor NIP are enough to ensure that convergent sequences of graphs contain
  linear-sized almost cliques or anti-cliques. As we will see in Section~\ref{sec:AEHPgraphs}, for graph
  theories, the approximate \Erdos--Hajnal property is equivalent to forbidding some induced subgraph of some
  recursive blow-up of the $4$-cycle.
\end{discussion}

Let us now prove some basic properties about $\AEHP$.
\begin{proposition}\label{prop:AEHPaddaxioms}
  If $T'\vdash T$ and $T\in\AEHP$, then $T'\in\AEHP$.
\end{proposition}

\begin{proof}
  This follows immediately since every $T'$-on is also a $T$-on.
\end{proof}

Next we will study the universal theory analogue of the substitution operation studied for the
original \Erdos--Hajnal property (see also Remark~\ref{rmk:subst} below and cf.~\cite{APS01}
and~\cite[\S 2]{Chu14}).
\begin{definition}
  Let $T_1$ and $T_2$ be universal theories in the same language and let $\forall x_1,\ldots,x_n,
  F(x_1,\ldots,x_n)$ be an axiom of $T_1$.

  We define the universal theory $T_1^{F\to T_2}$ as the theory obtained from $T_1$ by removing the
  axiom $\forall x_1,\ldots,x_n, F(x_1,\ldots,x_n)$ and for every axiom of $T_2$ of the form
  $\forall y_1,\ldots,y_m,G(y_1,\ldots,y_m)$, adding the axiom
  \begin{align*}
    \forall x_1,\ldots,x_{n+m-1},F^G(x_1,\ldots,x_{n+m-1}),
  \end{align*}
  where $F^G(x_1,\ldots,x_{n+m-1})$ is the formula
  \begin{align*}
    G(x_n,\ldots,x_{n+m-1})\lor\bigvee_{i=n}^{n+m-1} F(x_1,\ldots,x_{n-1},x_i).
  \end{align*}
\end{definition}

\begin{remark}\label{rmk:subst}
  When all predicate symbols of the language $\cL$ have arity at most $2$ and the theories $T_1$ and
  $T_2$ are of the form\footnote{Note that any canonical theory can be reaxiomatized as
    $\Forb_{T_\cL}(\cF)$ for some $\cF$.} $\Forb_T(\cF)$ for some canonical theory $T$ in $\cL$ and
  some family $\cF$ of finite models of $T$, that is, the axioms of $\Forb_T(\cF)$ are those of $T$
  along with $\forall\vec{x},\neg\Dopen(F)(\vec{x})$ for each $F\in\cF$, then the substitution
  operation can be done at the level of the family $\cF$ (cf.~\cite{APS01} and~\cite[\S 2]{Chu14}).

  Namely, given finite $\cL$-structures $F_1$ and $F_2$ and some $v\in V(F_1)$, the
  \emph{substitution} $F_1^{v\to F_2}$ of $v$ by $F_2$ in $F_1$ is the $\cL$-structure obtained by
  replacing the vertex $v$ with $\lvert F_2\rvert$ copies of it inducing a copy of $F_2$ (see
  Figure~\ref{fig:graphsubst} for an example in the theory of graphs).

  \begin{figure}[htb]
    \input{graphsubst}
  \end{figure}

  Note that if $V(F_1)=[n]$, then the formula $\neg\Dopen(F_1)^{\neg\Dopen(F_2)}$ is equivalent to
  $\neg\Dopen(F_1^{n\to F_2})$. Thus for families $\cF_1$ and $\cF_2$ of finite $\cL$-structures and
  for $F_1\in\cF_1$ with $V(F_1)=[n]$, the theory $\Forb_T(\cF_1)^{\neg\Dopen(F_1)\to\Forb_T(\cF_2)}$ is
  equivalent to the theory $\Forb_T(\cF')$, where
  \begin{align*}
    \cF' & \df (\cF_1\setminus\{F_1\})\cup\{F_1^{n\to F_2} \mid F_2\in\cF_2\}.
  \end{align*}

  However, note that when $\cL$ has predicate symbols of arity at least $3$, such easy description
  is not possible: the substitution operation will be completely agnostic about tuples containing at
  least two vertices of $F_2$ and at least one vertex of $F_1$ that is not $v$.
\end{remark}

\begin{remark}\label{rmk:persistent}
  Theories of the form $\Forb_T(\{F\})$ with $\AEHP$ also bring to light models that have
  \emph{positive density in all limits without trivial sub-objects}, namely, for each
  $\phi\in\HomT{T}$, let $\cC_P(\phi)\df\cM[\Th(\phi)]$ be the set of finite models $M$ of $T$ (up
  to isomorphism) such that $\phi(M) > 0$ and let
  \begin{align*}
    \cC_P(T) & \df \bigcap_\phi \cC_P(\phi),
  \end{align*}
  where the intersection is over all $\phi\in\HomT{T}$ that do not have any trivial sub-object (if
  $T$ already has $\AEHP$, this empty intersection is assumed to result $\cM[T]$ by convention).

  We claim that $\cC_P(T)$ is exactly the class of finite models $F$ of $T$ (up to isomorphism) such
  that $\Forb_T(\{F\})\in\AEHP$. Both containments are more easily shown by their contrapositive. If
  $\Forb_T(\{F\})\notin\AEHP$, then there must be some $\phi\in\HomT{\Forb_T(\{F\})}$ without any
  trivial sub-object, but for the \emph{axiom-erasing interpretation}
  $I\interpret{T}{\Forb_T(\{F\})}$ that acts identically on the language of $T$, we have
  $\phi^I\in\HomT{T}$ and since $\phi^I(F) = 0$, we have $F\notin\cC_P(T)$. On the other hand, if
  $F$ is a model of $T$ that is not in $\cC_P(T)$, then there exists some $\phi\in\HomT{T}$ without
  any trivial sub-object such that $\phi(F) = 0$, but the latter condition implies that $\phi$ can
  be seen as an element of $\HomT{\Forb_T(\{F\})}$ and thus $\Forb_T(\{F\})\notin\AEHP$.

  Note that we could have equivalently have defined $\cC_P(T)$ as the set of all finite models $F$
  that ``persist'' in the stronger sense that they have positive density in every sub-object $\psi$
  of every $\phi\in\HomT{T}$ that does not have trivial sub-objects. This seemingly stronger
  ``persistence'' definition yields precisely the same class of objects because of the
  quantification of $\phi$ and the fact that if $\phi$ does not have any trivial sub-object, then
  any sub-object of $\phi$ also has this property.
\end{remark}

Before we proceed, let us recall the definition of substitutionally closed theories
from~\cite[Definition~3.6]{CR20a}.
\begin{definition}
  Given an open formula $F(x_1,\ldots,x_n)$ and an equivalence relation $\sim$ on $[n]$ with $m$
  equivalence classes $C_1,\ldots,C_m$, we define the open formula $F_\sim(y_1,\ldots,y_m)$ as the
  formula $F(y_{\nu_1},\ldots,y_{\nu_n})$, where $\nu$ is the unique function such that $x_t\in
  C_{\nu_t}$ for every $t\in[n]$.

  A universal theory $T$ is said to be \emph{substitutionally closed} if for each axiom
  $\forall\vec{x},F(\vec{x})$ and each equivalence relation $\sim$, $T$ proves
  $\forall\vec{y},F_\sim(\vec{y})$ using only propositional rules and \emph{injective} renamings of
  variables (but replacing two different variables with the same variable is disallowed).

  The \emph{substitutional closure} of $T$ is the theory whose axioms are
  $\forall\vec{y},F_\sim(\vec{y})$ for each axiom $\forall x_1,\ldots,x_n,F(\vec{x})$ of $T$ and
  each equivalence relation $\sim$ on $[n]$.
\end{definition}

Note that if $T'$ is the substitutional closure of $T$, then $T\vdash T'$ and $T'\vdash T$, that is,
substitutional closedness is a property of the \emph{axiomatization} of $T$ rather than its set of
theorems. For substitutionally closed theories $T$, the next theorem from~\cite{CR20a} gives a
simpler characterization of $T$-ons as Euclidean structures satisfying the axioms of $T$.

\begin{theorem}[\protect{\cite[Theorem~3.7]{CR20a}}]\label{thm:substclosed}
  Let $T$ be a canonical substitutionally closed universal theory in a finite relational language
  and $\cN$ be an Euclidean structure on some space $\Omega=(X,\cA,\mu)$ in the language of
  $T$. Then $\cN$ is a weak (strong, respectively) $T$-on if and only if for every axiom $\forall
  x_1,\ldots,x_n,F(\vec{x})$ of $T$, we have $\mu(T(F,\cN))=1$
  ($T(F,\cN)\supseteq\cE_n(\Omega)\setminus\cD_n(\Omega)$, respectively).
\end{theorem}

\begin{remark}
  It is easy to see that the forward direction of Theorem~\ref{thm:substclosed} does not require the
  substitutional closedness property. For the backward direction, the necessity of the property is
  illustrated in~\cite[Example~37]{CR20a}.
\end{remark}

\begin{theorem}\label{thm:AEHPsubstitution}
  Let $T_1$ and $T_2$ be canonical universal theories in the same finite relational language and let
  $\forall x_1,\ldots,x_n,F(x_1,\ldots,x_n)$ be an axiom of $T_1$. If $T_1,T_2\in\AEHP$ and
  $T_1^{F\to T_2}$ is canonical, then $T_1^{F\to T_2}\in\AEHP$.
\end{theorem}

\begin{proof}
  Let us first prove the case when all axioms of $T_1$ and $T_2$ are of the form
  \begin{align}\label{eq:forcesubstclosed}
    \forall x_1,\ldots,x_t, \left(\bigwedge_{1\leq i < j\leq t} x_i\neq x_j \to A(\vec{x})\right)
  \end{align}
  for some open formula $A$. Note that under these conditions $T_1$, $T_2$ and $T_1^{F\to T_2}$ are
  substitutionally closed as any replacement of two different variables with the same variable leads
  to a tautology.

  Let $T\df T_1^{F\to T_2}$. We need to show that every $T$-on $\cN$ over some space $\Omega =
  (X,\cA,\mu)$ contains a trivial sub-object. By possibly applying Theorem~\ref{thm:theonremoval},
  we may suppose that $\cN$ is a strong $T$-on. If $\cN$ is a $T_1$-on, then this follows from
  $T_1\in\AEHP$, so suppose $\cN$ is not a $T_1$-on. Since the only axiom of $T_1$ that is not an
  axiom of $T$ is $F$, by Theorem~\ref{thm:substclosed}, we must have $\mu(T(F,\cN)) < 1$ and thus
  $\mu(T(\neg F,\cN)) > 0$.  By Fubini's Theorem, there exists some $z\in\cE_{n-1}(\Omega)$ such
  that the set
  \begin{align*}
    C(z) & \df
    \{y\in X^{\{\{n\}\}} \mid \mu(C(z,y)) > 0\},
  \end{align*}
  has positive measure, where
  \begin{align*}
    C(z,y) & \df \{w\in X^{r(n)\setminus (r(n-1)\cup\{\{n\}\})} \mid (z,y,w)\in T(\neg F,\cN)\}.
  \end{align*}
  By identifying $X^{\{\{n\}\}}$ with $X$, the set
  \begin{align*}
    U & \df \{y\in C(z) \mid \forall j\in[n-1], y\neq z_{\{j\}}\}
  \end{align*}
  also has positive measure.

  Let $\widetilde{F}$ be a measure-isomorphism modulo $0$ from $\Omega_U$ to $\Omega$. We claim that
  $\cN\rest_U^{\widetilde{F}}$ is a $T_2$-on. Suppose not. By Theorem~\ref{thm:substclosed}, there
  exists some axiom of $T_2$ of the form $\forall x_n,\ldots,x_{n+m-1},G(x_n,\ldots,x_{n+m-1})$ (we
  index the variables by $V\df\{n,\ldots,n+m-1\}$ for convenience) such that
  $\mu(T(G,\cN\rest_U^{\widetilde{F}})) < 1$. In particular, this means that there exists a point
  $y\in\cE_V(\Omega)\setminus\cD_V(\Omega)$ such that $y\notin T(G,\cN\rest_U^{\widetilde{F}})$ and
  $y_{\{v\}}\in U$ for every $v\in V$. Then we can define a point
  $\widetilde{z}\in\cE_{n+m-1}(\Omega)\setminus\cD_{n+m-1}(\Omega)$ as follows.
  \begin{enumerate}[label={\alph*.}, ref={(\alph*)}]
  \item For each $A\in r(n-1)$, define $\widetilde{z}_A\df z_A$.
  \item For each $A\in r(V)$, define $\widetilde{z}_A\df y_A^{\widetilde{F}}$, where $y_A^{\widetilde{F}}$ is
    given by~\eqref{eq:xF}.
  \item For each $i\in V$, since $y_{\{i\}}\in U\subseteq C(z)$, let $w^i\in C(z,y_{\{i\}})$ and define
    $\widetilde{z}_{A\cup\{i\}}\df w^i_{A\cup\{i\}}$ for every $A\in r(n-1)$.
  \item Define all other coordinates arbitrarily.
  \end{enumerate}
  The definition of $U$ ensures that $\widetilde{z}\notin\cD_{n+m-1}(\Omega)$. Furthermore, since
  $(z,y_{\{i\}},w^i)\in T(\neg F,\cN)$ for every $i\in V$ and $y\in T(\neg
  G,\cN\rest_U^{\widetilde{F}})$, it follows that $\widetilde{z}\notin T(F^G,\cN)$, contradicting
  the fact that $\cN$ is a \emph{strong} $T$-on.

  Therefore $\cN\rest_U^{\widetilde{F}}$ is a $T_2$-on and since $T_2\in\AEHP$, it must contain a
  trivial sub-object, which must also be a sub-object of $\phi_\cN$ (as
  $\phi_{\cN\rest_U^{\widetilde{F}}}$ is a sub-object of $\phi_\cN$).

  \medskip

  Let us now prove the case in which all axioms of $T_2$ are of the form~\eqref{eq:forcesubstclosed}
  but those of $T_1$ are not necessarily of this form. Let $T_1'$ be the theory whose axioms are
  \begin{align}\label{eq:forcesubstclosed2}
    \forall y_1,\ldots,y_m, \left(\bigwedge_{1\leq i < j\leq m} y_i\neq y_j \to A_\sim(y_1,\ldots,y_m)\right)
  \end{align}
  for every axiom $\forall x_1,\ldots,x_t, A(x_1,\ldots,x_t)$ of $T_1$, every equivalence relation
  $\sim$ on $[t]$ with $m$ equivalence classes $C_1,\ldots,C_m$ and without loss of generality, let
  us assume that we always enumerate these classes in a way that $x_t\in C_m$. Let us use the
  notation $A_\sim'$ for the open formula in~\eqref{eq:forcesubstclosed2}. Note that $T_1'\vdash
  T_1$, so $T_1'\in\AEHP$ by Proposition~\ref{prop:AEHPaddaxioms}.

  Let us now focus our attention on the open formula $F(x_1,\ldots,x_n)$ and let us enumerate all
  equivalence relations on $[n]$ as $\sim_1,\ldots,\sim_\ell$.

  We now define theories $T^i$ for $i\in\{0,\ldots,\ell\}$ inductively by letting $T^0\df T_1'$ and
  $T^{i+1}\df (T^i)^{F_{\sim_i}\to T_2}$. A simple induction shows that $T^i$ can be reaxiomatized
  so that all of its axioms are of the form~\eqref{eq:forcesubstclosed} and thus by the previous
  case (and Proposition~\ref{prop:AEHPaddaxioms}) another induction gives $T^i\in\AEHP$. On the
  other hand, it is straightforward to see that $T^\ell$ is a reaxiomatization of $T_1^{F\to T_2}$,
  so we get $T_1^{F\to T_2}\in\AEHP$ by Proposition~\ref{prop:AEHPaddaxioms}.

  \medskip

  Finally, for the case when both $T_1$ and $T_2$ are general, we can let $T_2'$ be the theory whose axioms
  are~\eqref{eq:forcesubstclosed2} but for every axiom of $T_2$ instead so that its axioms are all of the
  form~\eqref{eq:forcesubstclosed}. Then we clearly have $T_2'\vdash T_2$ and $T_1^{F\to T_2}\vdash T_1^{F\to
    T_2'}$ so the result follows from two applications of Proposition~\ref{prop:AEHPaddaxioms} and the
  previous case.
\end{proof}

Similarly to the results of Sections~\ref{sec:consgraphs} and~\ref{sec:consmodels}, the approximate
\Erdos--Hajnal property can also be pulled back to the finite world. Since the proofs are completely
analogous to those of Section~\ref{sec:consmodels}, we state these results without proof here.

\begin{theorem}\label{thm:AEHPconvseq}
  Let $T$ be a universal theory in a finite relational language $\cL$ such that $T\in\AEHP$ and let
  $(N_n)_{n\in\NN}$ be a convergent sequence of structures in $\cL$.

  Suppose there exists a subsequence $(N_{n_\ell})_{\ell\in\NN}$ of $(N_n)_{n\in\NN}$ and sets
  $U_{n_\ell}\subseteq V(N_{n_\ell})$ such that $\limsup_{\ell\to\infty} \lvert
  U_{n_\ell}\rvert/\lvert N_{n_\ell}\rvert > 0$ and for every finite $\cL$-structure $M$ that is
  \emph{not} a model of $T$, we have $\lim_{\ell\to\infty} p(M,N_{n_\ell}\rest_{U_{n_\ell}}) = 0$.

  Then there exist $c > 0$ and sets $U_n\subseteq V(N_n)$ such that $\lvert U_n\rvert\geq c\lvert
  N_n\rvert$ for every $n\in\NN$ and $(N_n\rest_{U_n})_{n\in\NN}$ is almost trivial.
\end{theorem}

\begin{theorem}\label{thm:AEHPcountablemodel}
  Let $T$ be a universal theory in a finite relational language $\cL$ such that $T\in\AEHP$ and let
  $N$ be a countable $\cL$-structure with $V(N)=\NN_+$.

  Suppose there exist a set $U\subseteq\NN_+$ and an increasing sequence $(n_\ell)_{\ell\in\NN}$ of
  positive integers such that for every finite $\cL$-structure $M$ that is \emph{not} a model of
  $T$, we have $\lim_{\ell\to\infty} p(M,N\rest_{U\cap [n_\ell]}) = 0$ and $\lim_{\ell\to\infty}
  \lvert U\cap [n_\ell]\rvert/n_\ell > 0$.

  Then there exist a set $U\subseteq\NN_+$ and an increasing sequence $(n_\ell)_{\ell\in\NN}$ of
  positive integers such that $(N\rest_{U\cap [n_\ell]})_{\ell\in\NN}$ is almost trivial and
  $\lim_{\ell\to\infty} \lvert U\cap [n_\ell]\rvert/n_\ell > 0$.
\end{theorem}

\begin{remark}
  Differently from the case of Theorems~\ref{thm:stableconvseqgraph},
  \ref{thm:stablecountablegraph}, \ref{thm:stableconvsequniversal}
  and~\ref{thm:stablecountablemodel}, in Theorems~\ref{thm:AEHPconvseq}
  and~\ref{thm:AEHPcountablemodel} we do not get an equivalence as the ``trivial'' implication
  breaks down: an almost trivial sequence $(N_n)_{n\in\NN}$ does not need to be a sequence of
  ``almost'' models of $T$. For example, any universal theory $T$ without infinite models (in flag
  algebra language, a degenerate theory) vacuously satisfies $\AEHP$ as it does not have any
  increasing sequence of models and by the same token it cannot have an increasing sequence of
  ``almost'' models. For a slightly less trivial example, if $T$ is the theory of empty graphs and
  $(N_n)_{n\in\NN}$ is a an increasing sequence of complete graphs, then $(N_n)_{n\in\NN}$ is
  (almost) trivial but does not contain any increasing subsequence of induced subgraphs that are
  ``almost'' empty graphs.
\end{remark}

\begin{discussion}\label{dsc:AEHPconvseq}
  A consequence of Theorem~\ref{thm:AEHPconvseq} is that $T\in\AEHP$ is equivalent to every
  \emph{convergent} sequence of models of $T$ having an almost trivial sequence of linear-sized
  induced submodels. Again, the convergence condition is essential (see
  Discussion~\ref{dsc:convergence}) and requiring almost trivial as opposed to trivial is also
  essential (see Discussion~\ref{dsc:disjunioncliques}).

  As it was already observed in~\cite[\S 5]{Chu14}, without the conditions above the problem
  completely trivializes for graphs: if we require a universal theory $T$ of graphs (i.e.,
  $T\vdash\TGraph$) to be such that every sufficiently large model $M$ of $T$ either contains a
  clique or anti-clique of size strictly larger than $\sqrt{\lvert M\rvert}$, then $T$ must forbid
  some disjoint union of cliques and some complete partite graph. This stems from the graphs of
  Discussion~\ref{dsc:disjunioncliques}: the largest cliques and anti-cliques in the graph $H_{m,m}$
  consisting of a disjoint union of $m$ cliques of size $m$ have size $m$ so some induced subgraph
  of $H_{m,m}$, which is necessarily a disjoint union of cliques, must be forbidden by
  $T$. Similarly, the complement $\overline{H}_{m,m}$ of $H_{m,m}$ shows that $T$ must forbid some
  complete partite graph.
\end{discussion}

\section{Characterization via forbidden subgraphs}
\label{sec:AEHPgraphs}

The purpose of this section is to completely characterize the approximate \Erdos--Hajnal property ($\AEHP$)
for universal theories of graphs. Specifically, we show (Theorem~\ref{thm:AEHPgraph}) that universal theories
of graphs with $\AEHP$ are precisely characterized as the ones that forbid some induced subgraph of some
recursive blow-up of the $4$-cycle $C_4$ (defined below). Let us remind the reader that even the
existence of such family characterizing $\AEHP$ for $\TGraph$ is a surprise: in general, it is not
clear that given a universal theory $T$, there exists a family $\cC$ such that any universal theory
$T'\vdash T$ has $\AEHP$ if and only if it forbids some element of $\cC$.

\begin{definition}\label{def:C4omega}
  For $\ell\in\NN$, the \emph{recursive blow-up of the $4$-cycle of height $\ell$} is the graph
  $C_4^\ell$ defined by $V(C_4^\ell) = [4]^\ell$ and in which two distinct vertices $\sigma,\tau\in
  [4]^\NN$ are adjacent if and only if $\sigma_i - \tau_i\equiv \pm 1\pmod{4}$, where $i\in[\ell]$
  is the first position in which $\sigma$ and $\tau$ differ (see Figure~\ref{fig:C4ell}).

  We also let $\cC_C$ be the set of all graphs (up to isomorphism) that are induced subgraphs of
  $C_4^\ell$ for some $\ell\in\NN$.
  
  The \emph{recursive blow-up of the $4$-cycle of countable height} is the graph $C_4^\omega$
  defined by $V(C_4^\omega) = [4]^\NN$ and in which two distinct vertices $\sigma,\tau\in [4]^\NN$
  are adjacent if and only if $\sigma_i - \tau_i\equiv \pm 1\pmod{4}$, where $i\in\NN$ is the first
  position in which $\sigma$ and $\tau$ differ.
\end{definition}

\begin{figure}[htb]
  \input{C4ell}
\end{figure}

\begin{remark}\label{rmk:Cc}
  It is easy to see that $\cC_C$ can alternatively be described as the class of finite graphs $G$
  that are induced subgraphs of $C_4^\omega$. If we wanted a smaller single graph $H$ whose class of
  finite induced subgraphs is $\cC_C$, we could also take $H$ as the disjoint union
  $\bigsqcup_{\ell\in\NN} C_4^\ell$ or as any direct limit $\mathop{\underrightarrow{\lim}}
  C_4^\ell$ (in the categorical sense) relative to any direct system of embeddings
  $C_4^\ell\rightarrowtail C_4^k$ ($\ell\leq k$); both of these are countable graphs.
\end{remark}

\begin{remark}\label{rmk:otherCc}
  Let us note that there is not much particularly special about $C_4$ in the definition of
  $\cC_C$. Namely, if $G\in\cC_C$ contains at least one edge and one non-edge, then by analogously
  defining the recursive blow-ups $G^\ell$ and $G^\omega$ of $G$ of height $\ell\in\NN$ and of
  countable height, respectively, it straightforward to check that $\cC_C$ is precisely the set
  graphs that are induced subgraphs of some $G^\ell$ or alternatively, the set of finite induced
  subgraphs of $G^\omega$.
\end{remark}

Let us now give an intuition of the steps required to show that a universal theory of graphs
$T\vdash\TGraph$ has $\AEHP$ if and only if some $F\in\cC_C$ is \emph{not} a model of $T$.

First, recall from Remark~\ref{rmk:persistent} that $\cC_P(\TGraph)$ is the class of all graphs $G$
(up to isomorphism) that ``persistently have positive density'' in the sense that $\phi(G)>0$ for
every $\phi\in\HomT{\TGraph}$ that does not have any trivial sub-object (i.e., every graphon without
any subgraphon that is an almost clique or almost anti-clique). Recall also from
Remark~\ref{rmk:persistent} that $\cC_P(\TGraph)$ can be described alternatively as the class of
finite graphs $F$ (up to isomorphism) such that $\Forb_{\TGraph}(\{F\})\in\AEHP$. Let now $\cC_M$ be
the union of all classes of graphs $\cF$ (up to isomorphism) that are minimal for the property that
$\Forb_{\TGraph}(\cF)\in\AEHP$, i.e., $\cC_M$ is the set of graphs that appear in some such minimal
class. Our characterization can then be restated as the equality $\cC_P(\TGraph) = \cC_M = \cC_C$.

To show these equalities, let us introduce one more class: let $\cC_S$ be the smallest class of
graphs (up to isomorphism) that contains all graphs of size at most $2$ (i.e., the trivial graph
$K_0$ with no vertices, the single vertex graph $K_1$, the edge $K_2$ and the non-edge
$\overline{K}_2$) and that is closed under the substitution operation of Remark~\ref{rmk:subst}
(note that substitutions of the form $G^{v\to K_0}$ are isomorphic to $G\rest_{V(G)\setminus\{v\}}$,
so we could have defined equivalently $\cC_S$ as the smallest class containing the edge, the
non-edge and that is closed under both the substitution operation and taking induced subgraphs).

The proof of Theorem~\ref{thm:AEHPgraph} can be informally summarized by the following steps.
\begin{enumerate}[label={\arabic*.}]
\item By Remark~\ref{rmk:subst} and Theorem~\ref{thm:AEHPsubstitution} (and the fact that we
  trivially have $\Forb_{\TGraph}(\{F\})\in\AEHP$ whenever $F$ has at most $2$
  vertices\footnote{There is a small difference between $\Forb_{\TGraph}(\{K_0\})$ and
    $\Forb_{\TGraph}(\{K_1\})$: the former has no models at all while the latter has only $K_0$ as
    its model (recall from Footnote~\ref{ftnt:emptyvertexset} that we allow our models to have empty
    vertex set). However, since neither of them contain any increasing sequences of models, they
    satisfy $\AEHP$ vacuously as they do not contain any limit object.}), it follows that
  $\cC_S\subseteq\cC_P(\TGraph)$.
\item In Lemma~\ref{lem:phiC4zerodensities}, we will show that $\cC_C\subseteq\cC_S$ with an
  inductive argument. The combined inclusion $\cC_C\subseteq\cC_S\subseteq\cC_P(\TGraph)$ along with
  Proposition~\ref{prop:AEHPaddaxioms} then implies that if $F\in\cC_C$ is \emph{not} a model of
  some universal theory of graphs $T$, then $T\in\AEHP$ (as $T\vdash\Forb_{\TGraph}(\{F\})$).
\item For the other implication, note that if all $F\in\cC_C$ are models of a universal theory of
  graphs $T$, then the limit $\phi_{C_4}$ of $(C_4^n)_{n\in\NN}$ (see Definition~\ref{def:recC4}) is
  a limit of $T$. By showing in Lemma~\ref{lem:phiC4notrivialsubobject} that $\phi_{C_4}$ does not
  have trivial sub-objects, we get $T\notin\AEHP$ and the theorem follows. Another interpretation of
  this final step is that the fact that $\phi_{C_4}$ does not have trivial sub-objects implies that
  any collection of finite graphs $\cF$ such that $\Forb_{\TGraph}(\cF)\in\AEHP$ must necessarily
  have some element of $\cC_C$ (otherwise $\phi_{C_4}$ would be a limit of $\Forb_{\TGraph}(\cF)$ as
  $\cC_C$ is downward closed). Since $\cC_C\subseteq\cC_P(\TGraph)$, any minimal such collection
  $\cF$ must be of the form $\{F\}$ for some $F\in\cC_C$ and thus $\cC_M\subseteq\cC_C$, which along
  with the trivial containment $\cC_P(\TGraph)\subseteq\cC_M$ gives the equality of all classes
  $\cC_S=\cC_C=\cC_P(\TGraph)=\cC_M$.
\end{enumerate}

\begin{remark}\label{rmk:cCPcCM}
  In the same way that $\cC_P(T)$ is defined for arbitrary universal theories $T$, we can also
  define $\cC_M(T)$ as the union of all families $\cF$ of finite models of $T$ (up to isomorphism)
  that are minimal for the property that $\Forb_T(\cF)\in\AEHP$. Again we trivially have
  $\cC_P(T)\subseteq\cC_M(T)$, but the other inclusion need not hold for general $T$. In fact, the
  equality $\cC_P(T) = \cC_M(T)$ is equivalent to the statement that there exists a family $\cC$
  such that $T'\vdash T$ if and only if $T'$ forbids some model of $\cC$ (namely, the family is
  $\cC\df\cC_P(T) = \cC_M(T)$).
\end{remark}

\begin{definition}\label{def:recC4}
  The \emph{limit recursive blow-up of $C_4$} is the limit object $\phi_{C_4}\in\HomT{\TGraph}$ that
  is the limit of the sequence $(C_4^n)_{n\in\NN}$. It is straightforward to check that this
  sequence is convergent, but we can also alternatively define $\phi_{C_4}$ by giving an explicit
  $\TGraph$-on $\cN^{C_4}$ representing it as follows. Let $\Omega\df([4]^\NN,\cA,\nu)$ be the
  quaternary Cantor probability space, that is, $\cA$ is the Borel $\sigma$-algebra of the product
  topology on $[4]^\NN$ and $\nu$ is the unique Borel measure such that $\nu(K_\sigma) = 4^{-t}$ for
  every $t\in\NN$ and every $\sigma\in [4]^{\{0,1,\ldots,t-1\}}$, where
  \begin{align}\label{eq:Ksigma}
    K_\sigma & \df \{\tau\in[4]^\NN \mid \forall i\in\{0,1,\ldots,t-1\},\tau_i = \sigma_i\}.
  \end{align}
  The $\TGraph$-on $\cN^{C_4}$ over $\Omega$ is defined by
  \begin{align*}
    \cN^{C_4}_E
    & \df
    \{x\in\cE_2(\Omega)\setminus\cD_2(\Omega) \mid (x_{\{1\}})_i - (x_{\{2\}})_i \equiv\pm 1\pmod{4}\},
  \end{align*}
  where $i$ is the first position in which $x_{\{1\}}$ and $x_{\{2\}}$ differ.

  The corresponding graphon $W^{C_4}$ over $\Omega$ as in Remark~\ref{rmk:TGraphons} is given by
  \begin{align*}
    W^{C_4}(x,y)
    & \df
    \begin{dcases*}
      1, & if $x\neq y$ and $x_i - y_i \equiv\pm 1\pmod{4}$,\\
      0, & otherwise,
    \end{dcases*}
  \end{align*}
  where $i$ is the first position in which $x$ and $y$ differ. By using the measure-isomorphism $F$ modulo $0$
  from $\Omega$ to $[0,1]$ that maps $\sigma\in[4]^\NN$ to $\sum_{i\in\NN} \sigma_i\cdot 4^{-i-1}$, we obtain
  the equivalent graphon $\widehat{W}^{C_4}$ of Figure~\ref{fig:recC4} given indirectly by
  $\widehat{W}^{C_4}(F(\sigma),F(\tau)) = W^{C_4}(\sigma,\tau)$. Under the interpretation that a
  $\{0,1\}$-valued graphon is simply a measurable graph, $W^{C_4}$ is just the recursive blow-up $C_4^\omega$
  of the $4$-cycle of countable height equipped with the quaternary Cantor probability measure.
\end{definition}

\begin{figure}[htb]
  \input{recC4}
\end{figure}

\begin{remark}
  As we will show in Lemma~\ref{lem:phiC4notrivialsubobject} below, $\phi^{C_4}$ does not contain
  any trivial sub-object and thus by Theorem~\ref{thm:stablegraphon}, it does not contain any almost
  stable sub-object. In particular, this means that $\cC_C$ must contain half-graphs of arbitrarily
  large order, which can be verified in an ad hoc fashion as follows.

  First, it is easy to see that $\cC_C$ is closed under substitutions as if $\alpha$ and $\beta$ are
  embeddings of $F_1,F_2\in\cC_C$ in $C_4^{\ell_1}$ and $C_4^{\ell_2}$, respectively and $v\in
  V(F_1)$, then defining the concatenation map $\gamma\function{V(F_1^{v\to
      F_2})}{[4]^{\ell_1+\ell_2}}$ by
  \begin{align*}
    \gamma(w) & \df
    \begin{dcases*}
      (\alpha(w), 1^{\ell_2}), & if $w\in V(F_1)\setminus\{v\}$,\\
      (\alpha(v),\beta(w)), & if $w\in V(F_2)$
    \end{dcases*}
  \end{align*}
  gives an embedding of $F_1^{v\to F_2}$ in $C_4^{\ell_1+\ell_2}$. Thus, we have
  $\cC_S\subseteq\cC_C$.

  Now, define a sequence $(\widehat{H}_n)_{n\in\NN}$ of \emph{clique-empty-half-graphs} inductively
  by $\widehat{H}_1\df K_2$ and
  \begin{align*}
    \widehat{H}_{n+1} & \df K_2^{v\to \overline{K}_2^{w\to\widehat{H}_n}},
  \end{align*}
  that is, starting from the edge $K_2=\widehat{H}_1$, we alternate substitution operations in
  $\overline{K}_2$ and in $K_2$ (obviously, the choices of the substituted vertex do not matter
  since $K_2$ and $\overline{K}_2$ are vertex-transitive). As the name suggests, $\widehat{H}_n$ is
  a half-graph of order $n$ in which one of the sides forms a clique and the other forms an empty
  graph (see Figure~\ref{fig:cehalfgraph}) and since $\widehat{H}_n\in\cC_S\subseteq\cC_C$, it
  follows that $\cC_C$ contains half-graphs of arbitrarily large order.

  \begin{figure}[htb]
    \input{cehalfgraph}
  \end{figure}
\end{remark}

Recall that a finite graph $G$ is called \emph{prime} if it cannot be obtained from smaller graphs
via substitution, that is, $G$ is not of the form $F_1^{v\to F_2}$ for any graphs $F_1,F_2$ and
$v\in V(F_1)$ with $\lvert F_1\rvert,\lvert F_2\rvert < \lvert G\rvert$.

\begin{lemma}\label{lem:phiC4zerodensities}
  We have $\cC_C\subseteq\cC_S$. In particular, if $G$ is a finite graph such that $\phi_{C_4}(G) >
  0$, then $G\in\cC_S$.
\end{lemma}

\begin{proof}
  Let $G\in\cC_C$ and let us show that $G\in\cC_S$ by induction on the size $n$ of $G$.

  The base cases are when $G$ is a prime graph. In this case, we will show that $G$ must be an
  induced subgraph of $C_4=C_4^1$. Let $\alpha$ be an embedding of $G$ in $C_4^\ell$ for some
  $\ell\in\NN_+$. If $G$ has size at most $1$, then it is trivially a subgraph of $C_4$. If not, let
  $\sigma$ be the longest string over $[4]$ that is a prefix of every string in $\im(\alpha)$ and
  let $t$ be its length. For each $i\in[4]$, let $V_i\df \{v\in V(G) \mid \alpha(v)_{t+1} = i\}$ and
  let $G_i\df G\rest_{V_i}$. Let also $I\df\{i\in[4]\mid V_i\neq\varnothing\}$ and let
  $H=C_4\rest_I$. Note that the structure of $C_4^\ell$ implies that $G$ can be obtained from $H$ by
  substituting each $i\in I$ by $G_i$. Since $\lvert G_i\rvert < \lvert G\rvert$ for every $i\in I$
  and $G$ is prime, we must have $\lvert H\rvert = \lvert G\rvert$, that is, $\lvert V_i\rvert = 1$
  for every $i\in I$ and thus the unique $\beta\injection{V(G)}{[4]}$ such that $v\in V_{\beta(v)}$
  is an embedding of $G$ in $C_4$.

  We claim that $G$ has size at most $2$. Indeed, this follows because there are no prime graphs of
  size $3$ and $C_4$ itself is not prime. Since $\lvert G\rvert\leq 2$, we trivially have
  $G\in\cC_S$.

  \medskip
  
  For the inductive step, note that if $G$ is not prime, then it is of the form $F_1^{v\to F_2}$ for
  some graphs $F_1,F_2$ and $v\in V(F_1)$ with $\lvert F_1\rvert,\lvert F_2\rvert < \lvert
  G\rvert$. By inductive hypothesis, we have $F_1,F_2\in\cC_S$ and since $\cC_S$ is closed under
  substitutions, we get $G\in\cC_S$.

  \medskip

  Finally, since $\phi_{C_4}$ is the limit of $(C_4^n)_{n\in\NN}$, every $G$ with $\phi_{C_4}(G) >
  0$ must be an element of $\cC_C$ and thus of $\cC_S$.
\end{proof}

\begin{lemma}\label{lem:phiC4notrivialsubobject}
  $\phi_{C_4}$ does not contain any trivial sub-object.
\end{lemma}

\begin{proof}
  By~\cite[Theorem~6]{CKP21} (see also Examples~\ref{ex:quasirandompermuton}
  and~\ref{ex:quasirandompermuton2}), to show that $\phi_{C_4}$ does not have trivial sub-objects,
  we need to show that
  \begin{align*}
    \lim_{n\to\infty} \phi_{C_4}(K_n)^{1/n} = \lim_{n\to\infty} \phi_{C_4}(\overline{K}_n)^{1/n} = 0.
  \end{align*}

  We claim that for every $n\geq 2$, we have
  \begin{equation}\label{eq:phiC4clique}
    \begin{aligned}
      \phi_{C_4}(K_n)
      & =
      \tind(K_n,\cN^{C_4})
      \\
      & =
      \sum_{m\in\NN} 4^m\cdot 4^{-nm}\cdot 4\cdot 4^{-n}\cdot
      \sum_{t=1}^{n-1}\binom{n}{t}\cdot\phi_{C_4}(K_t)\cdot\phi_{C_4}(K_{n-t})
      \\
      & =
      \frac{1}{4^{n-1} - 1}\cdot
      \sum_{t=1}^{n-1}\binom{n}{t}\cdot\phi_{C_4}(K_t)\cdot\phi_{C_4}(K_{n-t}).
    \end{aligned}
  \end{equation}
  The first formula can be deduced by considering the measure of all copies of $K_n$ in $\cN^{C_4}$
  such that the largest common prefix $\sigma$ of the vertex variables (which are strings in
  $[4]^\NN$) has length $m$: there are exactly $4^m$ such $\sigma$ and the set
  $U_\sigma\subseteq\cE_n(\Omega)$ of points whose vertex variables all start with the prefix
  $\sigma$ has measure $4^{-nm}$ (i.e., the vertex variables are in the set $K_\sigma$
  of~\eqref{eq:Ksigma}). Once in $U_\sigma$, to yield a copy of $K_n$, two vertex variables
  corresponding to different vertices $i,j\in[n]$ that differ in the $(m+1)$th position must satisfy
  $(x_{\{i\}})_{m+1} - (x_{\{j\}})_{m+1}\equiv \pm 1\pmod{4}$. This means that
  \begin{align*}
    \cC & \df \{t\in[4] \mid \exists i\in[n], (x_{\{i\}})_{m+1} = t\}
  \end{align*}
  must induce a clique of size at least $2$ in $C_4$ and in fact, of size $2$ as $C_4$ is
  triangle-free. There are exactly $4$ edges in $C_4$ and a requirement of the form
  $(x_{\{i\}})_{m+1}=c_i$ for each $i\in[n]$ gives a conditional probability of $4^{-n}$ conditioned
  on $U_\sigma$. Finally, the vertex variables must be split along the chosen edge of $C_4$ with $t$
  vertices to one side forming a $K_t$ and $n-t$ vertices to the other side forming a $K_{n-t}$ and
  the recursive structure of $\cN^{C_4}$ allows us to compute the conditional probability of these
  events inductively.

  With a similar argument, for every $n\geq 2$, we have
  \begin{align*}
    \phi_{C_4}(\overline{K}_n)
    & =
    \sum_{m\in\NN} 4^m\cdot 4^{-nm}\cdot 2\cdot 4^{-n}\cdot
    \sum_{t=1}^{n-1}\binom{n}{t}\cdot\phi_{C_4}(\overline{K}_t)\cdot\phi_{C_4}(\overline{K}_{n-t})
    \\
    & =
    \frac{1}{2\cdot(4^{n-1} - 1)}\cdot
    \sum_{t=1}^{n-1}\binom{n}{t}\cdot\phi_{C_4}(\overline{K}_t)\cdot\phi_{C_4}(\overline{K}_{n-t}).
  \end{align*}

  From this, a simple induction shows $\phi_{C_4}(\overline{K}_n)\leq\phi_{C_4}(K_n)$. Let $c$ be
  the limit $\lim_{n\to\infty} \phi_{C_4}(K_n)^{1/n}$ (which is guaranteed to exist
  by~\cite[Theorem~6]{CKP21}) and suppose toward a contradiction that $c > 0$. Let $n_0\in\NN$ be
  large enough so that
  \begin{align*}
    \frac{3}{4}\cdot c & \leq \phi_{C_4}(K_n)^{1/n} \leq \frac{5}{4}\cdot c
  \end{align*}
  for every $n\geq n_0$. Since $c > 0$, we can let
  \begin{align*}
    a
    & \df
    \min\left\{\phi_{C_4}(K_n)\cdot\left(\frac{3}{4}\cdot c\right)^{-n}
    \;\middle\vert\; n\leq n_0\right\}\cup\{1\},
    \\
    b
    & \df
    \max\left\{\phi_{C_4}(K_n)\cdot\left(\frac{5}{4}\cdot c\right)^{-n}
    \;\middle\vert\; n\leq n_0\right\}\cup\{1\},
  \end{align*}
  and note that since $\phi_{C_4}(K_n) > 0$ for every $n\in\NN$,
  it follows that $a > 0$. The definitions of $a$ and $b$ ensure that
  \begin{align*}
    a\cdot \left(\frac{3}{4}\cdot c\right)^n
    & \leq
    \phi_{C_4}(K_n)
    \leq
    b\cdot\left(\frac{5}{4}\cdot c\right)^n
  \end{align*}
  for \emph{every} $n\in\NN$ (as $a\leq 1\leq b$).

  Plugging these inequalities in~\eqref{eq:phiC4clique}, we get that for $n\geq 2$ we have
  \begin{align*}
    a\cdot \left(\frac{3}{4}\cdot c\right)^n
    & \leq
    \frac{1}{4^{n-1} - 1}
    \cdot
    b^2
    \cdot
    \left(\frac{5}{4}\cdot c\right)^n
    \sum_{t=1}^{n-1}\binom{n}{t}
    \\
    & \leq
    \frac{1}{4^{n-1} - 1}
    \cdot
    b^2
    \cdot
    \left(\frac{5}{4}\cdot c\right)^n
    \cdot 2^n,
  \end{align*}
  from which we conclude
  \begin{align*}
    a & \leq b^2\cdot \left(\frac{5}{3}\right)^n\cdot\frac{2^n}{4^{n-1} - 1},
  \end{align*}
  which by letting $n\to\infty$ yields $a = 0$, a contradiction. Therefore
  \begin{align*}
    \lim_{n\to\infty} \phi_{C_4}(K_n)^{1/n} = \lim_{n\to\infty} \phi_{C_4}(\overline{K}_n)^{1/n} = 0,
  \end{align*}
  as desired.
\end{proof}

\begin{remark}
  Similarly to Remark~\ref{rmk:otherCc}, the proof of Lemma~\ref{lem:phiC4notrivialsubobject} can be
  generalized to show that if $G\in\cC_C$ has at least one edge and one non-edge, then the limit
  $\phi_G$ of the sequence $(G^n)_{n\in\NN}$ of recursive blow-ups of $G$ does not contain any
  trivial sub-object.
\end{remark}

We can finally put all pieces together to characterize $\AEHP$ for universal theories of graphs.

\begin{theorem}\label{thm:AEHPgraph}
  The following are equivalent for a universal theory $T$ of graphs (i.e., $T\vdash\TGraph$).
  \begin{enumerate}
  \item We have $T\in\AEHP$.%
    \label{thm:AEHPgraph:AEHP}
  \item There exists a an induced subgraph $G\in\cC_C$ of the recursive blow-up $C_4^\omega$ of the
    $4$-cycle of countable height such that $G$ is \emph{not} a model of $T$.%
    \label{thm:AEHPgraph:Cc}
  \item $C_4^\omega$ is not a model of $T$.%
    \label{thm:AEHPgraph:C4omega}
  \end{enumerate}

  In particular, we have $\cC_S=\cC_P(\TGraph)=\cC_M=\cC_C$.
\end{theorem}

\begin{proof}
  The equivalence~\ref{thm:AEHPgraph:Cc}$\iff$\ref{thm:AEHPgraph:C4omega} follows from Remark~\ref{rmk:Cc} and
  the fact that $T$ is universal.

  \medskip

  For the implication~\ref{thm:AEHPgraph:Cc}$\implies$\ref{thm:AEHPgraph:AEHP}, first note that if
  $F\in\cC_C$, then Lemma~\ref{lem:phiC4zerodensities} implies that $F\in\cC_S$ so by Remark~\ref{rmk:subst}
  and Theorem~\ref{thm:AEHPsubstitution} (and the fact that trivially $\Forb_{\TGraph}(\{F\})\in\AEHP$
  whenever $\lvert F\rvert\leq 2$), we have $\Forb_{\TGraph}(\{F\})\in\AEHP$. At this point we
  have $\cC_C\subseteq\cC_S\subseteq\cC_P(\TGraph)$.

  On the other hand, if $T\vdash\TGraph$ is such that there exists $F\in\cC_C$ that is not a model of $T$,
  then $T\vdash\Forb_{\TGraph}(\{F\})$, so by Proposition~\ref{prop:AEHPaddaxioms}, we have $T\in\AEHP$.

  \medskip

  We prove the implication~\ref{thm:AEHPgraph:AEHP}$\implies$\ref{thm:AEHPgraph:Cc} by the
  contra-positive: if every $G\in\cC_C$ is a model of $T$, then $(C_4^n)_{n\in\NN}$ is a convergent
  sequence of models of $T$ whose limit $\phi_{C_4}$ does not have any trivial sub-object by
  Lemma~\ref{lem:phiC4notrivialsubobject}, thus $T\notin\AEHP$.

  This implication shows that any family $\cF$ that is minimal for the property
  $\Forb_{\TGraph}(\cF)\in\AEHP$ must intersect $\cC_C$. Since $\cC_C\subseteq\cC_P(\TGraph)$, the
  minimality of $\cF$ gives $\cF=\{F\}$ for some $F\in\cC_C$, thus $\cC_M\subseteq\cC_C$, which
  along with the trivial inclusion $\cC_P(\TGraph)\subseteq\cC_M$ and the already shown inclusion
  $\cC_C\subseteq\cC_S\subseteq\cC_P(\TGraph)$ gives the equality
  $\cC_S=\cC_P(\TGraph)=\cC_M=\cC_C$.
\end{proof}

We conclude this section by showing that any universal theory $T$ of graphs with $\AEHP$ also
satisfies the usual \emph{\Erdos--Hajnal property} ($\EHP$)\footnote{The reader familiar with $\EHP$
  may be more accustomed to the definition of $\EHP$ as a property of a graph $H$ corresponding to
  $\Forb_{\TGraph}(\{H\})$ having $\EHP$ as defined here.}, that is, there exists $c_T > 0$ such
that every graph of $T$ of size $n$ has a clique or anti-clique of size $n^{c_T}$. Note that a
priori it is not clear that the existence of linear-sized almost cliques or almost anti-cliques in
convergent sequences of $T$ should imply the existence of ``polynomial-sized'' cliques or
anti-cliques in all graphs of $T$. The proof of $\AEHP\implies\EHP$ for graphs instead relies on the
characterization of $\AEHP$ of Theorem~\ref{thm:AEHPgraph} and analogue of
Theorem~\ref{thm:AEHPsubstitution} for $\EHP$ from~\cite[Theorem~1.1]{APS01} (see
also~\cite[Theorem~2.3]{Chu14}) that inspired Theorem~\ref{thm:AEHPsubstitution}.

\begin{theorem}\label{thm:AEHP->EHP}
  If $T$ is a universal theory of graphs with $\AEHP$, then $T$ has $\EHP$.
\end{theorem}

\begin{proof}
  We prove this by the contra-positive. Assume $T$ does not have $\EHP$ and write $T$ as
  $\Forb_{\TGraph}(\cF)$ for some $\cF$ (see Remark~\ref{rmk:subst}).

  For each $F\in\cF$, since $T$ does not have $\EHP$, we know that
  $\Forb_{\TGraph}(\{F\})\notin\EHP$ (since all models of $T$ are obviously models of
  $\Forb_{\TGraph}(\{F\})$). Let $\cP_F$ be the set of prime graphs that are induced subgraphs of
  $F$. Since $F$ can be obtained from the graphs in $\cP_F$ via substitution, by the contra-positive
  of~\cite[Theorem~1.1]{APS01}, there exists $P_F\in\cP_F$ such that the theory
  $\Forb_{\TGraph}(\{P_F\})$ does not have $\EHP$.

  Since $\Forb_{\TGraph}(\{K_0\})$, $\Forb_{\TGraph}(\{K_1\})$, $\Forb_{\TGraph}(\{K_2\})$ and
  $\Forb_{\TGraph}(\{\overline{K}_2\})$ all have $\EHP$ (as the first theory has no models, the
  second is the theory whose unique model is $K_0$, the third is the theory of empty graphs and the
  fourth is the theory of complete graphs), we have $P_F\notin\{K_0,K_1,K_2,\overline{K}_2\}$.

  Thus every graph $F$ in $\cF$ has some prime subgraph $P_F$ that is not $K_0$, $K_1$, $K_2$ or
  $\overline{K}_2$, hence $\cC_S\subseteq\Forb_{\TGraph}(\cF)=\cM[T]$ as $\cC_S$ is the closure of
  $\{K_0,K_1,K_2,\overline{K}_2\}$ under substitutions.

  By Theorem~\ref{thm:AEHPgraph}, it follows that every graph in $\cC_C = \cC_S$ is a model of $T$,
  hence $T\notin\AEHP$.
\end{proof}

\section{Conclusion and open problems}
\label{sec:concl}

In this paper we studied the asymptotic consequences of stability in the finite when coupled with the notion
of convergence of densities, focusing particularly on producing linear-sized almost uniform sets in limits of
convergent sequences of models. Once such uniform sets are produced in the limit, they can be pulled back to
linear-sized almost uniform sets in convergent sequences of models or to positive upper-density almost uniform
sets in countable models. We then studied which universal theories have the approximate \Erdos--Hajnal
property ($\AEHP$), i.e., theories that must necessarily have linear-sized almost uniform sets in all of its
limit objects (equivalently, in all of its convergent sequences) and we characterized the particular case of
universal theories of graphs with $\AEHP$ as those that forbid some induced subgraph of some recursive blow-up
of the $4$-cycle.

\medskip

A consequence of Theorems~\ref{thm:stablecountablegraph} and~\ref{thm:stablecountablemodel} is that
any stable countable model must necessarily have an almost uniform set with positive upper
density. As we mentioned in Discussion~\ref{dsc:countable}, one cannot hope to upgrade these
theorems to produce almost uniform sets with positive density instead. A natural question is then
what extra hypothesis would allow such upgrade? More concretely, a natural extra condition would be
that of convergence of the marginals of the countable model, that is, is it true that a stable
countable model $N$ such that $(N\rest_{[n]})_{n\in\NN}$ is convergent must necessarily contain a
positive density almost uniform set? On the one hand, this rules out the example of
Discussion~\ref{dsc:countable} as it does not have convergent marginals, but on the other hand, in
the proofs of Theorems~\ref{thm:stablecountablegraph} and~\ref{thm:stablecountablemodel}, it is not
clear how to put together the sets $U_\ell$ returned by Theorems~\ref{thm:stableconvseqgraph}
and~\ref{thm:stableconvsequniversal} into a single almost uniform set $U$ of positive density even
in the presence of convergence of the marginals.

One of the interpretations of the usual \Erdos--Hajnal Conjecture is that graphs that are not random
have larger cliques or anti-cliques than the usual bound provided by Ramsey's Theorem, in other
words, the usual \Erdos--Hajnal property can be seen as ``failure of randomness''. In the case of
the approximate \Erdos--Hajnal property, this ``failure of randomness'' interpretation is even more
prominent: every $T$-on $\cN$ over a space $\Omega=(X,\cA,\mu)$ gives rise to a natural random
exchangeable countable model $\rn{K}$ of $T$ by sampling $\rn{x}$ in $\cE_{\NN_+}(\Omega)$ according
to $\mu$ and letting
\begin{align}\label{eq:rnK}
  (\rn{K}\vDash P(\alpha)) & \iff \alpha^*(\rn{x})\in\cN_P.
\end{align}
Given a positive measure $U\subseteq X$, we can also define a natural random exchangeable countable
model $\rn{K_U}$ via~\eqref{eq:rnK} but taking $\rn{x}$ in $\cE_{\NN_+}(\Omega)$ according to the
product measure that uses $\mu_U$ for variables indexed by vertices and $\mu$ for all other
variables. The natural quasirandomness property $\UInduce[1]$ in~\cite{CR20b} (generalizing the
graph quasirandomness property~\cite[$P_4$]{CGW89}) requires that $\rn{K_U}$ is equidistributed with
$\rn{K}$ for every positive measure $U\subseteq X$; informally, $\cN$ is ``random'' in the sense of
$\UInduce[1]$ if restricting $\phi$ to any positive measure set yields the same limit $\phi$. In the
context of $\AEHP$, a limit $\phi$ that does not contain any trivial sub-object fails randomness in
an even stronger sense: there exists a positive measure $U$ such that $\rn{K}_U$ is a deterministic
countable model (i.e., it is equal to some fixed $K$ with probability $1$).

Elaborating further on this notion of weak randomness, we could call a limit object
$\phi\in\HomT{T}$ \emph{weakly random} if it satisfies the following weakening of $\UInduce[1]$:
every sub-object $\psi$ of $\phi$ satisfies $\Th(\phi)=\Th(\psi)$, that is, restricting to positive
measure sets does not change which finite models have positive density. A consequence of the
equality $\cC_C=\cC_P(\TGraph)$ proved in Section~\ref{sec:AEHPgraphs} is that the limit recursive
blow-up of $C_4$ (and more generally, the limit recursive blow-up of any graph in $\cC_C$ that has
at least one edge and one non-edge) is weakly random. Just as in the theory of quasirandomness, it
is natural to ask for equivalent characterizations of this weak randomness notion and higher arity
generalizations of it.

In Theorem~\ref{thm:AEHP->EHP}, we used the characterization of $\AEHP$ from
Theorem~\ref{thm:AEHPgraph} to show that $\AEHP$ implies $\EHP$ for graphs, but the proof of
Theorem~\ref{thm:AEHP->EHP} is non-constructive, so it is very natural to ask if a constructive
proof is possible. More specifically, how does one find a clique or anti-clique of size $n^c$
knowing only that a linear-sized almost clique or almost anti-clique is guaranteed to exist in any
convergent sequence? A very basic instance of this question is as follows:
Theorem~\ref{thm:stablegraphon} provides an almost clique or anti-clique in a stable graphon by
constructing a $0$-good set in the limit (see Discussion~\ref{dsc:goodsets}), which along with
Theorem~\ref{thm:AEHP->EHP} implies that stable classes of graphs have $\EHP$; on the other hand,
this result is already known by stable Ramsey~\cite{MS14} but does not involve the known
construction of $\epsilon$-good sets in the finite. It is natural to begin by trying to prove that
stable classes of graphs have $\EHP$ from the existence of $\epsilon$-good sets alone.

In Theorem~\ref{thm:AEHPgraph}, we characterized universal theories of graphs with the approximate
\Erdos--Hajnal property ($\AEHP$) as precisely those that forbid some induced subgraph of a
recursive blow-up of the $4$-cycle. It is natural to ask what happens for more complicated universal
theories. For example, for universal theories of $k$-hypergraphs ($k\geq 3$), the substitution
operation of Theorem~\ref{thm:AEHPsubstitution} necessarily yields ``agnostic edges'' (see
Remark~\ref{rmk:subst}) and recursive blow-ups of a hypergraph $H$ also have a similar degree of
freedom: when we divide the space into parts $(V_i)_{i\in V(H)}$, how should we handle tuples
containing at least two vertices in one part $V_i$ but not all vertices in $V_i$?

The behavior of $\AEHP$ completely changes if allow predicates to be asymmetric, namely, in a
language $\cL\df\{E\}$ with a single binary predicate symbol $E$, if $\overline{K_2}$, $K_2$ and $A$
denote the anti-edge, the anti-parallel edges and the single edge, respectively (i.e.,
$V(\overline{K}_2)\df V(K_2)\df V(A)\df[2]$, $E^{\overline{K}_2}\df\varnothing$,
$E^{K_2}\df\{(1,2),(2,1)\}$ and $E^A\df\{(1,2)\}$), then the canonical theories
\begin{align*}
  \Forb_{T_\cL}(\{\overline{K}_2\}),\Forb_{T_\cL}(\{K_2\}),\Forb_{T_\cL}(\{A\})
\end{align*}
clearly do not have $\AEHP$: the first two because the sequence of transitive tournaments avoids
linear-sized almost uniform sets and the last because $\Forb_{T_\cL}(\{A\})\cong\TGraph$. On the
other hand, $\Forb_{T_\cL}(\{\overline{K}_2,K_2,A\})$ does not have any models of size $2$, so it
trivially has $\AEHP$ (as it has no limit object). This means that in general we cannot hope that
for a universal theory $T$ there exists a family $\cC$ such that any universal theory $T'\vdash T$
has $\AEHP$ if and only if it forbids some element of $\cC$, that is, in general we do not expect
that $\cC_P(T) = \cC_M(T)$ (see Remark~\ref{rmk:cCPcCM}). A natural problem is then to characterize
which theories have this ``principality'' property, and more generally, to study how different can
$\cC_P(T)$ be from $\cC_M(T)$.

One might think that the example of the previous paragraph stems from the requirement of almost
trivial being too strong for asymmetric predicates; after all, sets returned by Ramsey's Theorem do
not necessarily yield almost trivial sequences when asymmetric predicates are involved. Instead, one
could define a property $\AEHP'$ as $T\in\AEHP'$ if every $\phi\in\HomT{T}$ has a \emph{finitely
  categorical} sub-object $\psi$, that is, $\Th(\psi)$ is finitely categorical in the
model-theoretic sense (equivalently, for each $n\in\NN$ there is exactly one model $M_n$ of size $n$
up to isomorphism such that $\psi(M_n) > 0$). Finitely categorical limits are precisely the limits
of convergent sequences of sets that can be returned by Ramsey's Theorem. To show failure of
$\AEHP'$ the sequence of transitive tournaments is not good as it converges to a finitely
categorical limit. However, $\Forb_{T_\cL}(\{\overline{K}_2\})$ and $\Forb_{T_\cL}(\{K_2\})$ still
do not have $\AEHP'$ by using the quasirandom sequence of tournaments instead; thus ``principality''
still fails for $\AEHP'$ over $\cL\df\{E\}$. It is not clear which of $\AEHP$ or $\AEHP'$ is more
appropriate in the presence of asymmetric predicates.

\section*{Acknowledgments}

We are grateful to Avi Wigderson and Alexander Razborov for some useful comments on an earlier
version of this manuscript.

\bibliographystyle{alpha}
\bibliography{refs}

\appendix

\section{Ultraproduct method}
\label{sec:ultraproduct}

In this section, we present the formal definition of separable realizations from~\cite{ES12} (see
also~\cite{AC14}). Throughout this section, we assume that we have a fixed sequence $(V_n)_{n\in\NN}$ of
finite sets of increasing sizes (intended to be the vertex sets of a convergent sequence of models) and we
have fixed a non-principal ultrafilter $\cD$ over $\NN$.

\begin{definition}[Loeb measure]
  Given a finite set $U$, let $\tau(U)$ be the Boolean algebra of internal subsets of $\prod_{n\in\NN}
  V_n^U/\cD$. Let also $\mu^U\function{\tau(U)}{[0,1]}$ be the finitely additive measure defined by the
  ultralimit
  \begin{align*}
    \mu^U\left(\prod_{n\in\NN} A_n/\cD\right) & \df \lim_{n\to\cD} \mu_n^U(A_n),
  \end{align*}
  where $\mu_n^U(A)\df \lvert A\rvert/\lvert V_n^U\rvert$ is the normalized counting measure on $V_n^U$.

  A \emph{$\mu^U$-nullset} is a set $N\subseteq\prod_{n\in\NN} V_n^U$ such that for every $\epsilon > 0$,
  there exists $B\in\tau(U)$ such that $N\subseteq B$ and $\mu^U(B)\leq\epsilon$. Let $\sigma(U)$ be the
  collection of sets $A\subseteq\prod_{n\in\NN} V_n^U/\cD$ that differ from some set in $\tau(U)$ only by a
  $\mu^U$-nullset.
  
  A standard saturation argument (see~\cite[Lemma~2.4]{ES12}) shows that if $A_m\in\tau(U)$
  ($m\in\NN$) are internal sets, then there exists an internal set $B\supseteq\bigcup_{m\in\NN} A_m$
  with $\mu^U(B) = \lim_{m\to\infty}\mu^U(\bigcup_{m'\leq m} A_{m'})$. This in particular implies
  that $\sigma(U)$ is a $\sigma$-algebra and that $\mu^U$ is a finite pre-measure on $\tau(U)$ and
  thus \Caratheodory's Theorem shows that $\mu^U$ can be uniquely extended to a (complete) measure
  on $\sigma(U)$, called \emph{Loeb measure} and which we denote also by $\mu^U$ by abuse.
\end{definition}

As mentioned in Section~\ref{sec:consgraphs}, the probability space $(\prod_{n\in\NN}
V_n^U/\cD,\sigma(U),\mu^U)$ is far from being standard, namely, it is non-separable. Furthermore,
even the structure between these spaces is somewhat counter-intuitive: for $U_1,U_2$ disjoint and
non-empty, $\sigma(U_1\cup U_2)$ is much larger than the completion of the product $\sigma$-algebra
$\sigma(U_1)\otimes\sigma(U_2)$. Nevertheless, the following analogue of Fubini's Theorem still
holds.

\begin{theorem}[Fubini's Theorem for Loeb measures]\label{thm:FubiniLoeb}
  If $A\in\sigma(U)$ and $U'\subseteq U$, then for $\mu_{U'}$-almost every $x\in\prod_{n\in\NN}
  V_n^{U'}/\cD$, the set
  \begin{align*}
    A(x) & \df \left\{y\in\prod_{n\in\NN} V_n^{U\setminus U'}/\cD \;\middle\vert\; (x,y)\in A\right\}
  \end{align*}
  is in $\sigma(U\setminus U')$, the function $x\mapsto\mu^{U\setminus U'}(A(x))$ (defined
  arbitrarily when $A(x)$ is not in $\sigma(U\setminus U')$) is measurable with respect to
  $\sigma(U')$ and
  \begin{align*}
    \mu^U(A) & = \int_{\prod_{n\in\NN} V_n^{U'}/\cD} \mu^{U\setminus U'}(A(x))\ d\mu^{U'}(x).
  \end{align*}
\end{theorem}

Recall that an injection $\alpha\injection{U_1}{U_2}$ defines contra-variantly the ``projections''
$\alpha^*\function{\prod_{n\in\NN} V_n^{U_2}/\cD}{\prod_{n\in\NN} V_n^{U_1}/\cD}$ and
$\alpha^*\function{[0,1]^{r(U_2)}}{[0,1]^{r(U_1)}}$ via $\alpha^*(x)_u\df x_{\alpha(u)}$.

\begin{definition}[Separable realizations]
  Given finite sets $U'\subseteq U$, define the $\sigma$-algebra
  \begin{align*}
    \sigma(U',U) & \df (\iota^*)^{-1}(\sigma(U')) \df \{(\iota^*)^{-1}(A) \mid A\in\sigma(U')\},
  \end{align*}
  where $\iota\injection{U'}{U}$ is the inclusion map. We also let $\sigma(U',U)^*$ be the
  $\sigma$-algebra generated by $\{\sigma(U'',U) \mid U''\subsetneq U'\}$.

  Given $k\in\NN_+$, a \emph{separable realization of order $k$} is a measure-preserving function
  $\Theta\function{\prod_{n\in\NN} V_n^k/\cD}{[0,1]^{r(k)}}$ such that
  \begin{enumerate}
  \item For every $U\in r(k)$ and every Lebesgue measurable $A\subseteq[0,1]$, the set
    $(\pi_U\comp\Theta)^{-1}(A)$ is in $\sigma(U,[k])$ and is independent from $\sigma(U,[k])^*$,
    where $\pi_U\function{[0,1]^{r(k)}}{[0,1]}$ is the projection onto the $U$ coordinate.
  \item For every permutation $\sigma\in S_k$, we have $\sigma^*\comp\Theta = \Theta\comp\sigma^*$.
  \end{enumerate}

  Given $m\in[k]$, a \emph{restriction of $\Theta$ of order $m$} is a measure-preserving function
  $\Theta_m\function{\prod_{n\in\NN} V_n^m/\cD}{[0,1]^{r(m)}}$ such that for every injection
  $\alpha\injection{[m]}{[k]}$, we have $\alpha^*\comp\Theta = \Theta_m\comp\alpha^*$.

  Given $m\geq k$, a \emph{lifting of $\Theta$ of order $m$} is a measure-preserving function
  $\Theta_m\function{\prod_{n\in\NN} V_n^m/\cD}{[0,1]^{r(m,k)}}$ such that for every injection
  $\alpha\injection{[k]}{[m]}$, we have $\alpha^*\comp\Theta_m = \Theta\comp\alpha^*$.
\end{definition}

It is straightforward to check that the properties of separable realizations, restrictions and
liftings imply that $\Theta_m\comp\sigma^* = \sigma^*\comp\Theta_m$ for every $\sigma\in S_m$ both
for liftings and restrictions. This in particular implies that restrictions of order $m$ are also
separable realizations of order $m$. Furthermore, the definitions of restrictions and liftings
themselves already show their uniqueness: restrictions must be defined by
\begin{align}\label{eq:restriction}
  \Theta_m(x)_A & \df \Theta(y)_A,
\end{align}
where $y\in\prod_{n\in\NN} V_n^k/\cD$ is any point in $(\iota^*)^{-1}(x)$ and
$\iota\injection{[m]}{[k]}$ is the inclusion map and liftings must be defined by
\begin{align}\label{eq:lifting}
  \Theta_m(x)_A & \df \Theta(\alpha^*(x))_{[\ell]},
\end{align}
where $\alpha\injection{[k]}{[m]}$ is any injection with $\alpha([\ell]) = A$. It is then
straightforward to check that~\eqref{eq:restriction} and~\eqref{eq:lifting} give a restriction and a
lifting, respectively (see~\cite[Lemma~3.2]{ES12}).

\end{document}

%% file: halfgraphon.tex
\begingroup

\def\side{3}
\def\axis{3.25}

\newcommand{\halfgraphon}[1]{%
  \begin{tikzpicture}
    \pgfmathsetmacro{\twon}{2 * #1}
    \foreach \i in {0,...,\twon}{
      \pgfmathsetmacro{\x}{\i * \side / \twon}
      \foreach \j in {0,...,\twon}{
        \pgfmathsetmacro{\y}{\j * \side / \twon}
        \coordinate (P\i-\j) at (\x,\y);
      }
    }

    \foreach \i in {1,...,#1}{
      \pgfmathtruncatemacro{\twoi}{2*\i}
      \pgfmathtruncatemacro{\twoiminusone}{\twoi - 1}
      \pgfmathtruncatemacro{\twoiminustwo}{\twoi - 2}
      \foreach \j in {1,...,#1}{
        \pgfmathtruncatemacro{\twoj}{2*\j}
        \pgfmathtruncatemacro{\twojminusone}{\twoj - 1}
        \pgfmathtruncatemacro{\twojminustwo}{\twoj - 2}

        \fill (P\twoiminusone-\twojminustwo) -- (P\twoi-\twojminustwo) -- (P\twoi-\twojminusone)
        -- cycle;
        \fill (P\twoiminustwo-\twojminusone) -- (P\twoiminusone-\twoj) -- (P\twoiminustwo-\twoj)
        -- cycle;
      }
    }

    \coordinate (X) at (\axis,0);
    \coordinate (Y) at (0,\axis);

    \draw[->] (0,0) -- (X);
    \draw[->] (0,0) -- (Y);

    \node[below right] at (X) {$x$};
    \node[above left] at (Y) {$y$};
  \end{tikzpicture}
}
\begin{center}
  \begin{subfigure}{0.3\textwidth}
    \halfgraphon{1}
    \subcaption*{$k=1$}
  \end{subfigure}
  \quad
  \begin{subfigure}{0.3\textwidth}
    \halfgraphon{2}
    \subcaption*{$k=2$}
  \end{subfigure}
  \quad
  \begin{subfigure}{0.3\textwidth}
    \halfgraphon{3}
    \subcaption*{$k=3$}
  \end{subfigure}

  \captionsetup{singlelinecheck=off}
  \caption{Different representations $W_k(x,y)\df W(k x\bmod 1, k y\bmod 1)$ of the graphon
    $W$ given by
    \begin{align*}
      W(x,y)
      & \df
      \begin{dcases*}
        1, & if $\floor{x/2}\neq\floor{y/2}$ and $(\floor{x/2} < \floor{y/2}\tot x\bmod(1/2) < y\bmod(1/2))$,\\
        0, & otherwise.
      \end{dcases*}
    \end{align*}
    The graphon $W$ is called the \emph{half-graphon}.
  }
  \label{fig:halfgraphon}
\end{center}
\endgroup


%% file: halfgraph.tex
\begingroup
\def\n{7}
\def\horsep{1}
\def\vertsep{2}
\def\ptsize{2pt}

\begin{center}
  \begin{tikzpicture}
    \foreach \i in {1,...,\n}{
      \pgfmathsetmacro{\x}{\i * \horsep}
      \coordinate (X\i) at (\x,0);
      \coordinate (Y\i) at (\x,-\vertsep);

      \filldraw (X\i) circle (\ptsize);
      \filldraw (Y\i) circle (\ptsize);

      \node[above] at (X\i) {$x_{\i}$};
      \node[below] at (Y\i) {$y_{\i}$};
    }

    \foreach \i in {1,...,\n}{
      \foreach \j in {\i,...,\n}{
        \draw (X\i) -- (Y\j);
      }
    }
  \end{tikzpicture}

  \caption{Half-graph of order $\n$. Between two distinct $x_i$'s or between two distinct $y_i$'s there is no
    edge/non-edge requirement.}
  \label{fig:halfgraph}
\end{center}
\endgroup


%% file: tree.tex
\begingroup
\def\n{3}
\def\horsep{1}
\def\vertsep{1}
\def\bigvertsep{3}
\def\ptsize{2pt}

\begin{center}
  \begin{tikzpicture}
    \def\strings{v///,}
    \coordinate (Y) at (0,0);
    \foreach \i in {1,...,\n}{
      \let\oldstrings\strings
      \def\strings{}
      \foreach \v/\s/\sno/\syes in \oldstrings {
        \ifx\v\empty\relax
        \else
        \edef\newstrings{\strings v/\s0/{\sno Y\s,}/{\syes},v/\s1/{\sno}/{\syes Y\s,},}
        \global\let\strings\newstrings
        \pgfmathsetmacro{\step}{2^(\n - \i) * \horsep}
        \coordinate (Y\s0) at ($(Y\s) + (-\step,-\vertsep)$);
        \coordinate (Y\s1) at ($(Y\s) + (+\step,-\vertsep)$);

        \filldraw (Y\s) circle (\ptsize);

        \ifx\s\empty\relax
        \node[below] at (Y\s) {$y_\varnothing$};
        \else
        \node[below] at (Y\s) {$y_{\s}$};
        \fi
        \fi
      }
    }

    \pgfmathsetmacro{\y}{\n*\vertsep + \bigvertsep}
    \foreach \v/\s/\sno/\syes in \strings {
      \ifx\v\empty\relax
      \else
      \coordinate (X\s) at ($(Y\s) + (0,\y)$);

      \filldraw (X\s) circle (\ptsize);
      \node[above] at (X\s) {$x_{\s}$};

      \foreach \i in \sno {
        \ifx\i\empty\relax
        \else
        \draw[dashed] (X\s) -- (\i);
        \fi
      }

      \foreach \i in \syes {
        \ifx\i\empty\relax
        \else
        \draw (X\s) -- (\i);
        \fi
      }
      \fi
    }
  \end{tikzpicture}

  \caption{Tree of height $\n$. Solid lines correspond to edge requirements and dashed lines correspond to
    non-edge requirements.}
  \label{fig:tree}
\end{center}
\endgroup


%% file: agreementsgraphon.tex
\begingroup
\def\order{3}
\def\side{5}
\def\axis{5.5}

\begin{center}
  \begin{tikzpicture}
    \draw (0,0) -- (\side,0) -- (\side,\side) -- (0,\side) -- cycle;
    \draw[->] (0,0) -- (\axis,0);
    \draw[->] (0,0) -- (0,\axis);

    \node[below right] at (\axis,0) {$x$};
    \node[above left] at (0,\axis) {$y$};
    
    \pgfmathsetmacro{\smallside}{2^(-2*\order) * \side}
    \def\strings{v//0/0,}
    \foreach \i in {1,...,\order}{
      \let\oldstrings\strings
      \def\strings{}
      \foreach \v/\s/\done/\dtwo in \oldstrings {
        \ifx\v\empty\relax
        \else
        \pgfmathsetmacro{\newdone}{\done + 2^(-2*\i+1) * \side}
        \pgfmathsetmacro{\newdtwo}{\dtwo + 2^(-2*\i) * \side}
        \edef\newstrings{\strings v/\s0/\done/\dtwo,v/\s1/\newdone/\newdtwo,}
        \global\let\strings\newstrings
        \fi
      }
    }

    \foreach \v/\s/\done/\dtwo in \strings {
      \ifx\v\empty\relax
      \else
      \foreach \V/\S/\Done/\Dtwo in \strings {
        \ifx\V\empty\relax
        \else
        \compareStringsDo{\s}{\S}{
          \foreach \w/\t/\eone/\etwo in \strings {
            \ifx\w\empty\relax
            \else
            \foreach \W/\T/\Eone/\Etwo in \strings {
              \ifx\W\empty\relax
              \else
              \compareStringsDo{\t}{\T}{
                \fill ($(\done,\Done) + (\etwo,\Etwo)$) -- ++ (\smallside,0) -- ++ (0,\smallside)
                -- ++ (-\smallside,0) -- cycle;
              }{
              }{
              }
              \fi
            }
            \fi
          }
        }{
        }{
          \foreach \w/\t/\eone/\etwo in \strings {
            \ifx\w\empty\relax
            \else
            \foreach \W/\T/\Eone/\Etwo in \strings {
              \ifx\W\empty\relax
              \else
              \compareStringsDo{\t}{\T}{
              }{
              }{
                \fill ($(\done,\Done) + (\etwo,\Etwo)$) -- ++ (\smallside,0) -- ++ (0,\smallside)
                -- ++ (-\smallside,0) -- cycle;
              }
              \fi
            }
            \fi
          }
        }
        \fi
      }
      \fi
    }
  \end{tikzpicture}

  \caption{Approximation of the graphon $W'$ of Example~\ref{ex:quasirandompermuton}. The graphon $W'$ has a
    fractal structure, whose first $\order$ steps are represented in the picture.}
  \label{fig:agreementsgraphon}
\end{center}
\endgroup


%% file: graphsubst.tex
\begingroup
\def\baseangle{0}
\def\baseradius{1cm}
\def\labelradius{1.3cm}

\def\bigbaseangle{0}
\def\bigradius{2.5cm}
\def\biglabelradius{2.8cm}
\def\smallbaseangle{0}
\def\smallradius{0.4cm}
\def\smalllabelradius{0.9cm}

\def\ptsize{2pt}
\def\smallptsize{1pt}
\def\captionheight{-2cm}
\def\secondcaptionheight{-5cm}

\def\drawgraph#1#2#3#4{
  \begin{tikzpicture}
    \pgfmathtruncatemacro{\nmo}{#1-1}
    \foreach \i in {0,...,\nmo}{
      \pgfmathsetmacro{\angle}{\baseangle + \i * 360 / #1}
      \coordinate (P\i) at (\angle:\baseradius);
      \coordinate (L\i) at (\angle:\labelradius);

      \filldraw (P\i) circle (\ptsize);
      \node at (L\i) {#3};
    }

    \foreach \i/\j in #2{
      \draw (P\i) -- (P\j);
    }

    \node at (0,\captionheight) {#4};
  \end{tikzpicture}
}

\def\drawsubstitutedgraph#1#2#3#4#5#6#7#8{
  \begin{tikzpicture}
    \pgfmathtruncatemacro{\nmo}{#1-1}
    \foreach \i in {0,...,\nmo}{
      \pgfmathsetmacro{\angle}{\bigbaseangle + \i * 360 / #1}
      \coordinate (GP\i) at (\angle:\bigradius);
      \coordinate (GL\i) at (\angle:\biglabelradius);
    }

    \foreach \i in {1,...,\nmo}{
      \filldraw (GP\i) circle (\ptsize);
      \node at (GL\i) {#3};
    }

    \pgfmathtruncatemacro{\mmo}{#5-1}
    \foreach \i in {0,...,\mmo}{
      \pgfmathsetmacro{\angle}{\smallbaseangle + \i * 360 / #5}
      \coordinate (FP\i) at ($(GP#4) + (\angle:\smallradius)$);
      \coordinate (FL\i) at ($(GP#4) + (\angle:\smalllabelradius)$);

      \filldraw (FP\i) circle (\smallptsize);
      \node at (FL\i) {#7};
    }

    \foreach \i/\j in #6{
      \draw (FP\i) -- (FP\j);
    }

    \foreach \i/\j in #2{
      \ifnum\i=#4
        \foreach \k in {0,...,\mmo}{
          \draw (FP\k) -- (GP\j);
        }
      \else
        \ifnum\j=#4
          \foreach \k in {0,...,\mmo}{
            \draw (GP\i) -- (FP\k);
          }
        \else
          \draw (GP\i) -- (GP\j);
        \fi
      \fi
    }

    \node at (0,\secondcaptionheight) {#8};
  \end{tikzpicture}
}

\def\Gn{6}
\def\Gedges{0/1,0/2,0/5,1/2,1/3,1/4,3/5}
\def\Glabel{$v_{\i}$}

\def\Fn{4}
\def\Fedges{0/2,1/2,1/3}
\def\Flabel{$w_{\i}$}

\def\substvert{0}

\begin{center}
  \begin{subfigure}{0.45\textwidth}
    \begin{center}
      \drawgraph{\Gn}{\Gedges}{\Glabel}{$G$}
    \end{center}
  \end{subfigure}
  \qquad
  \begin{subfigure}{0.45\textwidth}
    \begin{center}
      \drawgraph{\Fn}{\Fedges}{\Flabel}{$F$}
    \end{center}
  \end{subfigure}

  \begin{subfigure}{0.95\textwidth}
    \begin{center}
      \drawsubstitutedgraph{\Gn}{\Gedges}{\Glabel}{\substvert}{\Fn}{\Fedges}{{\scriptsize\Flabel}}{$G^{v_{\substvert}\to F}$}
    \end{center}
  \end{subfigure}

  \caption{Substitution operation for graphs.}
  \label{fig:graphsubst}
\end{center}
\endgroup

%% file: C4ell.tex
\begingroup
\def\levels{4}
\def\initialradius{4.5}
\def\fracprop{0.35}
\def\initiallinewidth{30}

\begin{center}
  \begin{tikzpicture}
    \coordinate (P) at (0cm,0cm);

    \foreach[%
      remember=\centers as \prevcenters (initially P),
      remember=\radius as \prevradius (initially \initialradius),
      remember=\linewidth as \prevlinewidth (initially \initiallinewidth),
      evaluate=\radius using \prevradius * \fracprop,
      evaluate=\linewidth using \prevlinewidth * \fracprop,
      evaluate=\drawradius using 1/(1-\fracprop)*\radius%
    ] \i in {0,...,\levels}{
      \xdef\centers{\relax}
      \xdef\lastradius{\drawradius}
      \foreach \c in \prevcenters {
        \coordinate (\c1) at ($(\c) + (45:\prevradius)$);
        \coordinate (\c2) at ($(\c) + (135:\prevradius)$);
        \coordinate (\c3) at ($(\c) + (225:\prevradius)$);
        \coordinate (\c4) at ($(\c) + (315:\prevradius)$);

        \draw[line width=\prevlinewidth]
        (\c1) -- (\c2) -- (\c3) -- (\c4) -- cycle;

        \draw[fill=white] (\c1) circle (\drawradius);
        \draw[fill=white] (\c2) circle (\drawradius);
        \draw[fill=white] (\c3) circle (\drawradius);
        \draw[fill=white] (\c4) circle (\drawradius);

        \if\centers\relax
        \xdef\centers{\c1,\c2,\c3,\c4}
        \else
        \xdef\centers{\centers,\c1,\c2,\c3,\c4}
        \fi
      }
    }

    \foreach \c in \centers {
      \filldraw (\c) circle (\lastradius);
    }
  \end{tikzpicture}
  \caption{Pictorial view of the recursive blow-up $C_4^{\levels}$ of the $4$-cycle of height
    $\levels$.}
  \label{fig:C4ell}
\end{center}

\endgroup


%% file: recC4.tex
\begingroup
\def\order{3}
\def\side{5}
\def\axis{5.5}

\begin{center}
  \begin{tikzpicture}
    \draw (0,0) -- (\side,0) -- (\side,\side) -- (0,\side) -- cycle;
    \draw[->] (0,0) -- (\axis,0);
    \draw[->] (0,0) -- (0,\axis);

    \node[below right] at (\axis,0) {$x$};
    \node[above left] at (0,\axis) {$y$};

    \foreach \i in {1,...,\order}{
      \pgfmathsetmacro{\step}{4^(-\i+1) * \side}
      \pgfmathsetmacro{\baseside}{4^(-\i) * \side}
      \pgfmathtruncatemacro{\rep}{4^(\i-1)}
      \foreach \j in {1,...,\rep}{
        \pgfmathsetmacro{\p}{(\j-1) * \step}
        \pgfmathsetmacro{\ppone}{\p + \baseside}
        \pgfmathsetmacro{\pptwo}{\ppone + \baseside}
        \pgfmathsetmacro{\ppthree}{\pptwo + \baseside}
        \foreach \x/\y in {\p/\ppone,\ppone/\pptwo,\pptwo/\ppthree,\ppthree/\p}{
          \fill (\x,\y) -- ++(\baseside,0) -- ++(0,\baseside) -- ++(-\baseside,0) -- cycle;
          \fill (\y,\x) -- ++(\baseside,0) -- ++(0,\baseside) -- ++(-\baseside,0) -- cycle;
        }
      }
    }
  \end{tikzpicture}

  \caption{Approximation of the graphon $\widehat{W}^{C_4}$ of Definition~\ref{def:recC4}. The graphon
    $\widehat{W}^{C_4}$ has a fractal structure, whose first $\order$ steps are represented in the picture.}
  \label{fig:recC4}
\end{center}
\endgroup


%% file: cehalfgraph.tex
\begingroup
\def\n{7}
\def\horsep{1.7}
\def\vertsep{3}
\def\ptsize{2pt}
\def\controlbase{0.5}

\begin{center}
  \begin{footnotesize}
    \begin{tikzpicture}
      \foreach \i in {1,...,\n}{
        \pgfmathsetmacro{\x}{\i * \horsep}
        \coordinate (X\i) at (\x,0);
        \coordinate (Y\i) at (\x,-\vertsep);

        \filldraw (X\i) circle (\ptsize);
        \filldraw (Y\i) circle (\ptsize);

      }

      \foreach \i in {1,...,\n}{
        \foreach \j in {\i,...,\n}{
          \draw (X\i) -- (Y\j);
        }
      }

      \pgfmathtruncatemacro{\nmo}{\n-1}
      \foreach \i in {1,...,\nmo}{
        \pgfmathtruncatemacro{\ipo}{\i+1}
        \foreach \j in {\ipo,...,\n}{
          \pgfmathsetmacro{\c}{\controlbase * (\j - \i)}
          \draw (X\i) .. controls ($(X\i) + (0,\c)$) and ($(X\j) + (0,\c)$) .. (X\j);
        }
      }
    \end{tikzpicture}
  \end{footnotesize}

  \caption{Clique-empty-half-graph $\widehat{H}_{\n}$ of order $\n$. The top part forms a clique, the
    bottom part induces an empty graph and the edges in between form a half-graph.}
  \label{fig:cehalfgraph}
\end{center}
\endgroup
